\documentclass[10pt]{amsart}
\usepackage{graphicx,amssymb,amsfonts,amsmath,amsthm,newlfont}
\usepackage{epsfig,url}
\usepackage{axodraw4j}
\usepackage{color}

\usepackage[all,2cell]{xy} \UseAllTwocells \SilentMatrices

\newtheorem{thm}{Theorem}[section]
\newtheorem{cor}[thm]{Corollary}
\newtheorem{lem}[thm]{Lemma}
\newtheorem{prop}[thm]{Proposition}
\theoremstyle{definition}
\newtheorem{defn}[thm]{Definition}
\newtheorem{conj}{Conjecture} 

\newtheorem{ex}[thm]{Examples}
\newtheorem{example}[thm]{Example}

\theoremstyle{remark}
\newtheorem{rem}[thm]{Remark}
\numberwithin{equation}{section}
\newcommand{\Z}{\mathbb Z}

\newcommand{\p}{\mathfrak{p}}
\newcommand{\R}{\mathbb R}

\newcommand{\Pro}{\mathbb P}

\newcommand{\xx}{\mathsf{x}}

\newcommand{\gr}{\mathrm{gr}}
\newcommand{\sss}{\mathsf{s}}

\font \rus= wncyr10
\newcommand{\sha}{\, \hbox{\rus x} \,}

\newcommand{\Ho}{\mathcal{H}}

\newcommand{\MT}{\mathcal{MT}}

\newcommand{\PP}{\mathfrak{p} }

\newcommand{\ee}{\mathfrak{e} }
\newcommand{\ff}{\mathfrak{f} }
\newcommand{\hh}{\mathfrak{h} }

\newcommand{\gd}{\mathfrak{dg}^{\mathfrak{m}}}
\newcommand{\gm}{\mathfrak{g}^{\mathfrak{m}}}

\newcommand{\RR}{\mathfrak{R} }
\newcommand{\rr}{\overline{\rho} }

\newcommand{\Ss}{\mathsf{P} }
\newcommand{\sfe}{\mathsf{e} }
\newcommand{\sfo}{\mathsf{o} }
\newcommand{\Sf}{\mathsf{P}^{\sfe} }
\newcommand{\So}{\mathsf{P}^{\sfo} }

\newcommand{\sh}{\sigma^\mathfrak{h} }
\newcommand{\tauh}{\tau^\mathfrak{h}}

\newcommand{\qq}{\mathfrak{q} }

\newcommand{\eact}{\circledast}

\newcommand{\Der}{\mathsf{d}}

\newcommand{\ls}{\mathfrak{ls}}

\newcommand{\zetam}{\zeta^{ \mathfrak{m}}}

\newcommand{\stu}{*}

\newcommand{\Deltat}{\Delta_{\stu}}

\newcommand{\y}{\mathsf{y}}

\newcommand{\studot}{\underline{\cdot}\,}

\newcommand{\Osha}{{}^{\sha} \! \Or}
\newcommand{\Ostu}{{}^{\star} \Or}

\newcommand{\Q}{\mathbb Q}

\newcommand{\Lo}{\mathcal{L}}
\newcommand{\U}{\mathcal{U}}

\newcommand{\F}{\mathbb F}

\newcommand{\To}{\longrightarrow}

\newcommand{\A}{\mathbb{A}}

\newcommand{\x}{\mathsf{x}}

\newcommand{\pp}{\mathfrak{p}}

\newcommand{\tone}{\overset{\rightarrow}{1}\!}

\newcommand{\Or}{\mathcal{O}}

\newcommand{\g}{\mathfrak{g}}

\newcommand{\circb}{\, \underline{\circ}\, }

\newcommand{\dd}{\mathfrak{D}}
\newcommand{\Res}{\mathrm{Res}}
\newcommand{\V}{\mathcal{V}}

\newcommand{\sll}{\mathfrak{sl} }

\newcommand{\eis}{\mathfrak{u}^{\varepsilon} }
\newcommand{\eiso}{\mathfrak{u}^{\varepsilon} _0 }

\newcommand{\dmr}{\mathfrak{dmr} }

\newcommand{\eee}{\mathsf{e} }

\newcommand{\e}{\mathbf{e}}

\newcommand{\Lie}{\mathrm{Lie}\,}
\newcommand{\ad}{\mathrm{ad}}

\begin{document}
\author{Francis Brown}
\begin{title}[Anatomy of an associator]{Anatomy of an associator}\end{title}
\maketitle
\begin{abstract} 
We study some  Lie algebras defined by   solutions to the double shuffle equations with poles and 
 construct   families of  explicit solutions to these equations in all  weights and depths. These provide universal coordinates in which to write down  `zeta elements':  the images of  generators  of the Lie algebra of the motivic Galois group of mixed Tate motives over the integers. We expect that a similar statement holds for associators.  In particular, these coordinates  encode algebraic relations between multiple zeta values, and enable one to compress the currently used tables for relations between multiple zeta values in, for example, weights $\leq 13$, already by a factor of a thousand. 
The Lie algebras and groups studied here  form part of  a large algebraic structue which  is related to the work of Ecalle on the calculus of moulds, and also  related to the theory of  universal mixed elliptic motives, and modular forms for the full modular group.
  \end{abstract}
\section{Introduction}

\subsection{Remarks added in September 2017}
These are private notes  justifying the theorems announced in a talk at the IHES on the 5th December 2012, during the conference `Amplitudes and periods' and were written at that time. After distributing these notes to an increasing number of colleagues over the intervening years, and never finding the time to  completely rewrite  them to my satisfaction,  I finally decided to make them   available even if the `end' of this story is not even remotely in sight.

The impetus for this project was the following observation: in \cite{Depth}, we associated to any  normalised even period polynomial  a  solution to the linearised double shuffle equations in depth four. 
Applying this  to  non-normalised period polynomials  gave explicit solutions to the same equations \emph{with poles}. This led to the following questions: do there exist explicit solutions to the  full double shuffle equations with poles?  Can one use them to construct  rational associators? Astonishingly, the answer to both questions seems to be `yes'.  In my talk,  I described explicit elements $\psi_{2n+1}$ of weight $2n+1$ for every $n\geq 1$ which correspond to the odd zeta values $\zeta(2n+1)$, and a further element $\psi_{-1}$ of weight $-1$, which solve the double shuffle equations.  The space of solutions to the double shuffle equations with poles forms a Lie algebra which we call $\p\dmr$ in these notes, and the elements $\psi$ generate  a Lie subalgebra $\Lo$ of $\p \dmr$ with some special properties.  The definitions of  these objects  and proofs of  these  statements are given in the present notes.
Furthemore, I  conjectured that any zeta elements $\sigma_{2n+1}$, $n\geq 1$, which are images of generators of the Lie algebra of the motivic Galois group of mixed Tate motives over the integers,  can be expressed as Lie brackets of the $\psi$'s. I called such an expression, which encodes the arithmetic of relations between multiple zeta values,   its `anatomy'.   Furthermore, I explained how to extend the theory, at least in small depths ($\leq 4$), to torsors over the motivic Galois group, which are conjecturally the space of rational associators $\tau$.  In a similar way, therefore,  I expect that any rational associator $\tau$ should also admit a kind of `Taylor expansion' in terms of the Lie algebra $\Lo$ which encodes its internal and arithmetic structure. This was the justification for the title of the talk and hence this paper.  However, since this part of the story was subsequently published in \cite{Sigma} \S7, it has been removed from the present notes, so a more appropriate title  might  have been `anatomy of motivic zeta elements'. 
After giving my talk,  I learned  that a similar project had been undertaken by Ecalle using his language of moulds some years before, but with some notable  differences. It would be interesting to compare the approach described here with his and to see how they might  agree or differ.

The project took an unexpected turn when  I subsequently found a \emph{different} family of   canonical elements $\chi_{2n+1} \in \p\dmr$,  which are also solutions to the double shuffle equations to all depths, but constructed by a completely different method. These elements are  different from the $\psi$ family of solutions described above: they satisfy relations associated to cusp forms for the full modular group, and   have a quite different pole structure. In particular,  the elements $\chi_{2n+1}$ and $\psi_{2n+1}$ differ starting from depth 3, but strangely, $\chi_{-1}$ and $\psi_{-1}$ differ  starting from  depth \emph{five}. The elements $\chi$ are constructed out of a single  exceptional element  $\psi_0 \in \p \dmr$ in weight zero which again is possibly not  unique.   Using such an element, one can obtain an unconditional `anatomical' decomposition of zeta elements to all orders.

The statements of these facts are given in these notes, but the not the  proofs, since this aspect of the theory was partially superseded by the paper \cite{Sigma}.  In that paper, I realised how some of this structure could be interpreted geometrically using  the de Rham  fundamental group of the infinitesimal Tate curve. Thus, in small depths at least, a slightly  modified version of the elements $\chi$ (called $\xi$ in \cite{Sigma}), can be understood in terms of the interplay between Grothendieck-Teichm\"uller theory in genus zero and genus one. Whether this geometric interpretation extends to higher depths or not is still open.   Furthermore, I gave a geometric and unconditional interpretation of `anatomy' in  \cite{Sigma} remark 3.8 (where `anatomy' is used in a somewhat different sense from the one used above), which was proved in \cite{MMV} \S20.4.

Although the theory of motivic fundamental groups of curves in  genus $0$ and $1$ clarifies  a \emph{part} of the structure of $\p \dmr$, it  muddies the waters since  many of the structures and phenomena described in the present (older) notes  remain unexplained, and are all the more mysterious for it.
This is the reason why  I decided finally   to make them available.

In conclusion, therefore,  the Lie algebra $\p \dmr$ is a very rich  algebraic object which knows about the action of the motivic Galois group of $\MT(\Z)$ on fundamental groups of curves in both genus zero and one, but  also  seems to contain a considerable amount of further structure, whose geometric interpretation is completely unkown. These notes do not come close to a complete description of $\p \dmr$.
It would seem, with hindsight, that the different algebraic structures which coinhabit $\p \dmr$ should be singled out by placing further restrictions on the structure of the poles.

\subsection{The motivic Lie algebra}

Let $\MT(\Z)$ be the Tannakian category of mixed Tate motives over $\Z$. Our starting point is the fact   \cite{BrMTZ}  that its motivic Galois group, with respect to the de Rham fiber functor,  acts faithfully on the de Rham fundamental
groupoid of the projective line minus three points:
\begin{equation} \label{introGinjects}
\mathrm{Gal}( \MT(\Z) , \omega_{dR}) \hookrightarrow \mathrm{Aut}(\pi_1^{dR} ( \Pro^1 \backslash \{0,1,\infty\}),\tone_0,-\tone_1)\ .
\end{equation} 
The group on the left is a semi-direct product of  the multiplicative group  with a  prounipotent affine group scheme, whose  graded Lie algebra  $\gm$ will be called  the \emph{motivic Lie algebra}.
One knows from the theory of mixed Tate motives that
\begin{equation} \label{introgmfree}
\gm \cong \mathrm{Free}\, \Lie_{\!\Q} \langle \sigma_{3}, \sigma_5, \ldots \rangle\ .
\end{equation}Ê
It is the free graded Lie algebra with one generator $\sigma_{2n+1}$ in every (positive, in these notes)  odd degree $2n+1$, for $n\geq 1$.  The elements $\sigma_{2n+1}$ are only well-defined up to commutators.  
 The injectivity of $(\ref{introGinjects})$ gives rise to a canonical  embedding
 \begin{equation} \label{introgmintoT}
 \gm \hookrightarrow T( e_0 \Q \oplus e_1 \Q)
 \end{equation}
 in the tensor algebra on  $e_0, e_1$,  
 whose image is contained in the free graded Lie algebra generated by $e_0,e_1$. 
 The $\sigma_{2n+1}$ satisfy 
$$\sigma_{2n+1} = \ad(e_0)^{2n} e_1+   \hbox{ higher order terms} \ .$$
 Hereafter,  we shall always identify  $\gm$ with its image under the map   $ (\ref{introgmintoT})$.
  Although one can define
 canonical choices of generators $\sigma_{2n+1}$ of $\gm$ using motivic multiple zeta values (see \S\ref{sectCanRatAssoc}), 
 the image of the map $(\ref{introgmintoT})$ is very poorly understood from an  arithmetic point of view.

  In order to get to grips with $\gm$, we  use the standard relations for multiple zeta values. 
 There are three main types: the motivic relations, associator relations, and double shuffle relations. The  sets of solutions to these equations are related   
 as  follows:
 \begin{equation}\label{associnclusions}
\{\hbox{Motivic Associators}\} \subset \{ \hbox{Drinfeld Associators}\} \subset \{\hbox{Solutions to Dsh}\}\ ,
\end{equation}
where Dsh refers to the regularised double shuffle equations (see below). The second inclusion is a theorem due to Furusho \cite{F} which states that 
the coefficients of  Drinfeld associators  satisfy  regularised double shuffle equations. 
 The set of motivic associators is defined by the motivic relations between multiple zeta values, which are not known explicitly. In these notes we work with the double shuffle equations as a means of controlling $\gm$, because they are the most explicit, and adapted to the depth filtration.
 
  The above sets are in fact affine schemes which are  torsors over certain proalgebraic groups.   Passing to their underlying graded Lie algebras, we obtain the inclusions
  \begin{equation} 
   \gm \subset \mathfrak{grt} \subset \dmr 
   \end{equation}
 where $\mathfrak{grt}$ is the Grothendieck-Teichm\"uller Lie algebra, and 
 $\dmr$ is the Lie algebra of solutions to the double shuffle equations modulo products defined by Racinet \cite{R} (which he denotes $\dmr_0$. We drop the subscript $0$ for convenience).
 A conjecture due to Drinfel'd is equivalent to  the statement $\gm = \mathfrak{grt}$, and a conjecture due to Zagier, which states that all relations between multiple zeta 
 values are implied by double shuffle relations,  suggests  that  $\gm = \dmr$. This would  imply that   $(\ref{associnclusions})$ are all
 equalities.
 It is important to point out  that knowing these equalities would actually be  of very little help in  understanding the structure of $\gm$: for example the equations defining $\dmr$ cannot presently be solved in weight 30 since they are far beyond the reach of present computer algebra systems.  The goal of these notes is to address precisely this problem by constructing the \emph{solutions} to these equations.
 
 \subsection{Polar solutions}
 By a well-known trick of replacing non-commutative formal power series with power series in commuting variables, one can rewrite the defining equations of $\dmr$ as functional equations for certain sequences of polynomials. 
 We define $\p \dmr$, the Lie algebra  of \emph{polar solutions to the double shuffle equations modulo products} to be the set of solutions to these equations in a certain space of rational functions. This idea has previously been exploited in the work of Ecalle \cite{Ecalle1, Ecalle2, Ecalle3}.
 
 The Lie algebra   $\dmr$, which contains the  motivic Lie algebra $\gm$, can 
 be retrieved from $\p \dmr$ (it is not quite true that $\dmr$ is contained in the subspace of elements in $\p \dmr$ with no poles, because, for example, the zeta elements  $\sigma_{2n+1}$ satisfy the double shuffle equations modulo products only in depths $1\leq d \leq 2n$ and not in depth $2n+1$).
 
 Note that there is no obvious way of interpolating between the zeta elements $\sigma_{2n+1}$ for different $n$. However, we show:
 
 \begin{thm} There exist explicit elements $\psi_{2n+1} \in \p \dmr$ of every degree $2n+1$.
  \end{thm} 
 
 These elements  are defined by a sequence of rational functions 
 $$\psi^{(d)}_{2n+1} \in \Q(x_1,\ldots, x_d)$$
  defined by a  closed  formula (definition \ref{defnpsi2n+1}) which is uniform in $n$ and $d$.   To say that these rational functions lie in $\p \dmr$ is equivalent to an infinite sequence of functional relations. Despite their apparent simplicity, it is    highly-non trivial to find such  a family of elements since at each order $d$  their extension to the next order $d+1$ is not unique, and could possibly be obstructed if the wrong choices have  been made at  previous steps.
 
  The elements $\psi_{2n+1}$ are dual to the numbers $\zeta(2n+1)$  and coincide  with the $\sigma_{2n+1}$ to lowest orders (depths 1 and 2). Furthermore, we show 
 
 \begin{thm} There exists an explicit element $\psi_{-1} \in \p \dmr$ in degree $-1$.  
 \end{thm} 

It is defined in a quite different manner from the $\psi_{2n+1}$ by associating rational functions to  elements in a combinatorial Hopf algebra of trees.
Out of these elements $\psi$ we define a  graded Lie subalgebra of $\p \dmr$ denoted by  
 $$\Lo = \Lie_{\! \Q} \langle \psi_{-1}, \psi_{3}, \psi_{5},\ldots \rangle\ .$$
 The main point is that the generators of $\Lo$, and its Lie bracket, which we denote by  $\{ \,, \, \}$,  are completely explicitly defined.
 We conjecture that every element of $\gm$ can be expressed   as a (truncation of) an element in $\Lo$. 
  The  `anatomy' of the title is  a kind of `Taylor expansion' for elements in $\gm$ in terms of the larger   algebra  $\Lo$. 
  As a practical application, an expansion, or `anatomy', of generators  of the motivic Lie algebra
$$\sigma_{2n+1} =\psi_{2n+1} + \hbox{commutators of } \psi\hbox{'s}$$
  leads to extremely compact representations for the structural coefficients of the space of multiple zeta values, in contrast to  the vast tables  which are currently in use. For example, 
   the first two generators $\sigma_3, \sigma_5$ of $\gm$ can be written
\begin{eqnarray}
\sigma_3  &\equiv & \psi_3   \pmod{D^3}  \nonumber \\
\sigma_5  &\equiv & \psi_5 - {1 \over 60} \{ \psi_{-1},\{  \psi_{-1} , \psi_{7} \}\} -  {1 \over 5} \{ \psi_{3},\{  \psi_{3} , \psi_{-1} \}\} \pmod{D^5} \nonumber 
\end{eqnarray}
where the equivalence sign is modulo terms of depth $\geq$ to the weight. 
Further examples are given in \S\ref{ExamplesAnatomy1}.  Since the $\psi$ are explicit, the coefficients in such decompositions yield information about the  denominators of $\sigma_{2n+1}$.

 In \cite{Sigma}, \S7    we  outlined how the theory can be extended, at least in small depths  ($\leq 3$), to associators $\tau$, which justifies the title of the paper.  An `anatomical' decomposition of an associator $\tau$ encodes arithmetic  data relating to all  even zeta values $\zeta(2n)$, $n\geq 1$. 
 We also verified that  case of depth 4 works similarly, but did  not include it in the present notes.  It would be interesting to extend  this construction  to all higher depths.

\subsection{Weight zero and depth-splitting}
All the elements in $\p \dmr$ described above have \emph{odd} weights. In fact, we show that all elements  in  $\p \dmr$  of non-negative weight necessarily
have odd weight, except for the possible exception of weight zero. 

Therefore,  in \S\ref{sectDepthSplit}, we write down an element 
$$\psi_0   \in \p \dmr$$
of weight zero, which is completely explicit. It is a  generating series of a certain  copy of the Witt algebra
in the space of rational functions, equipped with the Ihara bracket, which acts on a Hopf algebra of rational functions encoded by certain trees. The geometric meaning of all this is mysterious. 
The element $\psi_0$  is again not unique, and it would be interesting to determine  fully the Lie subalgebra of $\p \dmr$ in weight zero. 

By twisting with  the element $\psi_0$, we can split the depth filtration on $\p \dmr$ and therefore lift any solution to the depth-graded, or linearized, double shuffle equations to the full double shuffle equations modulo products. This provides an unconditional way to decompose the elements of the motivic Lie algebra $\gm$ in terms of elements in $\p \dmr$. In particular, we define elements 
$$\chi_{2n+1} \in \p \dmr $$
for all $n\geq 1$, obtained by `lifting' the leading parts of the zeta elements $\sigma_{2n+1}$.  We similarly construct an element $\chi_{-1}$ of weight $-1$. Note that the $\chi_{2n+1}$ are quite different from the elements $\psi$ defined above: they have a more complicated pole structure, and satisfy relations. 
 Furthermore, the anatomical decompositions of zeta elements $\sigma_{2n+1}$ obtained in this manner involve, starting from depth 5, a new element $Q_4 \in \p \dmr$, defined in remark \ref{remQ4}, whose  meaning is also unclear. 
 
 The proofs of these statements are not included in the present document, but are essentially elementary statements about functional equations between rational functions. We expect that they can be obtained using the same techniques used to prove the two theorems described above.

 \subsection{Depth graded motivic Lie algebra}
 Since the depth filtration is used in a fundamental way in this theory, the depth-graded motivic multiple zeta values necessarily play a central role.
 It is expected \cite{Depth} that the latter are closely related to the theory of modular forms of level one. In  these notes we prove a number of properties relating to a Lie algebra $ \p \ls $ of polar solutions to the \emph{linearised double shuffle equations}, some of which, but not all, were subsequently reproduced in \cite{Sigma}, along with a geometric interpretation in terms of the motivic fundamental group of the infinitesimal Tate curve.  We refer to the introduction of that paper for further details.

 We formulate some conjectures about $\p \ls$ which are sufficient to ensure that every element of $\gm$ can indeed be expanded in the Lie algebra $\Lo$. To prove this involves some delicate results about cancellation of poles in $\Lo$, which are given in the third part of these notes.
 Some of the statements involving $\p \ls$ were subsequently reproduced in \cite{Sigma}, and especially its appendix.
 However the amount of overlap between it and the present notes is limited, and the proofs are occasionally different, so we decided to keep the relevant parts of these notes unchanged from their 2012 version.  
\\

{\it Acknowledgements}. This work was undertaken in 2012 and  partially supported by ERC grant PAGAP, ref. 257638.

\section{An aside: a canonical construction of rational associators} \label{sectCanRatAssoc}
\subsection{Hoffman basis}Let $\Ho$ denote the graded $\Q$-algebra of motivic multiple zeta values. 
We can define the motivic Drinfel'd associator
\begin{equation}
\label{Phimot}
\Phi^{\mathfrak{m}}(e_0,e_1) = \sum_{w \in \{e_0,e_1\}^{\times}} \zetam(w) w \quad \in \quad \Ho\langle \langle e_0, e_1 \rangle \rangle
\end{equation}
to be the generating series of motivic multiple zeta values $\zetam(w)$ (defined in \cite{BrICM}).\footnote{The following discussion was  later summarised in \cite{BrICM} \S3.1}

\begin{thm} \cite{BrMTZ} The set of Hoffman elements $\zetam(n_1,\ldots, n_r)$, where $n_i=2,3$, forms a graded $\Q$-basis for the vector space $\Ho$.
\end{thm}
Let $X_{3,2} =\{3, 2\}$ be the alphabet on two symbols $2$ and $3$, equipped with the ordering $3<2$. Recall that a Lyndon word $w \in X_{3,2}^{\times}$ is 
a word which is strictly smaller, in the lexicographic ordering, than every strict right factor of $w$.

\begin{thm} \cite{BrMTZ}  $\Ho$ is  isomorphic to  the polynomial ring generated by  the set of Hoffman-Lyndon elements
$\zetam(n_1,\ldots, n_r)$, where  $(n_1,\ldots, n_r)\in X_{3,2}^{\times}$ is a Lyndon word.\end{thm}

In the Hoffman-Lyndon basis, the role of $\zetam(2n+1)$ is played by $\zetam(3,2,\ldots, 2)$. We can simply replace the elements $\zetam(3,2^n)$  (a three followed by $n$ two's) with $\zetam(2n+1)$ in the generating set if we wish.

\subsection{Rational generators} Canonical Hoffman-Lyndon generators $\sh_{2n+1}$ of the motivic Lie algebra can be constructed as follows. 

\begin{defn} For each $n\geq 1$, 
consider the unique homomorphism
$$
\sigma_{2n+1} : \Ho  \To \Q 
$$
which sends any Hoffman-Lyndon element $\zetam(n_1,\ldots, n_r)$ to zero if at least two of the indices $n_i$ are equal to $3$, 
sends $\zetam(2)$ to zero, and satisfies 
$$\sigma_{2n+1} (\zetam(2k +1)) =\delta_{k,n}\ ,$$
where $\delta_{k,n}$ is the Kronecker delta.
In this way we obtain  elements
$$\sh_{2n+1} = \sigma_{2n+1}( \Phi^{\mathfrak{m}} ) \in \Q \langle \langle e_0, e_1 \rangle\rangle\ .$$
\end{defn}

\begin{cor} Each element $\sh_{2n+1}$ is a generator of the motivic Lie algebra.
\end{cor} 
Thus we can write, perhaps artificially, 
$$\gm \cong  \mathrm{Free}\,  \Lie_{\!\Q}\langle \sh_{3}, \sh_{5}, \ldots \rangle$$
\begin{defn}ÊNow let $\lambda \in \Q^{\times}$ and consider the unique linear map
$$ \tau_{\lambda}: \Ho \To \Q$$
which sends all Hoffman basis elements $\zetam(n_1,\ldots, n_r)$ which have  at least one $n_i$ equal to 3 to zero, and satisfies
$\tau_{\lambda}( \zetam(2^n) ) = (2n+1)! \big(2\lambda\big)^{2n}$.
We define
\begin{equation} \label{tauhdef}
\tauh_{\lambda} = \tau_{\lambda} (\Phi^{\mathfrak{m}}) \in \Q \langle \langle e_0, e_1 \rangle \rangle\ .
\end{equation}
\end{defn}
Since the motivic multiple zeta values satisfy the associator relations, we obtain:
\begin{cor}
 The element $\tauh_{\lambda}$ is a (motivic) associator.
 \end{cor} 
By acting on the left via elements of $\exp_{\circ}(\gm)$ (where $\circ$ denotes the Ihara action) we obtain
a large family of canonical rational associators.
\\

The goal of these notes is most emphatically \emph{not}  to define rational associators, since this problem is solved by the above construction.
 Instead, our purpose is to understand in greater detail the arithmetic and internal
structure of the motivic Lie algebra.

\newpage
\begin{center} \bf{I. Preliminaries}
\end{center} 
\section{Power series representations} \label{sectPowerSeriesRep}
We begin with some generalities on formal power series. At first we consider simultaneously the case of group-like elements (corresponding to solutions to shuffle or stuffle equations) which form a group,  and primitive elements (corresponding to  solutions shuffle or stuffle equations `modulo products') which form a Lie algebra. Later on, we focus exclusively on the latter situation.

Let $R$ be a commutative unitary $\Q$-algebra. Consider a non-commutative formal power series in two letters $e_0,e_1$: 
\begin{eqnarray} \label{Phiseries}
\Phi &  \in & R\langle \langle e_0, e_1 \rangle \rangle \\
\Phi & =  & \sum_{w \in \{e_0, e_1\}^{\times}} \Phi_w  w\nonumber  
\end{eqnarray}
where the sum is over all words $w$ in the letters $e_0, e_1$, and $\Phi_w\in R$. The \emph{depth filtration} on $R\langle \langle e_0, e_1\rangle\rangle$  is the decreasing filtration defined as follows:
\begin{equation} \label{Dfildef}
\dd^r  R\langle \langle e_0, e_1 \rangle \rangle = \{ \Phi \in R\langle \langle e_0, e_1 \rangle \rangle:  \Phi_w =0 \hbox{ if } \deg_{e_1} w <r \hbox{ and } w \neq 1 \}\ .
\end{equation}
It consists of series whose non-trivial coefficients are the empty word $1$, or words with $r$ or more occurrences of the letter $e_1$.  
The  depth filtration is induced  by the  $\dd$-degree on words, for  which $e_0$ has degree $0$, and  $e_1$ degree $1$.  Using this notion, we can   uniquely decompose
 every power series $\Phi\in R\langle \langle e_0, e_1 \rangle \rangle$  by its $\dd$-degree:
\begin{equation} \label{Phirdecomp}
\Phi =  \sum_{r\geq 0} \Phi^{(r)} \ , 
\quad  \hbox{ where } \quad  \Phi^{(r)} = \sum_{\deg_{\dd} w = r} \Phi_w w 
\end{equation}

Our starting point is the observation that there is a continuous  $R$-linear isomorphism  from the space of \emph{non-commutative} power series in words $w$ of $\dd$-degree $r$ (denoted by  $R\langle \langle e_0, e_1 \rangle \rangle^{(r)}$) to the  complete $R$-module  of  power series in $r+1$ \emph{commuting} variables:
\begin{eqnarray} \label{wordstopolys}
\rho^{(r)} : R\langle \langle e_0, e_1 \rangle \rangle^{(r)} &  \overset{\sim}{\To}  & R [[ y_0, y_1, \ldots, y_r ]]  \\
e_0^{a_0} e_1 e_0^{a_1} \ldots e_1 e_0^{a_r} & \mapsto & y_0^{a_0} y_1^{a_1} \ldots y_r^{a_r}\ . \nonumber 
\end{eqnarray}
Applying this isomorphism to each component of $(\ref{Phirdecomp})$, we can uniquely  represent any series   $(\ref{Phiseries})$ as an infinite sequence of power series
$$\rho^{(r)} \Phi^{(r)}(y_0,\ldots, y_r) \in R[[ y_0, \ldots, y_r]]\qquad \hbox{ for all } r\geq 0 \ . $$
Any such sequence uniquely defines a series $(\ref{Phiseries})$. We  shall hereafter denote $\rho^{(r)} \Phi^{(r)}$ simply by $\Phi^{(r)} (y_0,\ldots, y_r)$, and refer to it as the depth-$r$ component of $\Phi$.

\subsection{Translation invariance} We will often encounter power series $f \in R [[ y_0, \ldots, y_r]]$ which are invariant under translations of coordinates:
\begin{equation}\label{transinv}
f(y_0+\lambda, \ldots, y_r + \lambda) = f(y_0,\ldots, y_r) \hbox{ for all } \lambda \in R
\end{equation} 
In this case, we shall call the \emph{reduced power series} 
\begin{equation}  \label{funderbar}
   \overline{f}(x_1,\ldots, x_r)= f(0,x_1,\ldots, x_r)  \quad \in \quad \R[[x_1,\ldots, x_r]] \ .\end{equation} 
  By translation invariance, it determines 
$f(y_0,\ldots, y_r)=\overline{f}(y_1-y_0,\ldots, y_r-y_0)$.
As a general rule, we reserve  the variables $x_1,\ldots, x_r$ for  the reduced power series  $\overline{f}$
and use the variables $y_0, \ldots, y_r$  to denote the full power series $f$, bearing in mind that one can pass between the two with impunity by the previous remark.

The goal of the next sections is to translate Hopf-algebra theoretic properties of the power series $\Phi$ into functional equations for the power series $ \Phi^{(r)}(y_0,\ldots, y_r)$.

\section{Shuffle equations for  power series} \label{SECTshuff}
\subsection{Shuffle equations}  \label{sectshuffleequations} The space $R\langle \langle e_0, e_1 \rangle \rangle$ can be made into a complete Hopf algebra as follows.
The multiplication, which is non-commutative, is given by the concatenation of words, and the comultiplication, which we denote by 
\begin{equation} \label{Deltasha}
\Delta_{\sha} : R \langle \langle e_0, e_1 \rangle \rangle \To R \langle \langle e_0, e_1 \rangle \rangle \widehat{\otimes}_R R \langle \langle e_0, e_1 \rangle \rangle
\end{equation}
is the unique continuous  coproduct which  satisfies $\Delta_{\sha} e_i = e_i \otimes 1 + 1 \otimes e_i $  for $i=0,1$. It is cocommutative. The antipode is induced by the linear map
which sends the word $e_{i_1}\ldots e_{i_k}$ to $(-1)^k e_{i_k} \ldots e_{i_1}$ for any $i_1,\ldots, i_k \in \{0,1\}$. The underlying affine group scheme
is the functor
$$R \mapsto G_{\!\sha}(R) = \{ \Phi \in R\langle \langle e_0, e_1 \rangle \rangle^{\times} :   \Delta_{\sha} \Phi = \Phi\widehat{\otimes} \Phi\}$$
which sends any commutative unitary ring $R$ to the set of group-like, invertible formal power series in $e_0, e_1$. The group law on $G_{\!\sha}(R)$ is given by 
\begin{equation} \label{seriesconcat}
(\sum_u \Phi_u u ) \cdot  (\sum_{v} \Phi'_{v} v) = \sum_w \sum_{w=uv} \Phi_u \Phi'_{v} w\ .
\end{equation} 

\begin{defn} We say that an invertible series  $\Phi\in R\langle \langle e_0, e_1 \rangle \rangle^{\times}$ satisfies the \emph{shuffle equations }Êif
it is group-like: 
$\Delta_{\sha} \Phi = \Phi \otimes \Phi.$
A series $\Phi\in R\langle \langle e_0, e_1 \rangle \rangle$ satisfies the \emph{shuffle equations modulo products }Êif it satisfies $\Phi_{e_0}=0$ and is  primitive:
\begin{equation} \label{Phishaprim} \Delta_{\sha} \Phi = 1 \otimes \Phi + \Phi \otimes 1\ .
\end{equation} 
\end{defn}
The set of series satisfying the shuffle equations modulo products forms    a Lie algebra for the bracket $[\Phi, \Phi']_{\!\sha} = \Phi \cdot \Phi' - \Phi' \cdot \Phi$, where $\cdot$ is defined by
$(\ref{seriesconcat}).$

\subsection{The dual Hopf algebra}  \label{sectdualshuff}
Let  $\Q\langle e_0, e_1\rangle$ be the free $\Q$-vector space generated by words in the letters $e_0,e_1$, including the empty word which we denote by $1$. 
It is equipped with the shuffle product $\sha$,  defined recursively by
\begin{equation} \label{shuffprod}
e_i w \sha e_j w' = e_i (w \sha e_j w') + e_j (e_i w \sha w')
\end{equation} 
for all $w,w'\in \{e_0,e_1\}^\times$ and $i,j\in \{0,1\}$, and the property $1\sha w=w\sha 1=w$.  It is a Hopf algebra for the deconcatenation coproduct 
\begin{equation}\label{DeltaDeconc}
\Delta_{dec}: 
e_{i_1}\ldots e_{i_k} = \sum_{j=0}^k e_{i_1}\ldots e_{i_j} \otimes e_{i_{j+1}} \ldots e_{i_k} \ .
\end{equation}
 The antipode is  the linear map $e_{i_1}\ldots e_{i_k} \mapsto (-1)^k e_{i_k}\ldots e_{i_1}$.    There is  a pairing
\begin{eqnarray} \label{pairing}
R \langle \langle e_0, e_1 \rangle \rangle \otimes_{\Q} \Q \langle e_0, e_1 \rangle & \To & R  \\
\Phi \otimes w &  \mapsto  & \langle \Phi, w \rangle   \nonumber 
\end{eqnarray}
where $\langle \Phi,w\rangle = \Phi_w$ is the coefficient of the word $w$ in $\Phi$.  The algebra $\Q\langle e_0, e_1 \rangle$ is graded
by the weight; the weight of a word $w$ being the number of letters in that word.
The pairing $(\ref{pairing})$ induces a  (topological) duality between the complete Hopf algebra $\Q\langle \langle e_0, e_1\rangle \rangle$ and the graded Hopf algebra $\Q \langle e_0, e_1 \rangle$. In particular,  the functor  $R\mapsto G_{\!\sha}(R)$  is  $\mathrm{Spec\, } \Q  \langle e_0, e_1\rangle$.

This duality implies that a  linear map $\Phi: \Q\langle e_0, e_1 \rangle\rightarrow R$ is a homomorphism for the shuffle multiplication,
or $\Phi_{w} \Phi_{w'} = \Phi_{w \sha w'}$ for all $w,w'\in \{e_0,e_1\}^\times$,  if and only if  the generating series
$$\Phi=\sum_w \Phi_w w  \in R \langle \langle e_0,e_1 \rangle \rangle$$
is group-like. The series $\Phi$ is primitive (satisfies $(\ref{Phishaprim})$) if and only if 
 \begin{equation} \label{Phiwshuffmodprod} \Phi_{w \sha w'}=0 \quad   \hbox{  for all  }  \quad  1 \neq w,w'\in \{e_0,e_1\}^\times\ . 
 \end{equation}
In other words, $\Phi$ satisfies the shuffle relations modulo products if and only if its coefficients $\Phi_w$ satisfy the equations $(\ref{Phiwshuffmodprod})$ and vanishes in depth zero.

\subsection{Translation invariance}
\begin{lem} \label{lemtransinv} Suppose that $\Phi\in R\langle \langle e_0, e_1 \rangle \rangle$ satisfies the shuffle equations modulo products $(\ref{Phishaprim})$.
Then its depth-$r$ components $\Phi^{(r)}(y_0,\ldots, y_r)$ are translation invariant $(\ref{transinv})$.
\end{lem}
\begin{proof} 
 Let $\pi_0: R\langle\langle e_0, e_1 \rangle \rangle \rightarrow R$ denote the continuous linear map which sends the word $e_0$ to $1$ and all other words to $0$, and consider the map
$\delta_0 = (\pi_0 \otimes id) \circ \Delta_{\sha}$. It defines an operator
$\delta_0: R\langle\langle e_0, e_1 \rangle \rangle \rightarrow R\langle\langle e_0, e_1 \rangle \rangle$, satisfying $\delta_0 (S\cdot T)= S \cdot \delta_0(T) +  \delta_0(S) \cdot T$, 
and 
$$\delta_0 (e_0^{a_0} e_1 \ldots e_1 e_0^{a_r} )= \sum_{i=0}^r a_i  \, e_0^{a_0} e_1 \ldots e_0^{a_i-1} \ldots  e_1 e_0^{a_r} $$
 for all non-negative integers $a_0,\ldots, a_r$. 
By assumption, $\Phi\in R\langle \langle e_0, e_1 \rangle \rangle$  satisfies  $\Delta_{\sha}\Phi = 1 \otimes \Phi + \Phi \otimes 1$, and $\Phi_{e_0}=0$,  and hence  $\delta_0 \Phi=\Phi_{e_0}=0$. The previous equation implies that for all $r$,
 \begin{equation} \label{parsums}  \nabla \Phi^{(r)} (y_0,\ldots, y_r) =0\ , \quad   \hbox { where } \quad  \nabla=\sum_{i=0}^r { \partial  \over \partial y_i }\ .
 \end{equation} 
 If we write  $\Phi^{(r)}_{\lambda}(y_0,\ldots, y_r)= \Phi^{(r)}(y_0,\ldots, y_r) - \Phi^{(r)}(y_0+\lambda,\ldots, y_r+\lambda)$, equation $(\ref{parsums})$ gives ${\partial \over \partial \lambda} \Phi^{(r)}_{\lambda}=0$, and hence $\Phi^{(r)}$ is translation invariant.   \end{proof}
 
 \begin{defn} Let  $R\langle \langle e_0, e_1 \rangle \rangle^{tr} \subset R \langle \langle e_0, e_1 \rangle \rangle$ denote the subspace of  translation invariant series. 
 \end{defn} 
 Since  $R\langle \langle e_0, e_1 \rangle \rangle^{tr} $ is the 
 kernel of $\delta_0 = (\pi_0 \otimes id) \circ \Delta_{\sha}$, one easily verifies the
 \begin{cor}  \label{cortransinv}$R\langle \langle e_0, e_1 \rangle \rangle^{tr}$ is a Hopf subalgebra of $ R \langle \langle e_0, e_1 \rangle \rangle$.
 \end{cor} 
Note that a version of the previous lemma also holds for series $\Phi$ which satisfy the shuffle equations $\Delta_{\sha} \Phi = \Phi \otimes \Phi$ and the condition $\Phi_{e_0}=0$. The proof is similar.  

\subsection{Shuffle concatenation and involution} The Hopf algebra structure on  non-commutative  formal power series  $(\ref{Deltasha})$, $(\ref{seriesconcat})$ in $ R\langle \langle e_0,e_1 \rangle \rangle$  defines a Hopf algebra stucture on sequences of commutative power series via the rule $(\ref{wordstopolys})$.

First of all, by $(\ref{wordstopolys})$,  if
 $\Phi_1 \in R [[ y_0, \ldots, y_r ]]$  and  $\Phi_2 \in R [[ y_0, \ldots, y_s ]]$,  $(\ref{seriesconcat})$ becomes 
  \begin{equation} \label{shuffony}
\Phi_1 \cdot \Phi_2  (y_0, \ldots, y_{r+s})= \Phi_1(y_0, y_1,\ldots, y_r) \Phi_2(y_r, y_{r+1}, \ldots, y_{r+s})\ .
\end{equation}
The antipode on $R\langle \langle e_0, e_1 \rangle \rangle$ defines an involution
\begin{eqnarray} \label{sigmaantipodedef}
\sigma: R [[  y_0, \ldots, y_r ]]  & \To & R [[ y_0, \ldots, y_r ]] \\
\Phi(y_0,\ldots, y_r)  & \mapsto & (-1)^r \Phi(-y_r,-y_{r-1},\ldots, -y_1, -y_0) \nonumber
\end{eqnarray}
For series  which are translation invariant,    $(\ref{shuffony})$ induces a map
\begin{eqnarray} 
R [[ x_1, \ldots, x_r ]] \widehat{\otimes}_R  R [[ x_1, \ldots, x_s ]] & \To & R [[ x_1, \ldots, x_{r+s} ]] \nonumber \\
\overline{\Phi}_1\otimes \overline{\Phi}_2 & \mapsto & \overline{\Phi}_1 \cdot \overline{\Phi}_2  \nonumber
\end{eqnarray}
where
\begin{equation} \label{shuffleconcat}
\overline{\Phi}_1 \cdot \overline{\Phi}_2(x_1,\ldots, x_{r+s}) =  \overline{\Phi}_1(x_1,\ldots, x_r) \overline{\Phi}_2( x_{r+1}-x_r, \ldots, x_{r+s}-x_r) \ .\end{equation}
The map $(\ref{sigmaantipodedef})$ becomes the involution:
\begin{eqnarray} \label{sigmaonx}
\sigma: R  [[ x_1, \ldots, x_r ]] & \To & R [[ x_1, \ldots, x_r ]] \\
\overline{\Phi}(x_1,\ldots, x_r)  & \mapsto & (-1)^r \overline{\Phi}(x_r-x_{r-1},\ldots, x_r-x_1, x_r) \nonumber
\end{eqnarray}
It is convenient to express the coproduct on commutative power series using a  certain change of variables which was denoted by $\sharp$ in \cite{IKZ}:
\begin{eqnarray} \label{sharpdef}
\sharp:  \Q[[x_1,\ldots, x_r]]  & \To &   \Q[[x_1,\ldots, x_r]]  \\
f(x_1,\ldots, x_r)  & \mapsto &  f(x_1,x_1+x_2,\ldots, x_1+ \ldots + x_r) \nonumber
\end{eqnarray}

\subsection{A series}

For all $r\geq 0$, consider the   element 
$$  \e(y_0,\ldots, y_r)   =   \sum_{a_0, \ldots, a_r \geq 0}  (e_0^{a_0} e_1 \ldots e_1 e_0^{a_r})\,   y_0^{a_0} \ldots y_r^{a_r} \ ,$$
in $ \Q\langle e_0, e_1 \rangle \widehat{\otimes}_{\Q} \Q [[ y_0,\ldots, y_r]]  $, where $\Q\langle e_0, e_1\rangle$ is the Hopf algebra equipped with the shuffle multiplication defined in \S\ref{sectdualshuff}.  Note that $\e(y_0)$ is simply  the exponential of $e_0 y_0$   with respect to the shuffle product.
Then  $(\ref{wordstopolys})$  can be written   as a pairing:
$$\Phi^{(r)}(y_0,\ldots, y_r) = \langle \Phi,    \e(y_0,\ldots, y_r)\rangle\ ,$$
and the reduced depth-$r$ component of $\Phi$ can be computed from the formula:
\begin{equation} \label{redpoweraspairing}
\overline{\Phi}^{(r)}(x_1,\ldots, x_r) = \langle \Phi,  \e(0,x_1,\ldots, x_r)\rangle\ .
\end{equation}

\subsection{Shuffle equations and coproduct for power series}
Let $X_r= \{\xx_1,\ldots, \xx_r\}$, and let 
$\Q\langle X_r \rangle$ denote the vector space spanned by words in the  letters $\xx_1,\ldots, \xx_r$, and equipped  with the shuffle product. Let $\Q X_r$ denote the vector space with basis given by the elements of  $X_r$.
The series  above defines  a  map:
\begin{eqnarray}
\e^{\sharp} : \Q X_r \oplus \Q \langle X_r \rangle & \To &  \Q\langle e_0, e_1 \rangle \widehat{\otimes}_{\Q} \Q[[x_1,\ldots, x_r]] \nonumber \\ 
\e^{\sharp} (\alpha, \xx_{i_1}\ldots \xx_{i_k}) &= & \e(\alpha, \alpha+x_{i_1},\ldots, \alpha + x_{i_1}+  \ldots + x_{i_k})  \nonumber
\end{eqnarray}  which is  extended by  linearity  in the second argument, and where $\alpha \in \Q X_r$ is viewed as an element of $\Q[[x_1,\ldots, x_r]]$ via the linear map $\xx_i \mapsto x_i$. 
  In other words, for all $\lambda_1,\lambda_2 \in \Q$ and $w_1, w_2 \in X^{\times}_r $, we have $\e^{\sharp} (\alpha, \lambda_1 w_1+\lambda_2 w_2) = \lambda_1 \e^{\sharp} (\alpha, w_1) +\lambda_2 \e^{\sharp} (\alpha, w_2)$. 
\begin{lem} \label{lemeeisshuff} The map $\e^{\sharp}$ is a homomorphism:  
\begin{equation}  \label{eshuffequation}
\e^{\sharp}( \alpha_1, w_1) \sha \e^{\sharp} ( \alpha_2, w_2) = \e^{\sharp} ( \alpha_1+\alpha_2, w_1 \sha w_2)
\end{equation} 
for all $\alpha_1,\alpha_2 \in \Q X_r$ and $w_1,w_2 \in X_r^{\times}$.
\end{lem}
\begin{proof} Consider the linear maps $\partial_{0}, \partial_{1}: \Q  \langle e_0, e_1 \rangle\rightarrow  \Q \langle e_0, e_1\rangle$ whose action is defined on words by
$\partial_i (e_j w) = \delta_{ij} w$, where $\delta_{ij} $ is the Kronecker delta. We have
\begin{eqnarray}
\partial_0 \e(\alpha_0,\alpha_1, \ldots,  \alpha_r) &  = &  \alpha_0 \e(\alpha_0,\alpha_1, \ldots,  \alpha_r) \nonumber \\
\partial_1 \e(\alpha_0,\alpha_1, \ldots,  \alpha_r) &  = &  \e(\alpha_0+\alpha_1, \alpha_2, \ldots,   \alpha_r)\ . \nonumber 
\end{eqnarray}
Using the fact that  the operators $\partial_i$ are derivations for the shuffle product and equation $(\ref{shuffprod})$
one  sees that the formula $(\ref{eshuffequation})$ is stable under $\partial_0$, $\partial_1$.   Since $\ker \partial_0 \cap \ker \partial_1= \Q$, the lemma  follows easily by induction on the length of $w_1, w_2$. 
\end{proof}

Let us write, for any word $w=\xx_{i_1}\ldots \xx_{i_r}$ of 
length $r$
\begin{equation} \label{fsharpnotation} 
f^{\sharp}(w) = f(x_{i_1}, x_{i_1}+x_{i_2}, \ldots, x_{i_1}+ \ldots+ x_{i_r})
\end{equation}
and extend by linearity to linear combinations of such words.

 Let $p,q\geq 1$, and denote the $(p,q)$-th depth-graded component of $\Delta_{\sha}$ by 
$$\Delta^{p,q}_{\sha}: \gr^{p+q}_{\dd} R \langle \langle e_0, e_1 \rangle \rangle^{tr} \rightarrow \gr^{p}_{\dd} R \langle \langle e_0, e_1 \rangle \rangle^{tr} \, \widehat{\otimes}_R \,\gr^{q}_{\dd} R \langle \langle e_0, e_1 \rangle \rangle^{tr}$$
It induces a map on commutative power series we also denote by $\Delta_{\sha}^{p,q}$:
$$\Delta^{p,q}_{\sha}: R [[x_1,\ldots, x_{p+q}]]  \rightarrow R[[x_1,\ldots, x_p]]  \widehat{\otimes}_R \,R[[x_1,\ldots, x_q]]$$
Let  $m^{p,q} : R[[x_1,\ldots , x_p ]] \widehat{\otimes}_R \,R[[x_1,\ldots, x_q]] \rightarrow R[[x_1,\ldots, x_{p+q}]]$ be the continuous homomorphism  which sends
$x_i\otimes 1 \mapsto x_i$ for $1 \leq i \leq p$ and $1 \otimes x_i\mapsto x_{i+p}$ for $1\leq i \leq q$.

\begin{prop} \label{propshuffleequations} For all  $f \in R [[x_1,\ldots, x_{p+q}]]$ we have 
\begin{equation}Ê
m^{p,q} (\sharp \otimes \sharp) \Delta^{p,q}_{\sha} f = f^{\sharp}(x_1\ldots x_p \sha x_{p+1} \ldots x_{p+q})\ . \end{equation}
\end{prop}

\begin{proof}
Let $f \in R[[x_1, \ldots, x_{p+q}]]$. It corresponds, via $(\ref{wordstopolys})$ to a translation-invariant series $\Phi \in R\langle \langle e_0, e_1 \rangle \rangle^{tr}$ which is homogeneous of $\dd$-degree  $p+q$.
We have
\begin{eqnarray}
m^{p,q} (\sharp \otimes \sharp) \Delta^{p,q}_{\sha} f & = & \langle \Delta_{\sha}^{p,q} \Phi^{(p+q)},  m^{p,q}( \sharp \otimes \sharp) \e(0, \xx_1\ldots\xx_p) \otimes  \e(0, \xx_{1}\ldots\xx_{q})\rangle \nonumber \\
& = & \langle \Delta_{\sha}^{p,q} \Phi^{(p+q)}, \e^{\sharp} (0, \xx_1\ldots\xx_p) \otimes  \e^{\sharp} (0, \xx_{p+1}\ldots\xx_{p+q})\rangle \nonumber \\
& = & \langle \Phi^{(p+q)},  \e^{\sharp}(0, \xx_1\ldots\xx_p) \sha  \e^{\sharp}(0, \xx_{p+1}\ldots\xx_{p+q}) \rangle   \nonumber \\
 & =  &   \langle \Phi^{(p+q)},  \e^{\sharp}(0 , \xx_1\ldots\xx_p\sha  \xx_{p+1}\ldots\xx_{p+q}) \rangle \nonumber  \\
 & = & f^{\sharp} ( \xx_1\ldots\xx_p\sha  \xx_{p+1}\ldots\xx_{p+q})\ , \nonumber
 \end{eqnarray}
 by the duality of Hopf algebras defined in $(\ref{pairing})$, and  lemma  \ref{lemeeisshuff}.
 \end{proof}

\begin{cor} \label{corshuffleequations}A  series $\Phi\in R\langle \langle e_0, e_1 \rangle \rangle$ satisfies the shuffle equations modulo products if and only if  it is translation invariant, and  for all $p,q \geq 1$, 
\begin{equation} \label{shuffleequation}
(\overline{\Phi}^{(p+q)})^{\sharp}( \xx_1 \ldots \xx_p \sha \xx_{p+1} \ldots \xx_{p+q})=0 \ .
\end{equation} \end{cor}

\begin{proof} 
Clearly
$
\Delta_{\sha} \Phi = 1\otimes \Phi + \Phi \otimes 1$ if and only if $ \Delta^{p,q}_{\sha} \Phi^{(p+q)} =0$ for all $p,q \geq 1$.
Since $m^{p,q}(\sharp \otimes \sharp)$ is injective, this is equivalent  by proposition \ref{propshuffleequations} to 
$(\ref{shuffleequation})$.  
\end{proof}
Similarly, an invertible series $\Phi\in R\langle \langle e_0, e_1 \rangle \rangle$ satisfies the shuffle equations (is group-like) and the condition $\Phi_{e_0}=0$ if and only if it is translation invariant and 
$$(\overline{\Phi}^{(p+q)})^{\sharp}( \xx_1 \ldots \xx_p \sha \xx_{p+1} \ldots \xx_{p+q})= (\overline{\Phi}^{(p)})^{\sharp}( \xx_1 \ldots \xx_p)  (\overline{\Phi}^{(q)})^{\sharp}( \xx_{p+1} \ldots \xx_{p+q})$$  

\subsection{Examples}
We shall call equation $(\ref{shuffleequation})$ the $(p,q)$-th shuffle equation. 
Since $\Delta_{\sha}$ is cocommutative, the $(p,q)$-th equation is
equivalent to the $(q,p)$-th equation. 

 In depth 2, there is a unique shuffle equation of type  $(1,1)$, namely 
$f^{\sharp}( \xx_1 \sha \xx_2 ) = 0$.
Since $\xx_1\sha \xx_2 = \xx_1 \xx_2 + \xx_2 \xx_1$, this is equivalent to the equation:
\begin{eqnarray}
f(x_1, x_1+x_2) + f(x_2, x_1+x_2)=0\ .
\end{eqnarray}
In depth 3, there is a unique equation of type $(1,2)$ given by  $f^{\sharp}( \xx_1 \sha \xx_2 \xx_3 ) = 0$. \mbox{Expanding}  out  $\xx_1\sha \xx_2\xx_3 = \xx_1 \xx_2\xx_3 + \xx_2 \xx_1 \xx_3 + \xx_2 \xx_3 \xx_1$  gives the equation
\begin{multline}  \label{depth3shuffle}
f(x_1,x_1+x_2, x_1+x_2+x_3) + f(x_2,x_1+x_2, x_1+x_2+x_3)  \\ +  f(x_2, x_2+x_3,x_1+x_2+x_3)=0\ .
\end{multline}
In general, there are $\lfloor {r \over 2} \rfloor$ distinct equations in depth $r$.

\section{Stuffle equations for power series} \label{SECTstuff}
Let $Y=\{ \y_n, n\geq 1\}$ denote an alphabet with one letter $\y_i$ for every  $i \geq 1$, and  let $R$ be a commutative unitary $\Q$-algebra as above.
Consider the ring $R\langle \langle Y \rangle \rangle$ of non-commutative formal power series in $Y$ equipped with the concatenation product, which we denote by $\studot$ in this context. It is a complete Hopf algebra for the coproduct
$$\Deltat : R\langle \langle Y \rangle \rangle \rightarrow R \langle \langle Y \rangle \rangle \widehat{\otimes}_{R} R \langle \langle Y \rangle \rangle  $$ 
which, after setting $\y_0=1$,  is   defined on generators by the formula
\begin{equation}\label{Deltastuff}
\Deltat \y_n = \sum_{i=0}^n \y_i \otimes \y_{n-i}\ .\end{equation}
The depth filtration, again denoted by $\dd$, is the decreasing filtration defined as follows:
\begin{equation} \dd^r  R\langle \langle Y \rangle \rangle = \{ \Phi \in R\langle \langle Y \rangle \rangle:  \Phi_w =0 \hbox{ if } | w |<r \hbox{ and } w \neq 1\}\ , 
\end{equation}
where $|w|$ denotes the length of a word $w\in Y$.
The filtration is induced by the $\dd$-degree, for which every $\y_i$ has degree $1$.  
We denote by  $R\langle \langle Y \rangle \rangle^{(r)}$ the space of power series consisting only of terms of $\dd$-degree $r$, and  
  associate  to such a series a commutative power series via the  continuous  linear map:
\begin{eqnarray} \label{stufftopowerseries}
R\langle \langle Y \rangle \rangle^{(r)} & \To & R [[ x_1, \ldots, x_r]]  \\
\y_{i_1} \ldots \y_{i_r} & \mapsto & x_1^{i_1-1} \ldots x_{r}^{i_r-1} \nonumber 
\end{eqnarray}
This map is  an isomorphism of complete $R$-modules.
Thus, in the same way as before, a series $\Phi \in R \langle \langle Y \rangle \rangle$ can be written as the sum of depth-$r$ components
$\Phi = \sum_{r\geq 0}\Phi^{(r)}$, each of which can be uniquely represented as an element of $R[[x_1,\ldots, x_r]]$.

We say that a series  $\Phi \in  R\langle \langle Y \rangle \rangle^{\times}$ satisfies the \emph{stuffle equations} if it is group-like: $\Deltat \Phi = \Phi \otimes \Phi$. We say that a series
 $\Phi \in  R\langle \langle Y \rangle \rangle$  satisfies the \emph{stuffle relations modulo products} if it is primitive: 
 \begin{equation}  \label{stuffprim}
 \Deltat \Phi = 1\otimes \Phi + \Phi \otimes 1
 \end{equation}
The set of primitive series $\Phi$ forms a Lie algebra with respect to the bracket $[\Phi_1, \Phi_2]_{\star} = \Phi_1 \studot \Phi_2 - \Phi_2 \studot \Phi_1$  in the usual manner.

\subsection{The dual Hopf algebra} \label{sectStuffdual} Consider the   algebra
$\Q\langle Y\rangle$ equipped with the so-called stuffle product \cite{R}, which   is  defined recursively by
\begin{equation} \label{stuffprod}
\y_i w \stu \y_j w' = \y_i (w \stu \y_j w') + \y_j (\y_i w \stu w') + \y_{i+j} (w \stu w')
\end{equation} 
for all $w,w'\in Y^\times$ and $i,j \geq 1$, and the property that the empty word $1$ satisfies $1\stu w=w\stu 1=w$.  
It is graded by the weight, for which $\y_n$ has weight $n$.
The algebra $\Q\langle Y \rangle$ is a
commutative Hopf algebra for the deconcatenation coproduct. The pairing
\begin{eqnarray}  R \langle \langle Y \rangle \rangle  \widehat{\otimes}_{\Q} \Q \langle Y \rangle  & \To&  R  \\ 
 \Phi \otimes w &\mapsto &  \langle \Phi, w \rangle\nonumber
\end{eqnarray}
where $\langle \Phi, w\rangle = \Phi_w$, identifies $\Q\langle \langle Y \rangle \rangle$ with the topological dual Hopf algebra of $\Q \langle Y \rangle$.

A linear map $\Phi: \Q\langle Y \rangle\rightarrow R$ is a homomorphism for the stuffle multiplication,
or $\Phi_{w} \Phi_{w'} = \Phi_{w \stu w'}$ for all $w,w'\in Y^{\times}$,  if and only if  the series
$$\Phi=\sum_w \Phi_w w  \in R \langle \langle Y \rangle \rangle^{\times}$$
is group-like for $\Deltat$.
 Likewise, $\Phi$ satisfies the stuffle relations modulo products if and only if 
$\Phi_{w\stu w'} = 0$ for all $w,w' \in Y^{\times}$ with $w,w' \neq 1$ (whence its name).

\subsection{Stuffle automorphism and concatenation}
The map $ \y_{i_1}\ldots \y_{i_r}\mapsto \y_{i_r} \ldots \y_{i_1}$ which reverses words in $Y$ 
(which is \emph{not} the antipode) commutes with $\Deltat$, and in particular,  preserves the set of series $\Phi$ satisfying
$(\ref{stuffprim})$. It corresponds to the map
\begin{eqnarray} \label{stuffinvol}
\upsilon: R[[x_1,\ldots, x_r]]  & \To &  R[[x_1,\ldots, x_r]]  \\
f(x_1,\ldots, x_r) & \mapsto &  f(x_r,\ldots, x_1)\nonumber 
\end{eqnarray}
The multiplication of formal power series in $R\langle \langle Y \rangle \rangle$ is induced by concatenation in the $Y$ alphabet, and translates into the following operation on power series:
\begin{eqnarray} \label{stuffleconcat}
\qquad  R[[x_1,\ldots, x_p]] \widehat{\otimes}_R R[[x_1,\ldots, x_q]] & \To & R[[x_1,\ldots, x_{p+q}]] \\
 f_1(x_1,\ldots, x_p) \studot f_2(x_1,\ldots, x_q) &= & f_1(x_1,\ldots, x_p) f_2(x_{p+1},\ldots, x_{p+q}) \nonumber 
 \end{eqnarray}
Given series $\Phi_1 = \sum_{r\geq 0} \Phi_1^{(r)}$ and $\Phi_2 = \sum_{r\geq 0} \Phi_2^{(r)}$ we write
\begin{equation} \label{defstuprod}
(\Phi_1 \studot \Phi_2)^{(p+q)} = \sum_{p+q=r} \Phi_1^{(p)} \studot \Phi_2^{(q)}\ .
\end{equation} 

\subsection{The stuffle equations for power series}
Let us write
$$\alpha(x_1,\ldots, x_r) = \sum_{a_1,\ldots, a_r \geq 1} \y_{a_1} \ldots \y_{a_r} x_1^{a_1-1} \ldots x_r^{a_r-1}$$
viewed as an element of $\Q\langle  Y \rangle  \widehat{\otimes}_{\Q} \Q[[x_1,\ldots, x_r]]$,
where $\Q\langle Y \rangle$ is the Hopf algebra defined in \S\ref{sectStuffdual}. 
 Then, by a slight abuse of notation,  we can write the map  
$(\ref{stufftopowerseries})$ as a pairing:
\begin{equation} \label{Phistuffrpaired}
\Phi^{(r)}(x_1,\ldots, x_r) = \langle \Phi, \alpha(x_1,\ldots, x_r)\rangle\ .
\end{equation}
\begin{lem}  \label{lemalphastuffequation}ÊThe expression  $\alpha(x_1,\ldots, x_p) \stu \alpha(x_{p+1},\ldots, x_{p+q}) $ equals
\begin{multline}
 \alpha(x_1) \studot (\alpha(x_2,\ldots, x_p) \stu  \alpha(x_{p+1},\ldots, x_q)) +\\ \alpha(x_{p+1}) \studot (\alpha(x_1,\ldots, x_p) \stu  \alpha(x_{p+2},\ldots, x_q))+  \\
 {\alpha(x_1) -\alpha(x_{p+1}) \over x_1-x_{p+1}} \studot (\alpha(x_2,\ldots, x_p) \stu \alpha(x_{p+2},\ldots, x_{p+q}))\ .
 \end{multline}
\end{lem} 
\begin{proof} Observe that $\alpha(x_1,\ldots, x_r) = \alpha(x_1) \studot \alpha(x_2) \studot \ldots \studot \alpha(x_r)$. The formula follows  from the recursive
definition   $(\ref{stuffprod})$ of the stuffle product, on noticing that 
$$\sum_{m,n\geq 1} \y_{m+n} \,x_1^{m-1} x_2^{n-1} = {\alpha(x_1) -\alpha(x_2) \over x_1-x_2}\ . $$
\end{proof}
In order to write down the stuffle equations compactly, let us define   an operator
on the space of  sequences
of functions $f^{(1)}(x_1), f^{(2)}(x_1,x_2), \ldots$ by the formula  \begin{equation} \label{shiftoperatorforstuffle}
s_i f^{(r)}(x_1,\ldots, x_r ) = f^{(r+1)}(x_i, x_1, \ldots, x_{i-1}, x_{i+1}, \ldots, x_r) \quad  \hbox{ for } 1 \leq i \leq r
\end{equation}Ê
Now let us define the stuffle equations recursively by 
\begin{eqnarray} \label{Stuffequationsdefn}
f^{(r)} (1 \stu \x_{1} \ldots \x_r) & = &f^{(r)} (\x_{1} \ldots \x_r \stu 1)= f^{(r)} (\x_{1}, \ldots ,\x_r)  \\
f^{(r)}  (\x_{1} \ldots \x_i \stu \x_{i+1} \ldots \x_r ) & =  & s_1 f^{(r-1)}  (\x_{2} \ldots \x_i \stu \x_{i+1} \ldots \x_r )  \nonumber \\
&+ &   \qquad  s_{i+1} f^{(r-1)}  (\x_{1} \ldots \x_i \stu \x_{i+2} \ldots \x_r )  \nonumber \\
 &+ & \quad \qquad \Big({s_1 - s_{i+1} \over x_1 -x_{i+1} }\Big) f^{(r-2)} ( \x_{2} \ldots \x_i \stu \x_{i+2} \ldots \x_r ) \nonumber 
\end{eqnarray} 
where $1 \leq i \leq r$.

\subsection{Stuffle equations and the coproduct}
Now consider, for all $p,q\geq 1$,  the following piece of the stuffle coproduct:
\begin{equation}
\Delta_{\star}^{p,q}:    R \langle \langle Y \rangle \rangle  \To   R \langle \langle Y \rangle \rangle^{(p)} \, \widehat{\otimes}_{R}\,   R\langle \langle Y \rangle \rangle^{(q)} \ .
\end{equation} 
It factors through $\dd^{\max\{p,q\}}  R \langle \langle Y \rangle \rangle/ \dd^{p+q+1} R \langle \langle Y \rangle \rangle$, since the image of an element of $\dd$-degree $n$ under $\Delta_{\star}$ involves terms with $\dd$-degrees  $n-k \otimes n-\ell$ where $ 0 \leq k+ \ell \leq n$. 
Transposing to commutative power series, it corresponds to a map 
\begin{equation}
\Delta_{\star}^{p,q} : \bigoplus_{\max\{p,q\}\leq n\leq p+q} R[[x_1,\ldots, x_n]] \To  R[[x_1,\ldots, x_p]]  \, \widehat{\otimes}_{R}\,   R[[x_1,\ldots, x_q]]  \ .
\end{equation} 
Let $n_{p,q}: R[[x_1,\ldots, x_p]]  \widehat{\otimes}_{R}\,   R[[x_1,\ldots, x_q]] \rightarrow R[[x_1,\ldots, x_{p+q}]]$ denote the continuous homomorphism  which sends
$x_i \otimes 1$ to $x_i$ and $1 \otimes x_j$ to $x_{p+j}$ for all $1 \leq i \leq p$, $1 \leq j \leq q$.

\begin{prop}  \label{propstuffleequations}
Let $f\in \bigoplus_{0\leq n\leq p+q} R[[x_1,\ldots, x_n]] $ with components $f^{(n)}$. Then 
$$ n_{p,q} \Delta_{\star}^{p,q} f =  f( \xx_1 \ldots \xx_p \stu \xx_{p+1} \ldots \xx_{p+q})$$
\end{prop}

\begin{proof} The element $f$ corresponds to a series $\Phi \in R \langle \langle Y \rangle \rangle$. 
By the duality of Hopf algebras described in \S\ref{sectStuffdual} we have
$$ \langle \Delta_{\stu}^{p,q} \Phi,  \alpha(\x_1,\ldots, \x_p)  \otimes  \alpha(\x_{p+1},\ldots, \x_{p+q}) \rangle  =  \langle \Phi,    \alpha(\x_1,\ldots, \x_p)   \stu   \alpha(\x_{p+1},\ldots, \x_{p+q}) \rangle $$
The result follows easily from lemma   \ref{lemalphastuffequation}.
\end{proof} 

\begin{cor}  \label{corstuffleequations}
A series $\Phi\in R \langle \langle Y \rangle \rangle$ satisfies the stuffle equations modulo products if and only if   its   components    satisfy:
\begin{equation} \label{stuffleequation}
\Phi( \xx_1 \ldots \xx_p \stu \xx_{p+1} \ldots \xx_{p+q})=0\  \quad \hbox{ for all } p, q \geq 1
\end{equation} 
\end{cor}

\begin{proof} A series $\Phi \in R \langle \langle Y \rangle \rangle$ is primitive for $\Delta_{\stu}$ if and only if the images of its components in $\dd^{\max\{p,q\}}  R \langle \langle Y \rangle \rangle/ \dd^{p+q+1} R \langle \langle Y \rangle \rangle$ are in the kernel of  $\Delta^{p,q}_{\stu}$  for all $p,q \geq 1$. Conclude using proposition $\ref{propstuffleequations}$ and the fact that $n_{p,q}$ is injective. 
\end{proof}

\subsection{Examples} 
We call $(\ref{stuffleequation})$ the $(p,q)^{\mathrm{th}}$ stuffle equation. It is equivalent to the $(q,p)^{\mathrm{th}}$ stuffle equation. 
In depth two we have a single stuffle equation of type $(1,1)$ which corresponds to the equation
$\y_a \stu \y_b = \y_a \y_b + \y_b \y_a + \y_{a+b}$:
\begin{equation} \label{exdepth2stuffle}
\Phi^{(2)}(x_1,x_2) + \Phi^{(2)}(x_2,x_1) =   {\Phi^{(1)}(x_2) -\Phi^{(1)}(x_1)  \over x_1 -x_2} 
\end{equation} 
and in depth three a single stuffle equation of type $(1,2)$ which corresponds to the equation
$\y_a \stu \y_b \y_c = \y_a\y_b\y_c + \y_b\y_a\y_c + \y_b\y_c\y_a + \y_{a+b}\y_c+  \y_b\y_{a+c}$:
\begin{multline} \Phi^{(3)}(x_1,x_2,x_3) + \Phi^{(3)}(x_2,x_1,x_3) + \Phi^{(3)}(x_2,x_3,x_1) \\
 = {\Phi^{(2)}(x_2,x_3) -\Phi^{(2)}(x_2,x_1)  \over x_1 -x_3} + {\Phi^{(2)}(x_2,x_3) -\Phi^{(2)}(x_1,x_3)  \over x_1 -x_2}
 \end{multline}
In depth $r$, there are $\lfloor {r\over 2} \rfloor$ equations, involving the components $\Phi^{(r)}, \ldots, \Phi^{(\lceil {r\over 2}\rceil)}$.

\section{Double shuffle equations modulo products}
In order to compare the two sets of relations, consider the continuous $R$-linear map
\begin{eqnarray}
\alpha: R\langle \langle e_0, e_1 \rangle \rangle \To R \langle \langle Y \rangle \rangle 
\end{eqnarray}
which sends all words beginning in $e_0$ to zero, and satisfies
$$\alpha( e_1 e_0^{a_1} \ldots e_1 e_0^{a_r} ) = \y_{a_1+1} \ldots \y_{a_r+1}\ .$$
Note that it respects the $\dd$-degree on words. 
The  \emph{double shuffle equations modulo products} are the  linear equations, for $\Phi \in R \langle \langle e_0, e_1 \rangle \rangle$, 
\begin{eqnarray} \label{dbshfmodprod}
\Phi_{e_0} & = & 0 \\
\Delta_{\sha} \Phi  & =&  1 \otimes \Phi + \Phi \otimes 1 \nonumber \\ 
\Deltat \alpha(\Phi)  & =  & 1 \otimes \alpha(\Phi) + \alpha(\Phi) \otimes 1 \nonumber
\end{eqnarray}
Note that there are no regularization conditions \cite{R} in this setting, precisely because we are working modulo products.
Since $e_1$ and $ \y_1=\alpha(e_1)$ are primitive with respect to $\Delta_{\sha}, \Delta_{\stu}$ respectively, there is a trivial solution to 
$(\ref{dbshfmodprod})$ given by 
$\Phi = \mu  e_1$, for any  $\mu \in R$.
Therefore  the space of solutions to $(\ref{dbshfmodprod})$ is a direct sum of $R$ with the set of solutions to $(\ref{dbshfmodprod})$
which satisfy the extra condition $\Phi_{e_1}=0$.
By propositions \ref{propshuffleequations}  and \ref{propstuffleequations}, a solution to the   equations    $(\ref{dbshfmodprod})$   defines, 
 for each $r\geq 0$, a power series 
$$\overline{\Phi}^{(r)} \in R [[ x_1,\ldots, x_r]]$$
which satisfies the equations $(\ref{shuffleequation})$ and $(\ref{stuffleequation})$.  Conversely,
any set of solutions $f^{(r)}$ to these equations defines a formal power series solution to the equations $(\ref{dbshfmodprod})$  by 
defining $\Phi^{(r)}= f^{(r)}(y_1-y_0,\ldots, y_r-y_0)$,   and applying $(\ref{wordstopolys})$.

\subsection{Constraints in depth one}
The double shuffle equations modulo products reduce to translation invariance in depth one.  Nevertheless, there is a non-trivial constraint in depth one which comes from depth two.

\begin{lem} \label{lemevenindepth1}   Let $\Phi^{(1)}(x_1)\in R[[x_1]]$,  and  $\Phi^{(2)}(x_1,x_2) \in R[[x_1,x_2]]$ be  solutions to the double shuffle equations modulo products in depth two. Then
$\Phi^{(1)}$ is even:
\begin{equation} \Phi^{(1)}(x_1) = \Phi^{(1)} (-x_1)\ .
\end{equation}
\end{lem}
\begin{proof} The power series $\Phi^{(1)}, \Phi^{(2)}$ satisfy the equations
\begin{eqnarray}\label{inproofdepth2equations}
\Phi^{(2)}(x_1,x_{1}+x_2) + \Phi^{(2)}(x_2,x_{1}+x_2) & = &0  \\
\Phi^{(2)}(x_1,x_2) + \Phi^{(2)}(x_2,x_1)  & =  &  {\Phi^{(1)}(x_2) -\Phi^{(1)}(x_1)  \over x_1 -x_2} \nonumber
\end{eqnarray}
The first equation is equivalent to $(1+ \sigma) \Phi^{(2)}=0$, where $\sigma$ is defined by $(\ref{sigmaonx})$. The second is
$(1+ \upsilon) \Phi^{(2)} = I$, where $\upsilon$ is the involution defined in $(\ref{stuffinvol})$, and $I$  is the right-hand side of the second equation of  $(\ref{inproofdepth2equations})$. We have 
$ \upsilon \sigma(x_1,x_2)=(x_2, x_2-x_1)$ and hence  $( \upsilon \sigma )^3 = \iota$ where $\iota(x_1,x_2)=(-x_1, -x_2)$. Therefore
$$(1- \iota) \Phi^{(2)} =  (1+  \upsilon\sigma  + ( \upsilon\sigma )^2)  (1- \upsilon \sigma) \Phi^{(2)} = (1+  \upsilon\sigma  + ( \upsilon\sigma )^2) I\ . $$
Let $\Phi^{(r)}= \Phi^{(r)}_o + \Phi^{(r)}_e$ be the decomposition into odd and even parts with respect to $\iota$, for $r=1,2$.  Then $I_e = (x_1-x_2)^{-1}  (\Phi_o^{(1)}(x_2) -\Phi_o^{(1)}(x_1))$  is invariant under $\iota$, and hence even, and so  
$0=  (1+ \upsilon \sigma + (\upsilon \sigma )^2) I_e$. Using $\Phi^{(1)}_o(x)=-\Phi^{(1)}_o(x)$,  this  becomes
\begin{equation} \label{Phi1constraint}
{\Phi_o^{(1)}(x_1)  - \Phi_o^{(1)}(x_2)  \over  x_1-x_2}  + {\Phi_o^{(1)}(x_1)- \Phi_o^{(1)}(x_1 \! -\! x_2)   \over x_2} + {\Phi_o^{(1)}(x_2)+\Phi_o^{(1)}(x_1\!-\!x_2)\over x_1 }=0
\end{equation} 
Taking the limit as $x_2 \rightarrow 0$, and using $\Phi_o^{(1)}(0)=0$, gives the differential equation
\begin{equation} \label{Phiode} 
 {d\over dx_1} \Phi_o^{(1)}(x_1) =- 2 \, { \Phi_o^{(1)}(x_1) \over x_1} 
 \end{equation} which has no non-trivial solutions in $R[[x_1]]$. Thus $\Phi_o^{(1)}(x_1)=0$.
\end{proof}

\begin{rem} The equation  $(\ref{Phi1constraint})$ has a 
 unique solution (up to multiplication by an element of $R$) in the ring of $R$-Laurent series in one variable, namely
$$\Phi^{(1)}(x_1) = {1 \over x_1}\ .$$
It is a  remarkable fact that this solution can be extended to all higher depths, see \S \ref{sectDepthSplit}. 
\end{rem}

The evenness of $\overline{\Phi}$ in depth $1$ is in fact the only constraint. 

\begin{prop} For every even power series
$f(x_1)\in R[[x_1]]$,  there exists a rational solution  $\Phi$ to the double shuffle equations modulo products  such that   $\overline{\Phi}^{(1)}=f(x_1)$.
\end{prop}
\begin{proof} The trivial  solution  $\Phi= e_1$ has depth one component $\Phi^{(1)}=1$.  For all $n\geq 1$, we defined canonical elements $\sh_{2n+1}$ using the Hoffman-Lyndon basis for motivic multiple zeta values.
The depth one component of $\sh_{2n+1}$  is $x_1^{2n}$. The result follows by linearity and the continuity of $\Delta_{\sha}$ and $\Delta_{\star}$. 
\end{proof}

\subsection{Derivations}
Recall that $R\langle \langle e_0, e_1\rangle \rangle$ is the complete Hopf algebra
whose multiplication is given by concatenation, and whose   comultiplication $\Delta_{\sha}$ is determined by $\Delta_{\sha} e_i = 1\otimes e_i + e_i \otimes 1$ for $i=0,1$.
We call a derivation  a continuous linear map
\begin{eqnarray}
\delta : R\langle \langle e_0, e_1\rangle \rangle & \rightarrow & R\langle \langle e_0, e_1\rangle \rangle\nonumber \\
 \hbox{ such that }  \qquad\delta (w_1 w_2)  & = & \delta(w_1)  w_2 + w_1 \delta(w_2) \nonumber \\ 
 \hbox{ and }  \qquad \qquad\qquad\ \Delta_{\sha} \delta & = & (id \otimes \delta+ \delta \otimes id) \Delta_{\sha}  \ .\nonumber
\end{eqnarray} 
The completion of the free Lie algebra $\Lie_{\!R} \langle \langle e_0, e_1 \rangle \rangle$ (non-standard notation) may be identified with the subspace  of primitive elements in $R \langle \langle e_0, e_1 \rangle \rangle$.
For every $a \in  \Lie_{\!R} \langle \langle e_0, e_1 \rangle \rangle$, we obtain a derivation
\begin{eqnarray}
\delta_a : R\langle \langle e_0, e_1\rangle \rangle & \rightarrow & R\langle \langle e_0, e_1\rangle \rangle \nonumber \\
\delta_a(e_0)  & = & 0  \nonumber \\ 
\delta_a(e_1)  & = & [a ,e_1 ] \ . \nonumber
\end{eqnarray} 
The \emph{Ihara action} is  then defined to be the continuous bilinear map
\begin{eqnarray} 
 \circ : \Lie_{\!R} \langle \langle e_0, e_1 \rangle \rangle \widehat{\otimes}_{R}  R\langle \langle e_0, e_1 \rangle \rangle & \To &  R\langle \langle e_0, e_1 \rangle \rangle  \\
a \otimes w & \mapsto &  \delta_a w + w.a\nonumber
\end{eqnarray}
where the $w.a$ on the right is the concatenation product. These formulae are identical to those given in \cite{DG} except that all words are reversed. 
Note that since we are working on the Lie algebra level, it makes little difference (only up to a sign) whether we choose to view this as a left or right action. 

 Now since the antipode $a \mapsto a^*$, where  $(e_{i_1}\ldots e_{i_r})^* = (-1)^r e_{i_r} \ldots e_{i_1}$, acts by $-1$ on the set of primitive 
elements in $R \langle \langle e_0, e_1 \rangle \rangle$, we can write $[a,e_1] = a e_1 + e_1 a^*$.  This innocuous remark motivates  the following definition \cite{Depth}, \S2.2.

\begin{defn} Define a continuous $R$-bilinear map 
$$ \circb  : R\langle \langle e_0,e_1 \rangle \rangle  \widehat{\otimes}_{R} R\langle \langle e_0,e_1 \rangle \rangle \rightarrow R\langle \langle e_0,e_1 \rangle \rangle\ $$
 as follows. For   any words $a, w$ in $e_0,e_1$, and for any integer  $n\geq 0$,   let 
\begin{equation}\label{circbdef}
 a \circb  (e_0^{n} e_1 w) =        e_0^n  a e_1  w  +   e_0^n e_1 a^* w +e_0^n e_1 (a \circb w)   \end{equation}
where  $a \circb e_0^n =  e_0^n \, a $, and  for any $a_i \in \{e_0,e_1\}$, 
 $(a_1\ldots a_n)^* = (-1)^n a_n \ldots a_1.$
 \end{defn} 
We  call $\circb$ the `linearized Ihara action'. It is immediate to verify that the restriction of $\circb$ to $ \Lie_{\!R} \langle \langle e_0, e_1 \rangle \rangle \widehat{\otimes}_R  R\langle \langle e_0, e_1 \rangle \rangle $
gives back the action $\circ$ (Prop. 2.2 in \cite{Depth}).

\begin{defn} We shall call the following bilinear map the  `Ihara bracket':
\begin{eqnarray} 
\{ , \} : \textstyle{\bigwedge^{\!2} R}\langle \langle e_0,e_1 \rangle \rangle & \To & R\langle \langle e_0,e_1 \rangle \rangle  \\
\{ f, g\}  & = & f \circb g - g \circb f \nonumber
\end{eqnarray}
Its restriction to $\Lie_{\!R} \langle \langle e_0,e_1 \rangle \rangle $ is the usual Ihara bracket.
\end{defn} 
\begin{lem} For any $a,b,c \in R\langle \langle e_0, e_1 \rangle \rangle$, let $A(a,b,c) = a\circb ( b\circb c) - (a\circb b) \circb c$. Then 
\begin{equation} \label{Aabc}
A(a,b,c)= A(b,a,c) 
\end{equation} 
In particular, the Ihara bracket satisfies the Jacobi identity. 
\end{lem} 
\begin{proof} Equation $(\ref{Aabc})$ follows  from the definitions. Alternatively, 
 use  $(\ref{IharaderivShuffle})$ to reduce to the case $c=e_i$, and use $ a \circb e_0 = e_0 a$ and $ a\circb e_1 = a e_1$. 
It defines a pre-Lie structure on $R\langle \langle e_0, e_1 \rangle \rangle$, so its antisymmetrization
is therefore a Lie algebra.
\end{proof}

\subsection{Properties of $\circb$ and Racinet's theorem} \label{sectpropertiesofcircb}
The operator $\circb$  clearly  respects the $\dd$-grading. Therefore,  passing to  power series  defines a map
\begin{eqnarray} \circb: R[[y_0,\ldots, y_r]] \widehat{\otimes}_{R} R[[y_0,\ldots, y_s]]   & \To & R[[y_0,\ldots, y_{r+s}]]   \\
f(y_0,\ldots, y_r) \otimes g(y_0,\ldots, y_s)  & \mapsto &  f\circb g\, ( y_0,\ldots, y_{r+s}) \nonumber
\end{eqnarray}
which can be read off   equation $(\ref{circbdef})$. Explicitly, it  is
\begin{multline} \label{circformula}
f \circb g \, (y_0,\ldots, y_{r+s})  =  \sum_{i=0}^s f(y_i,y_{i+1}, \ldots, y_{i+r}) g(y_0,\ldots, y_i, y_{i+r+1}, \ldots, y_{r+s})  \\  + (-1)^{\deg f + r} \sum_{i=1}^s  f(y_{i+r},\ldots, y_{i+1},y_i) g(y_0,\ldots, y_{i-1}, y_{i+r}, \ldots, y_{r+s}) 
 \end{multline}
In the usual way, we define, for sequences $\Phi_i^{(r)} \in R[[y_0, \ldots, y_r]]$
$$(\Phi_1 \circb \Phi_2)^{(r)}=  \sum_{i +j=r} \Phi^{(i)}_1 \circb \Phi_2^{(j)}$$
There is  also a translation-invariant version of $(\ref{circformula})$
$$ \circb: R[[x_1,\ldots, x_r]]  \widehat{\otimes}_{R} R[[x_1,\ldots, y_s]]    \To  R[[x_1,\ldots, x_{r+s}]]  $$
which we spell out explicitly for later reference:
\begin{multline} \label{circformulax}
f \circb g \, (x_1,\ldots, x_{r+s})  =  \sum_{i=0}^s f(x_{i+1}-x_i, \ldots, x_{i+r}-x_i) g(x_1,\ldots, x_i, x_{i+r+1}, \ldots, x_{r+s})  \nonumber \\  + (-1)^{\deg f + r} \sum_{i=1}^s  f(x_{i+r-1}-x_{i+r},\ldots, 
x_i-x_{i+r}) g(x_1,\ldots, x_{i-1}, x_{i+r}, \ldots, x_{r+s}) 
 \end{multline}
 In the following, elements $f,g,h,\psi$ can be taken in $R\langle \langle e_0, e_1 \rangle \rangle$ or  $R [[y_0,\ldots, y_r]$; 
 or   $R \langle \langle Y \rangle \rangle$ or  $R[[x_1,\ldots, x_r]]$ (when referring to the stuffle product), as appropriate.
\begin{prop} The linearized Ihara action satisfies the identities
\begin{eqnarray}
f \circb (g \cdot h) & = & (f \circb g) \cdot h + g \cdot (f \circb h) - g \cdot f \cdot h  \label{IharaderivShuffle}   \\
f \circb (g \studot h) & = & (f \circb g) \studot h + g \studot (f \circb h)  - g \studot f \studot h    \label{IharaderivStuffle}  
\end{eqnarray} 
where $\cdot$ denotes the shuffle product, and $\studot$ denotes the stuffle product.
\end{prop}
\begin{proof} The first equation is almost immediate from the recursive definition    $(\ref{circbdef})$. 
We can view  $R \langle \langle Y \rangle \rangle$ as a subspace of  $R \langle \langle e_0, e_1 \rangle \rangle$
via the map  $\y_n \mapsto e_1 e_0^{n-1}$.  The second equation follows from the first since the stuffle concatenation product
is simply the  restriction of the shuffle concatenation  to the subspace  $R \langle \langle Y \rangle \rangle$.
 \end{proof}

If $\Delta$ is any coproduct, let  $\Delta^r = \Delta - 1 \otimes id - id \otimes 1$ be the  reduced coproduct.
\begin{lem} If $f,g $ are  solutions to the shuffle equations modulo products,
\begin{equation} \label{Deltarsha} 
\Delta^r_{\sha} (f \circb g) = f \otimes  g + g \otimes f\ .
\end{equation}
\end{lem}
\begin{proof} If $\Delta^r_{\sha} f=\Delta^r_{\sha} g=0$, then  in particular $f$ is a Lie element. Therefore, $\Delta^r_{\sha} (f\circb g) =  \Delta^r_{\sha} (f\circ g) =   \Delta^r_{\sha} (\delta_f g + f\cdot g)$. Since
$\delta_f$ is a derivation, $\delta_f g$ is primitive by definition, and     $     \Delta^r_{\sha} (\delta_f g)= 0$. Therefore 
$\Delta_{\sha}^r( f \circb g) =  \Delta_{\sha}^r (f \cdot g)= f \otimes g + g \otimes f$, since $f,g$ are primitive.  
 \end{proof}

The following theorem is a consequence of Racinet's thesis.
\begin{thm} \label{thmRacinet} \cite{R} Let $\psi$ be a solution to the double shuffle equations modulo products.
Then if $f$ is a solution to the stuffle equations modulo products,
$$\Delta^r_{\stu} (\psi \circb f)  = \psi \,\otimes f + f \,\otimes \psi\ .$$
In particular, the set of solutions to the double shuffle equations modulo products  $(\ref{dbshfmodprod})$ forms a Lie algebra with respect to the Ihara bracket $\{ , \}$.
\end{thm}

\section{Linearized double shuffle equations} \label{sectLinDS}

Retaining only the leading depth part of the double shuffle equations modulo products leads to a simpler set of  equations called the linearized double shuffle
equations \cite{IKZ}. The Ihara bracket acquires a remarkable dihedral symmetry after grading for the depth filtration, as first noticed by Goncharov \cite{GG}.

\subsection{Linearized equations}
The shuffle equations are graded for the depth, but the stuffle equations are not.
By inspection of the coproduct $\Deltat$, we see that:
\begin{eqnarray}
\gr_{\dd}\Deltat: \gr_{\dd} R\langle \langle Y \rangle \rangle  & \To &   \gr_{\dd} R\langle \langle Y \rangle \rangle\widehat{\otimes}_R  \, \gr_{\dd} R\langle \langle Y \rangle \rangle\\
\gr_{\dd}  \Deltat \y_n & =  & 1 \otimes \y_n + \y_n \otimes 1   \nonumber 
\end{eqnarray}
Thus the $\y_n$ are primitive for $\gr_{\dd}\Deltat$, and $\gr_{\dd} R\langle \langle Y \rangle \rangle$ is just  the completed shuffle Hopf algebra on $Y$. Another way to see this is 
on the dual Hopf algebra $\Q \langle Y \rangle$. Here, the depth filtration is the increasing filtration for which $\y_n$ has degree $1$.
By inspection of $(\ref{stuffprod})$, the term beginning with $\y_{i+j}$ is of smaller depth and  so the associated graded product $\star$ reduces to the shuffle product $\sha$ on the alphabet $Y$:
\begin{equation} 
\gr^{\dd} ( \Q \langle Y \rangle, \stu) \cong (\Q \langle Y \rangle, \sha)\ .
\end{equation} 

\begin{defn}  Let $\Phi \in R\langle \langle Y \rangle \rangle $. It satisfies the \emph{linearized stuffle equations} if 
  $$\Delta^{Y}_{\sha} \Phi = 1 \otimes \Phi + \Phi \otimes 1$$
  where $\Delta^{Y}_{\sha}$ is the continuous  coproduct for  which the  $\y_n$ are primitive.
\end{defn}
Power series versions of these equations can be obtained from proposition \ref{propstuffleequations} by suppressing all terms of lower depth. The $(p,q)^\mathrm{th}$ relation is 
\begin{equation} \label{linstuffpqequation} 
f(\x_1 \ldots \x_p \sha \x_{p+1}\ldots \x_{p+q})=0\ ,\end{equation}
which is exactly equation $(\ref{shuffleequation})$ without the $\sharp$.

\begin{defn}  Let $\Phi \in R\langle \langle Y \rangle \rangle $. We say that $\Phi$ satisfies the \emph{linearized double shuffle equations} if the following equations hold:
 \begin{eqnarray} \label{lindoubleshuffle}
\Phi_{e_0} & = & 0  \\ 
 \Delta_{\sha} \Phi  & =  &1 \otimes \Phi + \Phi \otimes 1   \nonumber \\
 \Delta^{Y}_{\sha} \alpha(\Phi) & = &  1 \otimes \alpha(\Phi) + \alpha(\Phi) \otimes 1\nonumber \\
\Phi_{e_0^i e_1} & = &   0\quad  \hbox{ if } i \hbox{ is odd or zero}.  \nonumber
\end{eqnarray}
\end{defn}
The last equation in  $(\ref{lindoubleshuffle})$ requires some explanation. We know from lemma \ref{lemevenindepth1} that any solution to the double shuffle equations modulo products
is even in depth $1$. This information, which is equivalent to the last line of $(\ref{lindoubleshuffle})$,  is lost after passing to the linearized stuffle equations, and  must added back.

\begin{defn} Let $\ls_r$ denote the graded $\Q$-vector space of solutions to the linearized double shuffle equations in depth $r$ and weight $>1$.
Let $\ls =\bigoplus_{r\geq 1} \ls_r$.
\end{defn}

A corollary of  the proof of theorem \ref{thmRacinet} is that $\ls$ is a Lie algebra \cite{Depth}.
\begin{thm} \label{thmlsisIharastable} The  vector space $\ls$ of solutions to the linearized double shuffle equations is a bigraded Lie algebra with respect to the linearized Ihara bracket $\{,  \}$.
\end{thm} 
An element of $\ls$ of weight $N$ and depth $r$ can be viewed, in the usual way, as a homogeneous polynomial in $\Q[x_1,\ldots, x_r]$ of degree $N-r$.

We shall think of the linearized equations as the homogeneous version of  the double shuffle  equations modulo products. For example, in depth $2$ they are
\begin{eqnarray}
 f^\sharp(\x_1 \sha \x_2)=0,  &  \hbox{ that is, } &  f(x_1,x_1+x_2) + f(x_2, x_1+x_2)  =  0\, ,\\
\hbox{ and } \quad  f(\x_1 \sha \x_2)=0,    & \hbox{ that is, } &   f(x_1,x_2) + f(x_2,x_1)  = 0 \ .\nonumber  
\end{eqnarray} 
The second equation is obtained from $(\ref{exdepth2stuffle})$  by dropping terms on  the right-hand side.

\subsection{Dihedral symmetries} \label{sectDihedralsymm}
The linearized equations admit a dihedral symmetry coming from the antipodal symmetries on the two Hopf algebra structures.
Since the  stuffle algebra, graded for the depth filtration, is isomorphic to the shuffle algebra on $Y$, its antipode
is  the map $\y_{i_1}\ldots \y_{i_r} \mapsto (-1)^r\y_{i_r}\ldots \y_{i_1}$. This defines an involution
\begin{eqnarray}
\overline{\tau}: R[[x_1,\ldots, x_r]] & \To & R[[x_1,\ldots, x_r]]  \\
\overline{\tau}(\overline{f}) (x_1,\ldots, x_r) & = & (-1)^r \overline{f}(x_r,\ldots, x_1)  \nonumber
\end{eqnarray}
  It follows that if  $\overline{f}\in R[[x_1,\ldots, x_r]]$ satisfies the linearized stuffle relations in depth $r$, then
$\overline{f} + \overline{\tau}(\overline{f})=0$.
Note that  the involution $\overline{\tau}$ lifts to an involution
 \begin{eqnarray}
 \tau: R[[y_0,y_1,\ldots, y_r]] & \To & R[[y_0, y_1,\ldots, y_r]]  \\
\tau(f) (y_0,y_1,\ldots, y_r) & = & (-1)^r f(y_0,y_r,\ldots, y_1)  \nonumber
\end{eqnarray}
and therefore if  $f\in R[[y_0,\ldots, y_r]]$  satisfies both translational invariance  and the linearized stuffle relations, we have
$f+ \tau(f)=0$.
Recall that $\sigma$ is the involution  $(\ref{sigmaantipodedef})$  which preserves the solutions to the  shuffle equations.
Therefore if $f$ is a solution to the linearized double shuffle equations, it satisfies
\begin{equation}
f+ \sigma(f)= 0 \quad \hbox{ and } \quad  f + \tau(f) = 0 \  .
\end{equation} 
The composition $\tau \sigma$ is a signed cyclic rotation of order $r+1$
\begin{equation} \label{tausigmaformula}
\tau \sigma(f) (y_0,\ldots, y_r) =   f(-y_r,-y_0,\ldots, -y_{r-1})
\end{equation}
and $\sigma,\tau$ generate a dihedral group $D_{r+1}=\langle \sigma, \tau \rangle $ of symmetries of  order $2r+2$.

\begin{lem} \label{lemfisdihed} Let $f$ be a  solution to the linearized double shuffle equations in depth $r$. Then $\Q f$ is isomorphic to the sign representation $\varepsilon$ of the dihedral group $D_{r+1}$.
\end{lem}

\subsection{Parity constraints} The double shuffle equations modulo products are a torsor over the linearized double shuffle equations in the following sense: given 
two solutions $\Phi_1, \Phi_2$ to the double shuffle equations modulo products such that $\Phi_1^{(r)}= \Phi_2^{(r)}$ for $1 \leq r <N$, the difference
$
\Phi^{(N)}_1- \Phi_2^{(N)}
$
is an element of $\ls_N$.

The following proposition was first proved by Tsumura \cite{Tsumura}. A proof is given in \cite{Depth}.

\begin{prop} \label{propparity} Let $f\in \Q[x_1,\ldots, x_r]$ be a homogeneous polynomial of degree $d$, which is a solution to the linearized double shuffle equations.
If $d$ is odd, then $f=0$. \end{prop}

A stronger statement which implies this proposition  is proved in \S\ref{sectproofofalphaid}.

\subsection{The dihedral Lie algebras and the Ihara bracket} \label{sectdihedihara}
For all $r\geq 1$, consider the  graded vector space $\PP_r$ of polynomials  $f\in \Q[y_0,\ldots, y_r]$ which satisfy 
\begin{eqnarray} \label{goodcyclicpolys}
f(y_0,\ldots, y_r) & =  &f(-y_0,\ldots,- y_r) \\ 
f + \sigma(f)  &= & f+   \tau (f) \quad  = \quad 0 \ .   \nonumber
\end{eqnarray}
In particular, they  invariant under cyclic permutations of $(y_0,y_1, \ldots, y_n)$.   Recall that $D_{r+1}$ is the dihedral symmetry group generated by $\sigma, \tau$.
\begin{prop} \label{propdihedbracket} Suppose that $f\in \PP_r$ and $g\in \PP_s$ are polynomials  satisfing $(\ref{goodcyclicpolys})$.  Then
the Ihara bracket is given by  summing over the dihedral symmetry group:
\begin{eqnarray} \label{dihedbracket}
\{f,g\} = \sum_{\mu  \in D_{r+s+1}} \varepsilon(\mu) \mu\big( f(y_0,y_1,\ldots, y_r) g(y_{r},y_{r+1}\ldots, y_{r+s})\big)
\end{eqnarray}
In particular,  $\{ . , .\}: \PP_r \times \PP_s \rightarrow \PP_{r+s}$, and $\PP=\bigoplus_{r\geq 1} \PP_r $ is a bigraded Lie algebra.
\end{prop}
\begin{proof}
A straightforward calculation from $(\ref{circformula})$   gives
$$\{f,g\} = \sum_{i }  f(y_i, y_{i+1}, \ldots, y_{i+r}) \big( g(y_{i+r}, y_{i+r+1}, \ldots,y_{i-1})- g(y_{i+r+1}, y_{i+r+2}, \ldots,y_i)    \big)$$
where the summation indices are taken modulo $r+s+1$. \end{proof}
The terms in the Ihara bracket can be represented pictorially as  polygons with  labels which are permuted cyclically, as follows (or  equivalently as marked points on a circle \cite{GG}).
The signs are fixed as shown: there is no dependence on $r$ or $s$.

\begin{figure}[h!]
  \begin{center}
   \epsfxsize=10cm \epsfbox{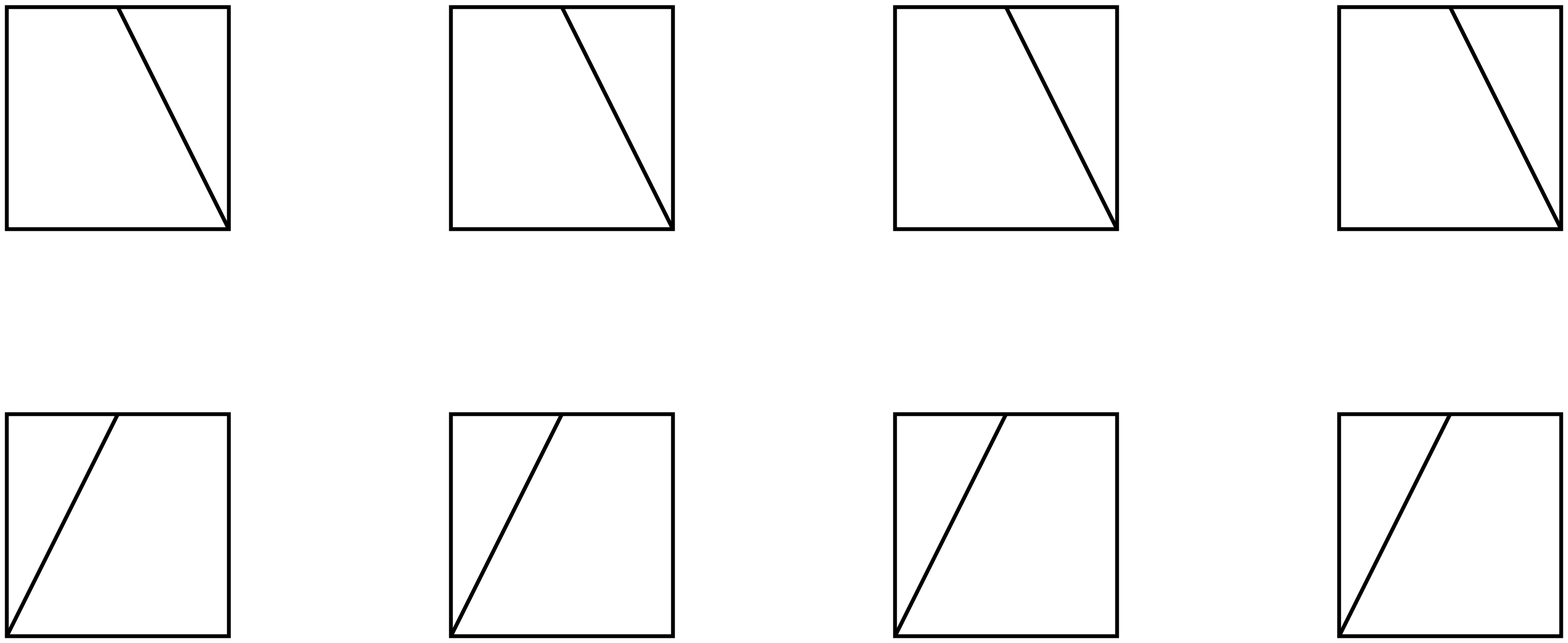}
  \label{Boxes}
\put(-310,95){$+$}
\put(-298,95){$y_3$} \put(-275,90){$f$}\put(-244,95){$y_1$}\put(-257,103){$g$}\
\put(-270,120){$y_0$} \put(-270,67){$y_2$}
\put(-229,95){$+$}
\put(-217,95){$y_0$} \put(-194,90){$f$}\put(-163,95){$y_2$}\put(-176,103){$g$}\
\put(-189,120){$y_1$} \put(-189,67){$y_3$}
\put(-149,95){$+$}
\put(-137,95){$y_1$} \put(-114,90){$f$}\put(-83,95){$y_3$}\put(-96,103){$g$}\
\put(-109,120){$y_2$} \put(-109,67){$y_0$}
\put(-68,95){$+$}
\put(-56,95){$y_2$} \put(-33,90){$f$}\put(-2,95){$y_0$}\put(-15,103){$g$}\
\put(-28,120){$y_3$} \put(-28,67){$y_1$}
\put(-310,20){$-$}
\put(-298,20){$y_3$} \put(-265,15){$f$}\put(-244,20){$y_1$}\put(-283,28){$g$}\
\put(-270,45){$y_0$} \put(-270,-7){$y_2$}
\put(-229,20){$-$}
\put(-217,20){$y_0$} \put(-184,15){$f$}\put(-163,20){$y_2$}\put(-202,28){$g$}\
\put(-189,45){$y_1$} \put(-189,-7){$y_3$}
\put(-149,20){$-$}
\put(-137,20){$y_1$} \put(-104,15){$f$}\put(-83,20){$y_3$}\put(-122,28){$g$}\
\put(-109,45){$y_2$} \put(-109,-7){$y_0$}
\put(-68,20){$-$}
\put(-56,20){$y_2$} \put(-23,15){$f$}\put(-2,20){$y_0$}\put(-40,28){$g$}\
\put(-28,45){$y_3$} \put(-28,-7){$y_1$}
\put(-275,98){$\curvearrowright$}\put(-265,23){$\curvearrowright$}
 \caption{The terms in $\{f,g\}$. The top left diagram corresponds to 
 $g(y_0, y_1) f(y_2,y_3,y_0)$, the bottom right one to $- g(y_2,y_3)f(y_3,y_0,y_1)$.}
  \end{center} 
\end{figure}
A polygon with  sides labelled by $\{y_{i_1}, \ldots, y_{i_n}\}$, and an inscribed $f$, for example,  denotes a term
$f(y_{i_1}, \ldots, y_{i_n})= f(y_{i_2}, \ldots, y_{i_n},y_{i_1})=\cdots= f(y_{i_n}, y_{i_1} \ldots, y_{i_{n-1}})$.

\begin{defn}
Let $\overline{\PP}_r \subset \PP_r $ denote the subspace of polynomials  which satisfy $(\ref{goodcyclicpolys})$ and are invariant under translation, and write $\overline{\PP}= \bigoplus_{r \geq 1} \overline{\PP}_r$. 
\end{defn}

\begin{lem} \label{lemtransinvisstable}  The space $\overline{\PP}$ is a bigraded Lie subalgebra of $\PP$.
\end{lem}
\begin{proof} It is clear from $(\ref{dihedbracket})$ that translation invariance is preserved by $\{, \}$.
\end{proof}
We can equivalently view $\overline{\PP}_r$ as the space of polynomials  $\overline{f}\in \Q[x_1,\ldots, x_r]$ such that
  $ \overline{f}(y_1 -y_0,\ldots, y_r-y_0)\in\PP_r$.  Explicitly,  $\overline{\PP}_r$ is the vector space of polynomials satisfying
 \begin{enumerate}
 \item  $\overline{f}(x_1,\ldots, x_r )=  \overline{f}(-x_1,\ldots, -x_r ) $
\item  $\overline{f}(x_1,\ldots, x_r ) + (-1)^r \overline{f}(x_r,\ldots, x_1 ) =0 $
\item $\overline{f}(x_1,\ldots, x_r ) + (-1)^r \overline{f}(x_{r-1}-x_r,\ldots, x_1-x_r, -x_r ) =0 $
\end{enumerate}
with the Lie bracket  induced by $(\ref{dihedbracket})$. 

\begin{thm}  \label{thmlsisclosed} The space of solutions to  the linearized double shuffle  equations $\ls$ are a bigraded Lie subalgebra of 
$\PP$.
\end{thm}
\begin{proof} This follows from propositions  \ref{propparity} and \ref{propdihedbracket}, and theorem   \ref{thmlsisIharastable}.
\end{proof} 
The dihedral symmetries on the depth-graded motivic Lie algebra was first discovered by Goncharov \cite{GG}. In \S\ref{sectEisder} we shall give an explanation of the cyclic symmetry in terms of the pro-unipotent fundamental group of the punctured infinitesimal Tate curve.

\section{Period polynomials and the Broadhurst-Kreimer conjecture}
We briefly review the results from the paper \cite{Depth}, which, if    conjectures due to Broadhurst-Kreimer and Zagier are correct, will  give a complete description of the  solutions to the linearized double shuffle equations.

\subsection{Reminders on period polynomials}
Let  $S_{2k}$  denote the space of   cusp forms of weight $2k$  for the full modular group $SL_2(\Z)$.

\begin{defn} \label{defperiodpoly} Let $n\geq 1$ and let $W_{2n}\subset \Q[x,y]$ denote the vector space of homogeneous polynomials $P(x,y)$  of degree  $2n-2$ satisfying
\begin{eqnarray}
&P(x,y) + P(y,x)=  0& \label{periodsign} \\ 
&P(x, y) + P(x-y,x) + P (-y, x-y) =0 \ . &  \label{periodrel}
\end{eqnarray}
\end{defn}
Let $W^e_{2n}\subset W_{2n}$ (respectively $\Ss^{\sfo}_{2n} \subset W_{2n}$) denote the subspace of polynomials which are even (respectively odd)  in  $x$ and $y$.
The  space $W_{2n}^e$ contains the polynomial $p_{2n}= x^{2n-2}-y^{2n-2}$, and is  a direct sum
$$W_{2n}^e \cong  \Ss^{\sfe}_{2n}  \oplus \Q \,p_{2n} $$
where $\Ss^{\sfe}_{2n}$ is the subspace of polynomials which vanish at $(x,y)=(1,0)$.  We write $\Sf= \bigoplus_{n} \Ss^{\sfe}_{2n}$,  and $\So= \bigoplus_{n} \Ss^{\sfo}_{2n}$. It follows from  the Eichler-Shimura theorem that $$\dim S_{2k} = \dim \Ss^{\sfe}_{2k}= \dim \Ss^{\sfo}_{2k}$$ and   the generating series for the dimension of the space of cusp forms is
\begin{equation} \label{cuspgf} \mathbb{S}(s)=\sum_{n\geq 0} \dim \Ss^{\bullet}_{2n}\, s^{2n} = {s^{12} \over (1-s^4)(1-s^6)}  \ . \end{equation}
A  generator of $\Ss^{\sfe}$ of lowest degree is 
$s_{12} = x^2y^2(x^2-y^2)^3  = x^8 y^2 -3 x^6 y^4+3 x^4 y^6-x^2y^8.$

\subsection{Generators in depth 1: Zeta elements} 
The  solutions to the linearized double shuffle equations in depth $1$  are given by $x_1^{2n}$ for all $n\geq 0$.
The case $n=0$ corresponds to the trivial solution $\Phi= e_1$. 
If $\sigma_{2n+1}$ denotes any choice of normalized generators of the motivic Lie algebra, we have
$$\sigma^{(1)}_{2n+1} = x_1^{2n}  \quad \hbox{ for } n \geq 1\ .$$
It follows that  $\gr^{1}_{\dd} \g  \cong \ls_1  $ (since by definition, $\ls$ is concentrated in weight $>1$).
\subsection{Quadratic relations in depth 2} The linearized Ihara bracket gives a map
\begin{equation}
\{ . , . \} : \ls_1 \wedge \ls_1 \To \ls_2\ .
\end{equation}
It follows almost immediately from the formula for the linearized Ihara bracket that 
the kernel is isomorphic to the space of even period polynomials, hence an exact sequence
$$ 0 \rightarrow \Sf \rightarrow \ls_1 \wedge \ls_1 \rightarrow \ls_2$$
The  element $s_{12}$ corresponds to the relation, due to Ihara:
\begin{equation} \label{IharaRel} 
\{ x_1^2, x_1^8\} - 3 \{ x_1^4, x_1^6\} =0
\end{equation}
\subsection{Generators in depth 4: exceptional elements} \label{sectExotic}
Any $f(x,y)\in \Sf$   vanishes along $x=0,y=0$ and $x=y$. Writing  $f_0= (xy(x-y))^{-1} f$, and $f_1= (xy)^{-1} f$, we 
defined an element
$\e_f \in  \Q[y_0,y_1,y_2,y_3,y_4] $ by the formula
$$
 \e_f  =   \sum_{\Z/ \Z 5}   f_1(y_4-y_3,y_2-y_1) + (y_0-y_1) f_0 (y_2-y_3,y_4-y_3)
 $$
 where the sum is over cyclic permutations
$(y_0,y_1,y_2,y_3, y_4 ) \mapsto (y_1,y_2,y_3, y_4, y_0)$.
Its reduction 
$\overline{\e}_f\in \Q[x_1,\ldots, x_4]$ is obtained by setting $y_0=0, y_i =x_i$, for $i=1,\ldots, 4$.

It was shown in \cite{Depth} that the elements $\overline{\e}_f$ are solutions to the linearized double shuffle equations in depth 4, i.e., for every $f \in \Sf$ we have exceptional generators
 $\overline{\e}_f \in \ls_4.$

\subsection{Quantitative Broadhurst-Kreimer-Zagier conjecture} \label{sectQBK}
The strongest conjecture that one can make is the following.

 \begin{conj}  \label{conjBKZls} (Strong Broadhurst-Kreimer and Zagier conjecture)
\begin{eqnarray} \label{homologyconj}
H_1(\ls, \Q) & \cong & \ls_1 \oplus \e(\Sf)  \\
H_2(\ls, \Q) & \cong & \Sf \nonumber \\
H_i(\ls, \Q) & =  & 0  \quad \hbox{for all} \quad  i \geq 3\ .\nonumber 
\end{eqnarray}
\end{conj}
Some variations and consequences of this conjecture are explored in \cite{Depth}. In particular, it implies that the solutions to the linearized double shuffle equations are generated by the elements
$x_1^{2n}$, for $n\geq 1$, and $\overline{\e}_f$, as $f$ ranges over a basis of $\Sf$.

There is an analogous conjecture for the depth-graded motivic Lie algebra $\gd$: first, that $\overline{\e}( \Sf) \subset \gd_4$ and that
\begin{eqnarray} \label{homologyconj}
H_1(\gd, \Q) & \cong & H_1(\gm;\Q)  \oplus \e(\Sf)  \\
H_2(\gd, \Q) & \cong & \Sf \nonumber \\
H_i(\gd, \Q) & =  & 0  \quad \hbox{for all} \quad  i \geq 3\ .\nonumber 
\end{eqnarray}
It is well-known that if  the homology of a pronilpotent Lie algebra vanishes in degree $i$ then it automatically vanishes in all higher degrees 
\cite{Sigma} remark 18.6, and 
\cite{EnriquezLochak}.
\newpage

\begin{center} \bf{II. Double shuffle equations with poles}
\end{center} 

\section{Polar solutions to the double shuffle equations}

We consider solutions to the double shuffle equations modulo products, and various other Lie algebras,  in the space of rational power series with poles.

\subsection{Rational power series} \label{sectLiealg}
Consider the complement of the set  of hyperplanes $\{x_i = x_j: i<j \}$ in $(\A^1 \backslash \{0\})^d$. Denote its ring of regular functions over $\Q$ by 
$$\Or_d = \Q\big[x_1,\ldots, x_d , {1 \over x_1}, \ldots, {1\over x_d},\Big( {1\over x_i - x_j}\Big)_{1\leq i < j \leq d}\big]\ .$$
There is an analogous ring of  formal power series  denoted by 
$$\widehat{\Or}_d = \Q\Big[ {1 \over x_1}, \ldots, {1\over x_d},\Big( {1\over x_i - x_j}\Big)_{1\leq i < j \leq d}\Big] [[ x_1, \ldots, x_d]]\ .$$
which  is required for the construction of associators, but   will not be used in the present paper. 
Taking the product over all $d$, we define
$$\Or = \prod_{d \geq 1} \Or_d\ .$$
As before, the component of depth $d$ of an  element $f\in \Or$ will be denoted by $f^{(d)} \in \Or_d$. In \S \ref{sectpropertiesofcircb}, we defined a linearized version of the Ihara action on power series. The same formula yields a bilinear map:
\begin{eqnarray} \circb:  \Or_r \otimes_{\Q} \Or_s   & \To &\Or_{r+s}   \\
f(x_1,\ldots, x_r) \otimes g(x_1,\ldots, x_s)  & \mapsto &  f\circb g\, ( x_1,\ldots, x_{r+s}) \nonumber
\end{eqnarray}
If $f=(f^{(d)})_{d\geq 1}$ and $g=(g^{(d)})_{d\geq 1}$ in $\Or$, we have
 $$(f\circb g)^{(d)} =\sum_{i+j= d} f^{(i)} \circb g^{(j)}\ .$$ 
Antisymmetrizing, we obtain  the linearized  Ihara  bracket:
\begin{eqnarray} \label{polariharadef}
\Or   \, \otimes_{\Q} \, \Or &  \To &  \Or   \\
\{ f, g\} &  =  &f \circb g - g \circb f \nonumber \ .
\end{eqnarray}
In the same vein, the formulae for shuffle concatenation $\cdot$ and stuffle concatenation $\studot$ 
give  multiplication laws
$\Or   \otimes_{\Q}  \Or \rightarrow  \Or$.
All the properties of $\circb$ described in \S\ref{sectpropertiesofcircb} remain valid in this setting. Since the  shuffle    equations $(\ref{shuffleequation})$ and stuffle  equations  $(\ref{stuffleequation})$ modulo products are formally defined over $\Or$, it makes sense to consider
solutions to the double shuffle equations modulo products in $\Or$. 

\begin{defn} Let $\p\dmr \subset \Or$ denote the set of solutions to the double shuffle equations modulo products which have  poles of the form described above.
\end{defn}

Since theorem \ref{thmRacinet}  is a statement about functions satisfying
certain functional equations, we immediately deduce (e.g., by the argument in \S \ref{trivialdensity}).

\begin{thm} \label{thmPolarRacinet}  $\p \dmr$ is a Lie algebra with respect to the Ihara bracket $(\ref{polariharadef})$. 
\end{thm}

We shall say that an element $f\in \Or$ is homogeneous of weight $w$ if for all $d\geq 1$,  $f^{(d)} \in \Or_d$ is homogeneous of degree $w-d$. If $f, g$ are homogeneous of weights $w_1,w_2$ respectively,
then $f\circb g $ is homogeneous of weight $w_1+w_2$.

There is an  obvious analogue of the  rings $\Or$ in  variables $y_0,\ldots, y_n$, in which we can consider solutions
to the   `unreduced' double shuffle equations modulo products. 

\subsection{Polar version of the dihedral Lie algebra} \label{sectPolardihedradLie}
We  need to consider a version of the dihedral Lie algebra $\p$  with poles
along  $y_i=y_j$ for \emph{consecutive} $i,j$ only.

\begin{defn} \label{defncr} For all $r \geq 1$, consider the  polynomial
\begin{equation} c_r = - \prod_{i=0}^r (y_i-y_{i+1})
\end{equation}
where the indices are taken modulo $r+1$.  Thus $c_1 = (y_1-y_0)^2$. 
Define $\pp\pp_r$
to be the  vector space of homogeneous rational functions $f$ in $y_0,\ldots, y_1$ such that
$$ c_r f \in \Q[y_0,\ldots, y_r]\ ,$$
which are even and  satisfy the dihedral symmetry properties $(\ref{goodcyclicpolys})$, i.e., 
\begin{eqnarray} \label{dihedsymmetries}
f(y_0,\ldots, y_r) & = & f(-y_0, \ldots, -y_r) \\
f(y_0,\ldots, y_r)  & = & f(y_1,y_2, \ldots, y_r,y_0 ) \nonumber \\
 f(y_0,\ldots, y_r) &  = &  (-1)^{r+1} f(y_r,\ldots, y_0) \ . \nonumber 
\end{eqnarray}
Then $\p\p_r$ is graded by weight, where the weight of a homogeneous rational function  $f$ is defined to be  $\deg f + r$. Define a bigraded vector space  
$\p\p= \bigoplus_{r \geq 1} \p\p_r.$ 
\end{defn}

Elements of $\pp\pp_r$ can be represented by  polygons in an identical manner to  \S\ref{sectdihedihara}.

\begin{lem} \label{lemppisclosed} $\pp\pp$ is a bigraded Lie algebra with respect to the bracket $\{\,\, ,\,\, \}$.
\end{lem}

\begin{proof} 
The  element $c_1^{-1} \in \p\p_1$  has a double pole, so we first check that 
$\{ c_1^{-1}, f\}$ has no double poles for all $f\in \p\p_r$.  There are only two terms in the formula for $\{c_1^{-1}, f\}$ (see figure 1) 
which could  potentially give rise to a double pole along $y_0=y_1$, namely
$$ (y_1-y_0)^{-2} \big( f(y_0,y_2,\ldots, y_r) -  f(y_1,y_2,\ldots,  y_r))\ ,$$
and the double pole clearly cancels.  By cyclic symmetry $\{c_1^{-1}, f\}$ has no double poles along  $y_i =y_j$
for any consecutive $i,j$.  Now let $f \in \p\p_r, g\in \p\p_s$.  It suffices to show that  $\{f,g\}$ has no  poles
along $y_k=y_0$ for all $k=2,\ldots, r+s-2$. By  the cyclic symmetry of $\{f,g\}$ it will then follow that $\{f,g\}$ only has 
simple poles along the $D_{r+s+1}$ orbit of $y_1=y_0$.  Since $f$ and $g$ only have poles along  $y_i=y_j$ for consecutive $i,j$, there are at most 
four cuts that could potentially give a pole along $y_k=y_0$ (figure 2).
\begin{figure}[h!]
\vspace{-0.2in}  \begin{center}
   \epsfxsize=9cm \epsfbox{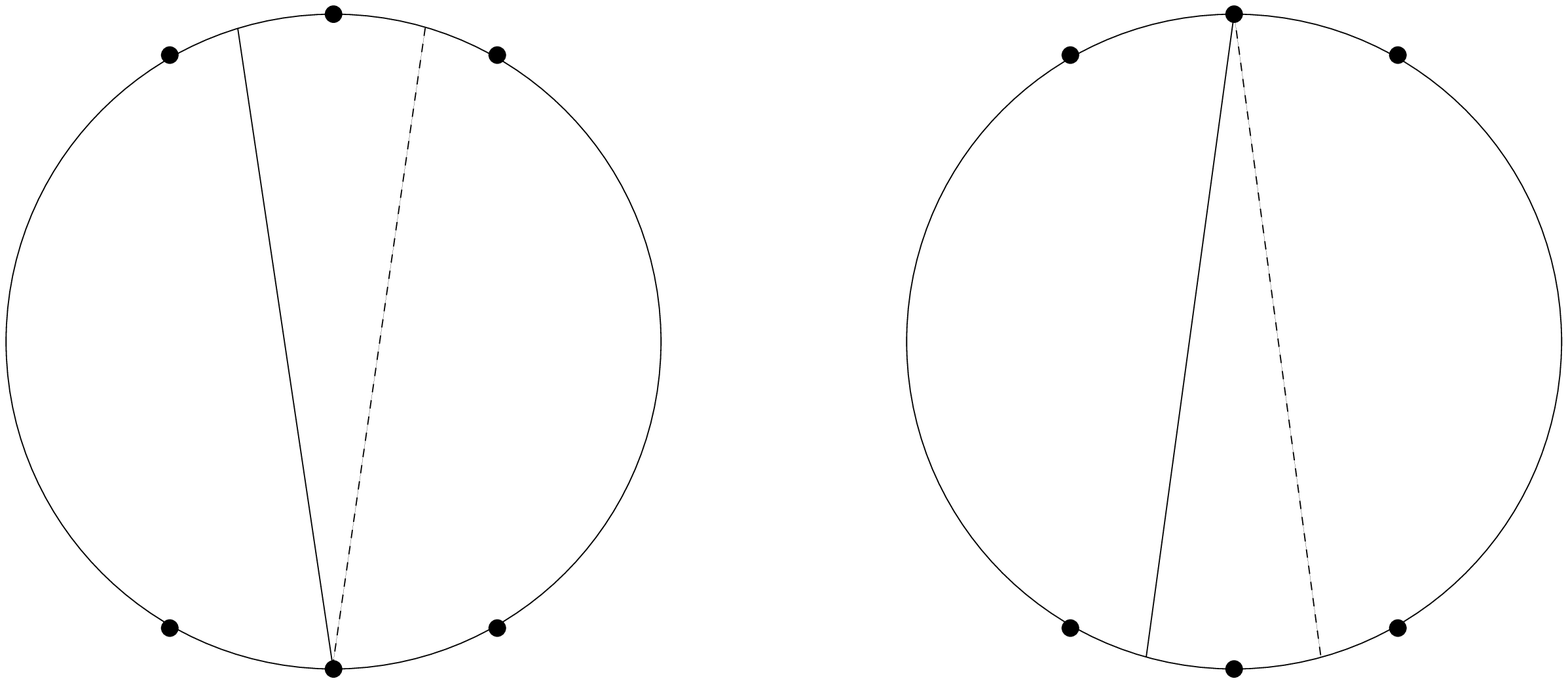}
  \label{TwoCircles}
\put(-210,115){$y_0$} \put(-180,110){$y_1$}  \put(-243,110){$y_{r+s}$} 
\put(-210,-7){$y_k$} \put(-180,-2){$y_{k-1}$}  \put(-240,-2){$y_{k+1}$} 
\put(-210,60){$\curvearrowright$}
\put(-60,115){$y_0$} \put(-30,110){$y_1$}  \put(-93,110){$y_{r+s}$} 
\put(-60,-7){$y_k$} \put(-30,-2){$y_{k-1}$}  \put(-90,-2){$y_{k+1}$} 
\put(-62,60){$\curvearrowright$}
\caption{(Proof of lemma  \ref{lemppisclosed}). Two pairs of terms in the linearized Ihara bracket with non-trivial poles along $y_k=y_0$.}
  \end{center}
\end{figure}
The  pair of diagrams  with solid lines can only arise when $r=k$, the pair with dashed lines  only when $s=k$. The  total contribution from the  former  pair (solid lines) is
$$  f(y_0,y_1, \ldots, y_k) \big( g(y_k, \ldots, y_{r+s}) - g(y_{k+1}, \ldots, y_{r+s},y_0)\big) $$
and so any pole at $y_k=y_0$ cancels by cyclic symmetry of $g$.   The same is true for the second pair.
\end{proof}

\begin{defn}  Let $\overline{\p\p}_r$  be the subspace of $\p\p_r$ consisting of those elements which are  translation invariant   $(\ref{transinv})$ in addition to $(\ref{dihedsymmetries})$.
\end{defn}
It follows from lemma \ref{lemtransinv} that $\overline{\p\p}= \bigoplus_{r\geq 1} \overline{\p\p}_r$ is a bigraded Lie subalgebra of $\p\p$. 
In the usual manner, we can view the depth $r$ part  $\overline{\p\p}_r$ as elements of
$$\overline{c_r}^{-1} \Q [x_1,\ldots, x_r]\ , \quad  \hbox{ where } \quad  \overline{c}_r = x_1(x_2-x_1)\ldots (x_r-x_{r-1})x_r\ . $$

\subsection{Linearized double shuffle equations with poles} 
We shall mainly consider solutions to the linearized double shuffle equations with a restricted set of  poles.

\begin{defn} Let $\p\ls_r$ denote the graded vector space of rational functions 
$$f \in {c_r}^{-1}\Q[y_0,\ldots, y_r]$$
for $r\geq 1$
which are  homogeneous solutions to the linearized double shuffle equations. We shall write
$\p \ls = \bigoplus_{r\geq 1} \p\ls_r$ and denote the reduced version by 
$\overline{\p \ls} = \bigoplus_{r\geq 1} \overline{\p\ls}_r$.
\end{defn} 

\begin{thm} \label{thmplsclosed} The space $\p \ls$ is a bigraded Lie subalgebra of $\overline{\p \p}$.
\end{thm}

\begin{proof}
 This follows from lemma \ref{lemppisclosed} and theorem \ref{thmlsisclosed}.  
\end{proof}

It follows from the definitions
 that  $\overline{\p \ls}_1 \cong \overline{\p \p}_1 = \bigoplus_{n \geq 0} \Q x_1^{2n-2}$. 

\begin{lem}\label{lemweightzeroofpls} If $W_{\cdot}$ denotes the filtration  associated to the weight grading, then
$$W_{-1} \p\ls = \Q x_1^{-2} \ , \hbox{ and  } \gr^W_0 \p \ls =0\ .$$
\end{lem}
\begin{proof} See \S\ref{sectproofoflemweightzero}.  \end{proof}

\subsection{The residue filtration} The solutions to the polar double shuffle equations can be filtered by their pole structure as follows.  For all $r\geq 1$, define a  map
\begin{eqnarray} \label{Rrdefn}  R_r :  c_r^{-1} \Q[y_0, y_1,\ldots, y_r]   &\To  &  c_{r-1}^{-1} \Q[y_0,y_1,\ldots, y_{r-1}]  \\
{ f(y_0,\ldots, y_r) \over c_r}   & \mapsto &  { f(y_0,\ldots, y_{r-1}, y_0) \over c_{r-1}} \   ,  \nonumber 
\end{eqnarray}
where $f\in   \Q[y_0, y_1,\ldots, y_r] $, $c_0=1$. Let   $R_0$ be the zero map. Define an increasing filtration $\RR$ on $\p\p$ by   $\RR_{-1} \p\p_r = 0$, and by the property
\begin{equation} 
 f \in \RR_k  \p\p_r \quad \Longleftrightarrow \quad R_{r-k} \cdots R_{r-1}   R_{r} f  =0 \quad \hbox{ for all } k \geq 0 \ . 
 \end{equation}
 Since $R$ preserves translation invariance,  we obtain an analogous filtration, also denoted by  $\RR_k$, on $\overline{\p\p}$.
 From now on, we shall work with reduced functions in $\overline{\p\p}$ for convenience.
  In this case, note that if $\overline{f} \in \p\p_r$, then the map $R_r$ is just the residue:
  $$R_r \overline{f} = \Res_{x_r=0} \overline{f}$$
  when   $r\geq 2$, or  $\overline{f}$ is homogeneous of weight $\geq 0$. Note, however, that    $R_1   x_1^{-2} =1$.  

\begin{rem}  \label{remarkonRRfiltration} Let $\gr^{\RR}_r$ denote the associated grading. We have
\begin{equation} 
\gr^{\RR}_{0} \overline{\p\p}_1 = \bigoplus_{n\geq 0} \Q x_1^{2n} \quad  \hbox{ and } \quad   \gr^{\RR}_{1} \overline{\p\p}_1 \cong \Q x_1^{-2} \ .
\end{equation}
In general, we have $\RR_r \overline{\p \p}_r = \overline{\p \p}_r$ and    $\gr^{\RR}_r \overline{\p\p}_r \cong \gr^W_{-1} \overline{\p \p}_r$.
\end{rem}

  If $f$ is any rational function of $x_1,\ldots, x_r$, let us write
\begin{equation} \nabla_r f = \sum_{i=1}^r {\partial f \over \partial x_i}\ . \end{equation}

\begin{lem} \label{lempropertiesofRes} Suppose that $f\in \overline{\p\p}_r$ and $\overline{c_s} g \in \Q[x_1,\ldots, x_s]$, where $s >1$. 
Then the order of  pole of $f\circb g$ along $x_{r+s}=0$ is at most one. We have
\begin{eqnarray} \label{Resfcircg}
\Res_{x_{r+s}=0}  (f\circb g)   &  = &   f \circb (\Res_{x_s=0}\, g) - ( \Res_{x_s=0} \, g) \, \studot f  \ . \\
\Res_{x_{r+1} =0} ( f\circb x_1^{-2}) & = & \nabla_r f\ . \label{Resfnabla}
\end{eqnarray} 
For any  $f, g$ such that    $\overline{c_r} f \in \Q[x_1,\ldots, x_r]$ and  $\overline{c_s} g \in \Q[x_1,\ldots, x_s]$
we have
\begin{eqnarray} \label{Resfstug}
\Res_{x_{r+s}=0}  (f \,\studot  g)  &   =   & f \, \studot (\Res_{x_s=0} \, g) \ , \\
 \Res_{x_{r=0}}  (\nabla_r f)  &=  &\nabla_{r-1} (\Res_{x_r=0} f)\ .  \label{ResnablaRes}
\end{eqnarray} 
\end{lem}

\begin{proof}
If $f$ is of even degree, one verifies  using the formula for $\circb$ given in \S\ref{sectpropertiesofcircb} that 
\begin{multline}
\Res_{x_{r+s}=0}  (f\circb g)     =     f \circb (\Res_{x_s=0}\, g) + \Res_{x_{r+s}=0} f(x_{s+1}-x_s,\ldots, x_{r+s}-x_s) g(x_1,\ldots, x_s)  \nonumber \\ 
 +  (-1)^r \Res_{x_{r+s}=0} \big(f(x_{r+s}-x_{r+s-1}, \ldots, x_{r+s}-x_{s}) g(x_1,\ldots, x_{s-1}, x_{r+s}  )  \big) 
 \end{multline}
 The second term on the right-hand side of the equality vanishes. 
This implies  $(\ref{Resfcircg})$ on applying  the symmetry
$f(x_1,\ldots, x_r) = (-1)^{r+1} f(-x_r, \ldots, -x_1)$.
 For the second equation, the same formula for $\circb$ gives 
\begin{multline} 
f \circb x_1^{-2} =f(x_1,\ldots, x_r )  x_{r+1} ^{-2} 
+ f(x_2-x_1,\ldots x_{r+1}-x_1)  x_{1}^{-2}  \\
+(-1)^r f(x_r -  x_{r+1},\ldots, x_1- x_{r+1}) x_{r+1}^{-2} 
\end{multline}
and only the last term contributes to the residue.  Equation   $(\ref{Resfnabla})$ follows on applying the same symmetry for $f$ as above and expanding in $x_{r+1}$.
Equation $(\ref{Resfstug})$ follows immediately from  definition  $(\ref{defstuprod})$, and equation 
$(\ref{ResnablaRes})$ is again immediate from the definition, since $\Res_{x_r=0} { \partial \over \partial x_r} f= 0$ and $\Res_{x_r=0}$ commutes with $\partial \over \partial x_i$ for $i<r$.
\end{proof}

\begin{prop} \label{propRisIharacompatible} The filtration $R$ is compatible with the  Ihara bracket: 
\begin{equation} 
\{\RR_p \overline{\p\p} ,  \RR_q \overline{\p\p}\}\subset  \RR_{p+q} \overline{\p\p} \ . 
\end{equation}
\end{prop}
\begin{proof} Suppose that $f\in \RR_p \overline{\p\p}_r$, and $g\in   \RR_q \overline{\p\p}_s$ are homogeneous of weight $\geq 0$.
Since the map $R$ preserves the weight (the degree plus the number of variables), and coincides with the residue on elements of  weight $\geq 0$,
 it suffices to show that
\begin{equation}  \label{restoshow}
\Res_{x_{r+s - p -q}=0} \cdots \Res_{x_{r+s} =0} (f\circb g) = 0\ .
\end{equation}
Repeatedly applying  equation $(\ref{Resfcircg})$ and   identity $(\ref{Resfstug})$ gives
$$ R^{[m]} ( f \circb g) = f \circb R^{[m]} g - \sum_{i+j = m, i \geq 1} R^{[i]} g  \, \studot  R^{[j]} f$$
where $R^{[m]}$ denotes applying $R_{k}$ $m$ times. This proves $(\ref{restoshow})$ since $ R^{[p]}  f = R^{[q]} g=0$.

In the case when 
$g$ is homogeneous of  weight $-1$, and $f$ is as above,  we need to invoke $(\ref{Resfnabla})$ after taking $s-1$ iterated residues. By applying $(\ref{ResnablaRes})$,
 we obtain  $(\ref{restoshow})$ in this case also.   The case when $f$ has weight $-1$ and $g$ has weight $\geq 0$ is similar and left to the reader. 
  Note that,   by lemma \ref{lemppisclosed}, 
  $ \{  \gr^W_{-1} \overline{\p \p},  \gr^W_{-1} \overline{\p \p}\} =0$, since $\overline{\p\p}$ has weights $\geq -1$, so we have covered all cases.
 We conclude that $\{f,g\} \in \RR_{p+q} \overline{\p\p}$. \end{proof}

It is clear by cyclic symmetry that the elements in $\RR_0 \p\p$ are precisely those elements which are polynomials, i.e., which do not have any poles.
Therefore
\begin{equation} \RR_0 \p\p = \p \quad  \hbox{ and }  \quad \RR_0  \overline{\p\p} = \overline{\p}\ ,
\end{equation}
The dihedral Lie algebra is thus the lowest slice in the polar dihedral Lie algebra.
Similarly,  the linearized double shuffle Lie algebra  is characterized by
\begin{equation} \label{lsisR0pls}
\ls= \RR_0 \p \ls \ .
\end{equation}
 
\section{The  Lie algebra $\Lo$}

We construct an explicit Lie algebra $\Lo \subset \p \dmr$ of solutions to the polar double shuffle  equations modulo products. The interpretation of the associated depth-graded version of this 
Lie algebra is postponed to \S\ref{sectEisder}. 

\subsection{Special elements}  We first write down  an infinite family of solutions  $\psi_{2n+1}$ to the double shuffle equations  modulo products of weights $2n+1$, for $n\geq 1$. 

For any sets of indices $A,B \subset \{0,\ldots, d\}$, let us write
\begin{equation} \label{Xnotation}
x_{A,B} = \prod_{a\in A, b\in B} (x_a - x_b)\ .
\end{equation}
If $A$  or $B$ is the empty set,  $x_{A,B}=1$.
\begin{defn} \label{defnpsi2n+1} For every $n,d\geq 1$, define an element $\psi^{(d)}_{2n+1} \in \Or_d$ by the formula
\begin{eqnarray}
 \psi^{(d)}_{2n+1}   & =  & {1 \over 2} \sum_{i=1}^d \Big({ (x_i - x_{i-1})^{2n} \over x_{\{0,\ldots, i-2\},\{i-1\}}\, x_{\{i+1,\ldots,d\},\{i\}}}  +  { x_d^{2n} \over x_{\{1,\ldots, i-1\},\{0\}}\, x_{\{i,\ldots,d-1\},\{d\}}}Ê\Big)   \nonumber \\
  & & + {1\over 2 }    \sum_{i=1}^{d-1}\Big( { (x_1-x_d)^{2n} \over x_{\{2,\ldots, i\},\{ 1\} } \, x_{\{i+1,\ldots, d-1,0\}, \{d\} }}   -  { x_{d-1}^{2n} \over x_{\{d, 1,\ldots, i-1\}, \{0\}} \, x_{\{i,\ldots, d-2\}, \{d-1\}}} \Big) 
  \nonumber 
\end{eqnarray}
where $x_0=0$, and let $\psi^{(0)}=0$.   Let $\psi_{2n+1}$ be the element in  $\Or$  whose depth $d$ component is $\psi^{(d)}_{2n+1}$.
It is   homogeneous of weight $2n+1$. 
\end{defn}

The first main result is that the elements $\psi_{2n+1}$ are solutions to the  double shuffle equations modulo products.
 Its proof is postponed  to part III.
\begin{thm} \label{theorem: main1} For all $n\geq 1$,  $\psi_{2n+1}\in \p\dmr$.
\end{thm}

In depths one and two   definition \ref{defnpsi2n+1} gives 
\begin{eqnarray} 
\psi_{2n+1}^{(1)} & = & x_1^{2n}  \nonumber  \\
\psi_{2n+1}^{(2)} & = & {1 \over 2} \Big( {x_2^{2n}- (x_1-x_2)^{2n} \over x_1 } +   {(x_1-x_2)^{2n} -x_1^{2n} \over x_2 }+   {x_2^{2n} -x_1^{2n} \over x_1-x_2 } \Big)\nonumber 
\end{eqnarray}
which are polynomials.  Since the linearized double shuffle equations in even depth and odd weight have no solutions, $\psi_{2n+1}^{(2)}$ is uniquely determined by 
$\psi_{2n+1}^{(1)}$ and so 
$$\psi^{(1)}_{2n+1} =  \sigma_{2n+1}^{(1)} \qquad \hbox{and} \qquad \psi^{(2)}_{2n+1} =  \sigma_{2n+1}^{(2)} \ .$$
 For $d\geq 3$, the  $\psi^{(d)}_{2n+1}$ are rational functions and have non-trivial poles.

\begin{rem} \label{rempsinnotunique} The elements $\psi^{(d)}_{2n+1}$ are not  unique: in \S \ref{sectDepthSplit}  we shall show how to  construct many other universal solutions
to the double shuffle equations modulo products in all weights and depths.  A key property  of the elements $\psi^{(d)}_{2n+1}$ is discussed in 
\S\ref{sectResiduesFinal}.

\end{rem}

 By theorem \ref{thmPolarRacinet}, the elements $\psi_{2n+1}$ generate an infinite-dimensional Lie algebra of polar solutions to the double shuffle equations modulo products via the Ihara bracket. In order to cancel out poles, we construct a pure polar solution  $\psi_{-1}$ of homogeneous weight $-1$. Its  purpose is to provide counterterms for  the polar parts
of the $\psi_{2n+1}$.

\subsection{Vineyards} The definition of the element $\psi_{-1}\in \Or$ involves  terms which are indexed by a certain family of trees, which are constructed as follows.

\begin{defn} We shall call a \emph{bunch of} $n$ \emph{grapes} $g_n$  any  tree with vertices labelled from the set $\{i, i+1,\ldots, i+n\}$ of the form depicted below:

\begin{center}
\fcolorbox{white}{white}{
  \begin{picture}(100,74) (196,-40)
    \SetWidth{1.0}
    \SetColor{Black}
    \Vertex(241,11){3}
    \Vertex(200,-19){3}
    \Vertex(221,-19){3}
    \Vertex(262,-19){3}
    \Vertex(284,-19){3}
    \Line(241,12)(261,-19)
    \Line(241,12)(283,-19)
    \Line(241,12)(221,-19)
    \Line(241,12)(199,-19)
  %  \Text(246,48)[lb]{\small{\Black{$i-1$}}}
    \Text(248,12)[lb]{\small{\Black{$i$}}}
    \Text(188,-30)[lb]{\small{\Black{$i+1$}}}
    \Text(215,-30)[lb]{\small{\Black{$i+2$}}}
%    \Text(245,-30)[lb]{\small{\Black{$i+n-1$}}}
    \Text(275,-30)[lb]{\small{\Black{$i+n$}}}
       \Text(170,10)[lb]{\Black{$g_n$}}
    \Text(237,-19)[lb]{\small{\Black{$\ldots$}}}
   \Line(241,12)(199,-19)
  \end{picture}
}
\end{center}
\noindent
The vertex $i$ will be called the \emph{stalk}, and the vertices $i, \ldots, i+n$ the \emph{grapes}.

A \emph{vine} is a rooted tree whose vertices have  distinct labels $\{0,1,\ldots, n\}$,   where $0$ denotes the root vertex, obtained by grafting bunches of grapes together as follows.
Any vine $v$ is uniquely represented by a sequence $v=g_{i_1}\ldots g_{i_k}$
 where the stalk of each bunch of grapes $g_{i_{\ell}}$  is grafted to the  grape  with the highest label of the vine $g_{i_1}\ldots g_{i_{\ell}-1}$. The index $k$ will be called the \emph{height} $h(v)$ of the vine.
 
Let $\V_n$ denote the set of vines with $n$ grapes. A \emph{vineyard} is any (possibly infinite)  $\Q$-linear combination of vines.  A vineyard can be identified with an element of the 
free associative algebra generated by elements $g_i$ of degree $i$, for $i\geq 1$, and completed with respect to the degree:
\begin{equation} 
\widehat{\V}= \Q\langle \langle g_1, g_2, \ldots \rangle\rangle\ .
\end{equation}
By \S\ref{SECTstuff}, the algebra $\widehat{\V}$ is isomorphic to the stuffle Hopf algebra, and hence  a complete Hopf algebra with respect to the continuous coproduct 
$$\Delta: \widehat{\V} \To \widehat{\V} \, \widehat{\otimes}_{\Q}  \, \widehat{\V}$$
which satisfies $\Delta (g_n ) = \sum_{i+j=n} g_i \otimes g_j,$ where we set $g_0=1$.  
\end{defn}

\noindent
To every vine we associate a polynomial as follows.
\begin{defn} Let $x_0 =0$.   For every vine $T \in \V_n$, define
\begin{equation} \label{xofvine}
x_T = \prod_{(i,j)\in E(T)} (x_j -x_i) \quad  \in \quad  \Z[x_1,\ldots, x_n]
\end{equation}Ê
where   $(i,j)$ denotes the edge   between  endpoints $i,j$ of $T$,  where $0\leq i<j \leq n$.  \end{defn}

\begin{example}
 The sets $\V_1,\V_2,\V_3$ are depicted below:
\begin{center}
\fcolorbox{white}{white}{
  \begin{picture}(556,102) (43,10)
    \SetWidth{0.5}
    \SetColor{Black}
   %T1
      \Text(44,105)[lb]{\small{\Black{$g_1$}}}
    \Text(44,85)[lb]{\small{\Black{$0$}}}
    \Text(44,45)[lb]{{\Black{$1$}}}
       \Vertex(46,78){2}
    \Vertex(46,56){2}
   \Line[%arrow,arrowpos=0.5,arrowlength=3,arrowwidth=1,arrowinset=0.2
   ](46,78)(46,57)
  %T2
   \Text(98,105)[lb]{{\Black{$g_2$}}}
   \Text(98,85)[lb]{{\Black{$0$}}}
     \Text(87,45)[lb]{{\Black{$1$}}}
    \Text(109,45)[lb]{{\Black{$2$}}}
    \Vertex(100,78){2}
    \Vertex(90,56){2}
    \Vertex(111,56){2}
       \Line[](100,78)(90,56)
    \Line[](100,78)(111,56)
%T3
   \Text(145,105)[lb]{{\Black{$g_1g_1$}}}
   \Text(151,85)[lb]{{\Black{$0$}}}
    \Text(158,55)[lb]{{\Black{$1$}}}
    \Text(158,30)[lb]{{\Black{$2$}}}
\Vertex(153,78){2}
    \Vertex(153,56){2}
    \Vertex(153,34){2}
        \Line[](153,78)(153,57)
    \Line[](153,56)(153,35)
    %T4
      \Text(215,105)[lb]{{\Black{$g_3$}}}
       \Text(218,85)[lb]{{\Black{$0$}}}
      \Text(194,45)[lb]{{\Black{$1$}}}
    \Text(218,45)[lb]{{\Black{$2$}}}
    \Text(240,45)[lb]{{\Black{$3$}}}
    \Vertex(219,78){2}
    \Vertex(219,56){2}
    \Vertex(196,56){2}
    \Vertex(240,56){2}
       \Line[](219,78)(219,56)
    \Line[](219,78)(196,56)
    \Line[](219,78)(240,56)
    %T5
      \Text(275,105)[lb]{{\Black{$g_1g_2$}}}
      \Text(281,85)[lb]{{\Black{$0$}}}
      \Text(288,58)[lb]{{\Black{$1$}}}
    \Text(269,23)[lb]{{\Black{$2$}}}
    \Text(294,23)[lb]{{\Black{$3$}}}
  \Vertex(283,78){2}
    \Vertex(283,56){2}
    \Vertex(272,34){2}    \Vertex(294,34){2}
  \Line[](283,78)(283,57)
    \Line[](283,56)(272,35)
    \Line[](283,56)(294,35)
%T6
   \Text(328,105)[lb]{{\Black{$g_2g_1$}}}
  \Text(335,85)[lb]{{\Black{$0$}}}
   \Text(324,45)[lb]{{\Black{$1$}}}
    \Text(354,54)[lb]{{\Black{$2$}}}
    \Text(354,30)[lb]{{\Black{$3$}}}
    \Vertex(336,78){2}
    \Vertex(326,56){2}
  \Vertex(347,56){2}
    \Vertex(347,34){2}
 \Line[](336,78)(326,56)
    \Line[](336,78)(347,56)
    \Line[](347,56)(347,34)
   %T7
    \Text(378,105)[lb]{{\Black{$g_1g_1g_1$}}}
         \Text(388,85)[lb]{{\Black{$0$}}}
      \Text(396,55)[lb]{{\Black{$1$}}}
    \Text(396,33)[lb]{{\Black{$2$}}}
    \Text(396,14)[lb]{{\Black{$3$}}}
      \Vertex(389,78){2}
    \Vertex(389,56){2}
    \Vertex(389,34){2}
    \Vertex(389,13){2}
         \Line[](389,78)(389,57)
    \Line[](389,56)(389,35)
    \Line[](389,34)(389,13)
  \end{picture}
}
\end{center}
We have   $x_{g_1g_2}= x_1 (x_2-x_1)(x_3-x_1)$ and 
 $x_{g_1g_1g_1}= x_1 (x_2-x_1) (x_3-x_2)$.

\end{example}

\subsection{The  element $\psi_{-1}$}

\begin{defn} \label{Psi-1def}
 For all $d\geq 1$  define
$$\psi^{(d)}_{-1} = \sum_{v\in \V_d}  { (-1)^{h(v)+1} \over h(v)} { 1 \over  x_v x_d} \ .$$
Define $\psi_{-1}\in \Or$ to be the element whose depth $d$ component is $\psi^{(d)}_{-1}$.
\end{defn}

The element $\psi_{-1}$ is homogeneous of weight $-1$. 
 Its first two terms are \begin{equation}
\psi_{-1}^{(1)}= { 1 \over x_1^2} \quad  \hbox{ and }   \quad 
\psi_{-1}^{(2)}= {1 \over x_1 x_2^2}-{1 \over  2\, x_1 (x_2-x_1)x_2 }\ .
\end{equation} 

The element $\psi_{-1}$ satisfies the  double shuffle equations modulo products.
\begin{thm} \label{theorem: main2}  $\psi_{-1} \in \p \dmr$.
\end{thm}
The proof of the theorem is again  rather technical, and postponed to  part III. 

\begin{rem} \label{rempsi1notunique} The element $\psi_{-1}$ is not the unique  solution to the polar double shuffle equations of weight $-1$.  For example, we could set
$n=-1$ in   definition \ref{defnpsi2n+1} to obtain another solution. A crucial feature of $\psi_{-1}$ is the fact that it has only single poles along  all divisors 
$x_i=0$ and $x_i=x_j$, except for a  unique double pole along $x_n=0$.
\end{rem}

\subsection{The Lie algebra $\Lo$} 
\begin{defn} Let $\Lo= \Lie_{\!\Q} \langle \psi_{-1}, \psi_{3}, \psi_{5}, \ldots \rangle$ denote the Lie subalgebra  of $\Or$  generated by  $\psi_{-1}$ and $\psi_{2n+1}$ for all $n\geq 1$, equipped with the Ihara bracket $\{, \}$. 
\end{defn}

It follows from  theorems \ref{thmPolarRacinet}, \ref{theorem: main1} and \ref{theorem: main2},   that every element of $\Lo$  is a solution to the  double shuffle equations modulo products.

\begin{thm} $\Lo \subset \p \dmr$.
\end{thm}

The Lie algebra has a natural integral structure defined by the elements $\psi_{2n+1}, \psi_{-1}$. Note that this integral structure is not identical to, but related to, the 
integral structure on $\Or$;  the elements $\psi_{2n+1}$ are half-integral in this respect, but  the elements $\psi_{-1}$ have denominators $\gcd(1,\ldots, d)$ in depth $d$.

\begin{rem}  I do not know if the Lie algebra $\Lo$ is free. 
\end{rem} 

\subsection{Cancellation of poles and filtrations on $\Lo$}
The algebra $\Lo$ potentially contains elements with double poles due to the presence of double poles in the definition of $\psi_{-1}$  at all depths.
It turns out that they cancel. 
\begin{prop}  \label{propCancelofpoles} Let $\xi \in \Lo$ be any element of  non-negative weight. Then the depth $d$ component  $\xi^{(d)}\in \Or$  has at most simple poles
for all $d$. 
\end{prop}

Proofs are given in \S\ref{sectResiduesFinal}.
 The next,  crucial, result states that the first non-polynomial component of every element of $\Lo$ only has poles  along the cyclic orbits of $x_1=0$.
\begin{thm} \label{thmcycpoles} Let $\xi \in \Lo$ such that $\xi^{(1)}, \ldots, \xi^{(d-1)}$ have no poles. Then 
$\xi^{(d)} \in \Or_d$
 has poles only along the main cyclic orbit $x_1=0, \ldots, x_i=x_{i+1},\ldots, x_d=0$.
 \end{thm}
 The proof is postponed to \S\ref{sectResiduesFinal}. It depends in an essential way on special properties of the  generators
$\psi_{-1}$ and $\psi_{2n+1}$ (see remarks \ref{rempsinnotunique} and \ref{rempsi1notunique}).  As a consequence, 
 \begin{cor}
  If  $\xi \in \Lo$  such that $\xi^{(1)}, \ldots, \xi^{(d-1)}$ vanish, then 
  $\xi^{(d)} \in \p \ls_d$ is a solution to the polar linearized double shuffle equations in depth $d$.
\end{cor}
The depth filtration $\dd$  is  the decreasing filtration  on $\Lo$ 
such that  $\dd^r \Lo$ consists of elements $\xi\in \Lo$ whose components $\xi^{(1)}, \ldots, \xi^{(r-1)}$ vanish.  In particular, $\dd^1 \Lo = \Lo$. The previous corollary can be reformulated as the following theorem.
\begin{thm} \label{thmgrdLoispls}
$\gr_{\dd} \Lo \subset \p \ls$
\end{thm}

Lemma \ref{lemweightzeroofpls} corresponds to the following statement about the weights of $\Lo$.

\begin{cor}  If $W$ is the increasing filtration associated to the weight grading,
$$ W_{-1} \Lo \cong \Q \psi_{-1} \quad \hbox{ and }  \quad  \gr^W_0 \Lo \cong 0\ .$$
In particular, the commutator subalgebra $\{\Lo,\Lo\}$ has positive weights.
\end{cor} 
\begin{proof} It is enough to notice that $\gr_{\dd} \psi_{-1} = x_1^{-2}$.
\end{proof} 

Let us define $\widehat{\Lo}$ to be the completion of $\Lo$ with respect to the depth filtration.
If the conjecture  $\mathfrak{u}^{\varepsilon} = \p\ls$   is true  then the completion $\widehat{\Lo}$  cannot be free.

\section{Anatomy of the motivic Lie algebra} \label{sectAnatomy} 
We state the main conjecture and show how it implies the existence of anatomical decompositions
for the motivic Lie algebra. 

\subsection{The main conjectures} Since the double shuffle equations are known to be motivic,     $\gm \subset \dmr$.
A reformulation of a  conjecture  due to Zagier and Racinet \cite{R} is
\begin{conj}   \label{ConjZag} $Ê\gm = \dmr.$
\end{conj} 
Conjecture \ref{ConjZag} is implied by the stronger conjecture: $\gd= \ls$, which is a versio of a conjecture stated in \cite{IKZ}.
The next conjecture is, in a sense to be explained in \S\ref{sectEisder}, an elliptic version of conjecture \ref{ConjZag}.
\begin{conj} \label{MainConjVersion1}  $\p\ls$ is generated by $\p\ls_1$.
\end{conj}
This  states that the set of solutions to the polar linearized double shuffle equations is 
generated by  $x_1^{2n-2}$, for all $n\geq 0$, under  the linearized Ihara bracket. The relations in the algebra $\p \ls$ are not  known explicitly at present.

We have verified  conjecture \ref{MainConjVersion1} in depths $\leq 3$, and in certain limits.

\begin{thm}\label{thmdepth3}   $\p \ls_d$ is generated by $\p\ls_1$ for $d\leq 3$.
\end{thm} 

\begin{rem} Conjecture \ref{MainConjVersion1} is most interesting in depth four, since  by $(\ref{lsisR0pls})$, we have 
$\ls \subset \p \ls$, which contains the exotic depth four generators $\ee_f$  of \S\ref{sectExotic}. Thus we expect that for each $f\in \Sf$, $\ee_f\in\p \ls$
is a linear combination of  Ihara brackets of $x_1^{2n-2}$.
\end{rem} 

In particular, since $\gd \subset \ls \subset \p \ls$,  conjecture \ref{MainConjVersion1} implies a non-trivial statement about the depth-graded motivic Lie algebra.

\subsection{Residues and truncation}
Since  an element $\sigma$   of the motivic Lie algebra $\gm$  of weight $N$ does \emph{not} satisfy the double shuffle equations in all depths (i.e. $\dmr$ is not contained in $\p \dmr$) but only in depths $<N$, 
we must `truncate' the algebra $\Lo$.
\begin{defn} Let $\Or_d^+ \subset \Or_d$  denote the graded subspace spanned by  homogeneous rational functions of  degree $\geq 1$, and write $\Or^+ = \prod_{d \geq 1} \Or_d^+$. Let 
us denote by $+$ the projection $\Or\rightarrow \Or^+$, and let $\Lo^+ \subset \Or^+$ be the image of $\Lo$ under $+$. \end{defn} 

If $\xi=(\xi^{(d)})_{d\geq 0}$ is an element of $\Lo$ which is homogeneous of weight $N$, then its image in $\Lo^+$ is the truncation
$\xi^+ = (\xi^{(1)}, \ldots, \xi^{(N-1)},0,0, \ldots )$. In particular, $\psi^+_{-1}=0$. 

Now
consider, for every $d\geq 1$, the map
\begin{eqnarray} \label{resdmap}
\Res^{(d)} : \Lo^+  & \To &  \Or_{d-1} \\
\xi & \mapsto & \Res_{x_d=0} (\xi^{(d)})  \ , \nonumber
\end{eqnarray} 
which to the depth $d$ component of an  element  in $\Lo^+$, written as a rational function of $x_1,\ldots, x_d$, associates its residue along $x_d=0$.
Define the total residue by
\begin{equation}
\Res= \sum_{d\geq 1} \Res_d :  \Lo^+  \To \Or\ .
\end{equation}
The sum is finite by the above remarks. 

\begin{defn}Let us define $\kappa= \ker( \Res: \Lo^+ \To \Or)$.
\end{defn}

\begin{lem}  \label{lemkerresispoly} $\kappa$ is the subspace of $\Lo^+$  of elements which have no poles at all.
\end{lem} 

\begin{proof} Let $\xi \in \ker( \Res)$.   Since $\xi \in \Lo^+$, $\xi^{(1)}$ is a polynomial. Suppose by induction that 
$\xi^{(1)}, \ldots, \xi^{(k)}$ are polynomials. Then by theorem \ref{thmcycpoles}, $\xi^{(k+1)}$ has at most simple poles along the main cyclic orbit $x_i=x_{i+1}$.
Consider the double shuffle equations modulo products in depth $k+1$. The right-hand side of the stuffle equations is a polynomial, and so it follows that  $\xi^{(k+1)}$ satisfies the linearized double shuffle equations  modulo polynomials, and is therefore dihedrally symmetric modulo polynomials. Thus if  $\Res^{(k+1)} \xi^{(k+1)}=0$ it follows that $\xi^{(k+1)}$ is a polynomial, since 
the orbit of $x_{k+1}=0$ under the dihedral symmetry group contains all divisors of the form $x_i=x_{i+1}$.
\end{proof}

\subsection{Anatomy of the motivic Lie algebra}

\begin{thm} If conjecture \ref{MainConjVersion1} is true then 
\begin{equation} \label{gminkernel} \gm \subset \ker( \Res: \Lo^+ \To \Or)\ . 
\end{equation}
If Zagier's conjecture $\gm=\dmr$ is   true then 
\begin{equation} \label{gmcontainsnkernel} \ker( \Res: \Lo^+ \To \Or)  \subset \gm\ .
\end{equation}
If both conjectures hold, there is an exact sequence of bigraded vector spaces
\begin{equation} \label{gmiskernel}
0 \To \gm \To \Lo^+  \overset{\Res}{\To} \Or\ .
\end{equation}
 In other words, the motivic Lie algebra is precisely the kernel of the residue map.
\end{thm}
\begin{proof} Suppose  that conjecture \ref{MainConjVersion1} is true.
Let  $\sigma \in \gm $ be of weight $N$. Since $\sigma$ satisfies the double shuffle equations in depths $<N$, we can suppose by induction on $1\leq d<N$ that we have found $\alpha_d \in \Lo$ with 
$\sigma^{(i)}=  \alpha^{(i)}_d $ for $1 \leq i < d$. Then $\alpha_d$  has no poles in depths $\leq d-1$,  since it coincides with  (the polynomial representation of) $\sigma$ in this range, and therefore by theorem  \ref{thmcycpoles}
the component $\alpha^{(d)}_d$ has poles of cyclic type only. Since $\sigma$ has no poles, we deduce that $ \sigma^{(d)} - \alpha^{(d)}_d \in \p\ls_d$. 
Conjecture \ref{MainConjVersion1} implies that $\p \ls$ is generated in depth $1$, and hence by choosing generators,  there is a surjective map $\Lo \rightarrow \p\ls_d$ (since this is  true for $d=1$). Therefore we can find an element $e_d \in \dd^d \Lo$, which is not unique, such that $e_d^{(d)}= \sigma^{(d)}- \alpha^{(d)}_d $. Setting $\alpha_{d+1} = \alpha_d + e_d$ completes the induction step. The induction stops when $d$ attains the weight $N$ of $\sigma$. 

For the second part, let $\sigma$ be a homogeneous element  in $\Lo^+$ of weight $N$ such that $\Res_d\,  \sigma^{(d)} =0$ for all $d$. We know by lemma \ref{lemkerresispoly} that $\sigma \in \Lo^+$ is a polynomial in all depths, and satisfies the double shuffle equations in depths $<N$. It  therefore defines an element in $\dmr$. By Zagier's conjecture  $\gm = \dmr$, we deduce that $\sigma \in \gm$.
\end{proof} 

Equation $(\ref{gminkernel})$ suggests that every zeta element $\sigma_{2n+1}$ of weight $2n+1$  has a decomposition into  Lie brackets of odd weight in  
$\psi_{-1}$ and $\psi_{2m+1},$ for $m\geq 1$.  Examples are given in \S \ref{ExamplesAnatomy1} below.
 
  Equation $(\ref{gmiskernel})$ suggests that,  conversely, every such linear combination of Lie brackets which has no poles should be an element of $\gm$. 
Thus one should be able to  construct 
generators $\sigma_{2n+1}$ of $\gm$ by taking linear combinations of Lie brackets of elements $\psi$ with no residues.
Since Lie brackets of elements
$x_1^{2n}$, for $n\geq 1$, have no residues at all, we can furthermore assume that all commutators in the anatomy contain at least one $\psi_{-1}$.
An unexpected consequence is that this significantly rigidifies the space of possible generators of $\gm$, see \S\ref{sectDualityandanatomy}.\footnote{An unconditional  `anatomical decomposition'  in  a geometric context was subsequently stated in \cite{Sigma},  remark 3.9, and proved in \cite{MMV}, theorem 20.4. It is not clear what the precise relationship between this `geometric' anatomy, and  the conjectural decomposition  of zeta elements in $\Lo^+$ is. }

\subsection{Examples in low weights}  \label{ExamplesAnatomy1}
We have the following complete decompositions of the generators of the motivic Lie algebra up to weight 7:
\begin{eqnarray} 
\sigma_3 & = & \Big(\psi_3\Big)^+ \nonumber \\
\sigma_5 & = & \Big(\psi_5 - {1 \over 60} \{ \psi_{-1},\{  \psi_{-1} , \psi_{7} \}\} -  {1 \over 5} \{ \psi_{3},\{  \psi_{3} , \psi_{-1} \}\}\Big)^+ \nonumber \\
 \sigma_7  & = & \Big(\psi_7 -  {\frac {1}{112}}\,\psi_{_{\{ -1,-1,9 \}}} - \,{ 1 \over 14} \psi _{\{ 5,3,
-1 \}} -{\frac {29}{224}}\,\psi _{\{ 3,5,-1 \}}  \nonumber \\
& & +{\frac {1}{
24192}}\,\psi_{_{\{ -1,-1,-1,-1,11 \}}} +{\frac {1}{50400}}\,\psi
 _{\{ -1,3,-1,-1,7 \}} -{\frac {1}{20160}}\,\psi _{\{ -1,7,-1,
3,-1 \}} \nonumber \\
& &   +{\frac {41}{16128}}\,\psi _{\{ 3,-1,-1,7,-1 \}} -{
\frac {143}{24192}}\,\psi _{\{ -1,5,-1,5,-1 \}} +{\frac {73}{
8064}}\,\psi _{\{ 5,-1,-1,5,-1 \}} \nonumber \\
&&  -{\frac {1187}{26880}}\,\psi
 _{\{ 3,3,-1,3,-1 \}} +{\frac {8951}{403200}}\,\psi _{\{ -1,3,
3,3,-1 \}}\Big)^+ \nonumber 
\end{eqnarray}
In the last example we use the shorthand notation
$$\psi_{\{ a_1,a_2,\ldots, a_n \}} \hbox{ to denote }  \{\psi_{a_1}, \{ \psi_{a_2}, \ldots, \{ \psi_{a_{n-1}}, \psi_{a_n}\}  \cdots     \}\ .$$

\subsection{Duality  and anatomy} \label{sectDualityandanatomy}
The projective line minus three points $\Pro^1\backslash \{0,1,\infty\}$ admits an involution
$z \mapsto 1-z$ which preserves the standard tangential basepoints at $0$ and $1$.
As a result, the motivic Lie algebra $\gm$ has an induced involution which is known as the duality relation. It follows from it  that for any element $\sigma \in \gm$ of weight $2n+1$, the components 
$\sigma^{(r)}$, for all $r$  are explicitly determined from those of depth $r\leq n$. Symbolically,
\begin{equation}
\sigma_{2n+1}^{(r)} \leftrightarrow \sigma_{2n+1}^{(2n-r)}
\end{equation}Ê
In general it is not clear how to deduce the duality relation from the double shuffle equations as presented earlier, and it is far from obvious in the examples for $\sigma_3,\sigma_5,\sigma_7$ above
that the duality is indeed satisfied.

In practice, this means that the anatomical decompositions of generators $\sigma_{2n+1}$ need only be determined up to half of their weight: the rest can be reconstructed by duality. Thus the canonical generators
$\sigma_3,\ldots, \sigma_{9}$ can be expressed more compactly as:
\begin{eqnarray}\label{anatomycompact}
\sigma_3 & \equiv & \psi_3 \nonumber   \\
\sigma_5 & \equiv &  \psi_5 \nonumber \\
\sigma_7 & \equiv & \psi_7 - {1 \over 112} \{ \psi_{-1},\{  \psi_{-1} , \psi_{9} \}\} -  {1 \over 14} \{ \psi_{5},\{  \psi_{3} , \psi_{-1} \}\} - {29 \over 224} \{ \psi_{3},\{  \psi_{5} , \psi_{-1} \}\} \nonumber \\
\sigma_9 & \equiv & \psi_9 - {1 \over 180} \{ \psi_{-1},\{  \psi_{-1} , \psi_{11} \}\} -  {7 \over 180} \{ \psi_{7},\{  \psi_{3} , \psi_{-1} \}\}  \nonumber \\
& &   \qquad \qquad- {113 \over 180} \{ \psi_{3},\{  \psi_{7} , \psi_{-1} \}\}  - {1 \over 16} \{ \psi_{5},\{  \psi_{5} , \psi_{-1} \}\}   
\end{eqnarray}
The symbol $\equiv$ means that the left and right-hand sides of the equation for $\sigma_{2n+1}$  coincide in depths $r \leq n$.
Only starting from  weights $11$ are quintuple brackets required. These formulae  reproduce all the coefficients of 
$\sigma_3,\ldots, \sigma_9$ in Drinfel'd's associator. This representation
is  extremely compact: it replaces 170 numbers with just 7.

An interesting new phenomenon occurs in weight $11$.
As we shall discuss later,  a consequence of the  Ihara-Takao relation $(\ref{IharaRel})$ is the identity
\begin{multline} 
 \{ \psi_{3},\{  \psi_{9} , \psi_{-1} \}\}- \{ \psi_{9},\{  \psi_{3} , \psi_{-1} \}\}    \equiv  3   \big( \{ \psi_{7},\{  \psi_{5} , \psi_{-1} \}\}-  \{ \psi_{5},\{  \psi_{7} , \psi_{-1} \}\}    \big) \\  \mod (\hbox{depth } \geq 5) 
 \end{multline}
A priori, the  element $\sigma_{11}$ is only defined up to multiples of $\{ \sigma_3,\{ \sigma_5,\sigma_3\}\}$, and is therefore ambiguous in depth $3$.  However, we can restrict ourselves in the 
anatomical decomposition to brackets $\{\psi_{a_1}, \{ \psi_{a_2}, \psi_{a_3}\}\}$, with at least one of $\{a_1,a_2,a_3\}$ equal to $-1$. This  uniquely fixes a generator $\sigma_{11}$, which we can write  
\begin{eqnarray}
 \sigma_{11} & = & \psi_{11} - {1 \over 264} \{ \psi_{-1},\{  \psi_{-1} , \psi_{13} \}\} -  {241 \over 2112} \{ \psi_{9},\{  \psi_{3} , \psi_{-1} \}\}  \nonumber \\
 &&   \qquad \qquad   + {479 \over 2112} \{ \psi_{7},\{  \psi_{5} , \psi_{-1} \}\} - {2053 \over 6336} \{ \psi_{5},\{  \psi_{7} , \psi_{-1} \}\}  +   \quad  \{\hbox{depth } \geq 5\}  \nonumber
 \end{eqnarray}
In conclusion, we obtain\footnote{This was subsequently extended to all weights in \cite{Sigma}} a \emph{canonical} zeta  element in weight 11, which is uniquely determined by the property
$\sigma_{11}Ê(\zetam(2,7,2)) = 25/3$.

\section{Derivations from Eisenstein series and equations} \label{sectEisder}

\subsection{Geometric derivations}
Let $T(a, b) =T(\Q a\oplus \Q b) $ denote the tensor algebra generated by elements $a,b$. It is the 
$\Q$-vector space spanned by words in $a$ and $b$, equipped with the concatenation product, and the coproduct for which $a$ and $b$ are primitive.
It is bigraded by weight and depth\footnote{called $B$-degree in \cite{Sigma}}, where 
 the weight  is the grading for which $a$ and $b$ have degree $1$, and the depth grading 
gives $a$ degree $0$, and $b$  degree $1$.

 We call a derivation on $T(a, b) $  a  linear map  satisfying
  \begin{eqnarray} 
  \delta: T(a, b)  & \rightarrow & T(a, b)  \nonumber \\ 
  \delta ( u \cdot v)  & =  & \delta(u) \cdot  v + u \cdot \delta(v)  \quad \hbox{ for all } u, v \in \Q \langle a,b\rangle \nonumber  \\
 \Delta  \delta   & = &  (\delta\otimes id + id \otimes \delta) \Delta \nonumber 
  \end{eqnarray}
  The set of derivations on $T(a, b) $ is a Lie algebra for the bracket $\{\delta_1, \delta_2\}= \delta_1 \delta_2 - \delta_2 \delta_1$.
      Let $\pi_a : T(a, b)  \rightarrow \Q a$ denote the projection onto the word $a$.  
  Consider the following subspace of the set of all derivations on $T(a, b) $
  \begin{equation}
  \Der = \{ \delta  \hbox{ derivation on }T(a, b)  \hbox{ such that }  \delta(ab)=\delta(ba) \hbox{ and } \pi_a \delta(b)=0\}
  \end{equation}
  It is easy to verify that $\Der$ is stable under $\{ , \}$, and is therefore  a Lie subalgebra of the space of derivations of $T(a, b) $. Furthermore, we have  an injective map:
\begin{eqnarray}
\Der & \To &T(a, b)  \nonumber \\
\delta & \mapsto & \delta(a) \nonumber 
\end{eqnarray}
In other words, any element $\delta \in \Der$ is uniquely determined by its action on the word $a$. This follows from the equation 
$[\delta(a), b] + [a, \delta(b)]=0$, which determines $\delta(b)$ up to a  commutator of $a$. Since $\delta(b)$ is primitive,  and since the commutator of $a$ in $T(a, b) $ is   $a^n \Q$, the element
$\delta(b)$ is well-defined up to a multiple of $a$, which is zero by  $\pi_a \delta(b)=0$. We deduce that there is an injective map
\begin{equation} \label{dinj}
\delta \mapsto \delta(a) : \Der {\To} \Lie_{\!\Q} ( a, b)
\end{equation}
where $\Lie_{\!\Q}(a, b)$ is the Lie algebra of primitive elements in $T(a, b) $.

The algebra  $\Der$ is bigraded by  weight  and depth in the natural way;  the weight of $\delta$ is the weight of $\delta(a)$ minus one, and the depth of $\delta$ is the depth of $\delta(a)$.
The following Lie subalgebra of $\Der$ was first introduced  by Tsunogai \cite{Tsunogai}.

\begin{defn}  For every $n\geq 0$ there exists a unique element  $\check{\varepsilon}_{2n} \in \Der$  defined by
\begin{equation}
\check{\varepsilon}_{2n} (a) = \ad(a)^{2n} (b)
\end{equation} 
Let $\eis$  denote the graded Lie subalgebra of $\Der$ generated by the $\check{\varepsilon}_{2n}$.
\end{defn}

The Lie algebra $\eis$ will be called the Lie algebra of geometric derivations. The elements $\check{\varepsilon}_{2n}$ satisfy many  relations which were 
 studied by Pollack in \cite{P}. In particular, one readily checks that the element $\check{\varepsilon}_2$ is central in $\Der$
 \begin{equation} \label{eps2central}
 \{ \check{\varepsilon}_2,  \delta \} =0 \quad \hbox{ for all } \delta \in \Der\ .
 \end{equation}
 In particular, we  have a decomposition
 $$\eis \cong \eiso \times  \Q \check{\varepsilon}_2$$
where $\eiso\subset \eis$ is the Lie subalgebra of $\eis$ spanned by $\check{\varepsilon}_0, \check{\varepsilon}_4, \check{\varepsilon}_6, \ldots $.
Note that it is sometimes convenient to include $\check{\varepsilon}_0$ in the definition of $\eis, \eiso$, sometimes not, depending on the context.

\subsection{Representation as rational functions} \label{SectRepDerasratfunct} The Hopf algebra
$T( a,b)$ is isomorphic to the associated graded for the weight filtration of the Hopf algebra considered in \S\ref{sectshuffleequations}.  By $(\ref{wordstopolys})$, we  can thus represent elements of  $T( a,b)$  as  
polynomials:
\begin{eqnarray} \label{abtopoly}
 \rho: \gr_{\dd}^r   T( a, b)& \To &  \Q[y_0,\ldots, y_r] \\
a^{k_0} b a^{k_1} \ldots b a^{k_r}  & \mapsto &  y_0^{k_0} \ldots y_r^{k_r} \nonumber 
\end{eqnarray}
The action of $\Der$ on the element $a\in T( a,b)$  defines an injective map
\begin{eqnarray} \label{rhoronder}
\rho^{(r)}: \Der & \To &   \Q[y_0,\ldots, y_r]  \\
\delta & \mapsto & \rho(\gr^{r}_{\dd} \delta(a))  \ .\nonumber
\end{eqnarray} 

\begin{lem}  \label{lemimageofDer} The image of   $\Der$ under $(\ref{rhoronder})$  is the set of  $f\in \Q[y_0,\ldots, y_r]$ such that

1). $f$ satisfies the shuffle equations modulo products,

2). $f$ satisfies the cyclic symmetry
\begin{equation} \label{fcyclicsym} 
f(y_0,y_1,\ldots, y_r)  = f(y_1,y_2,\ldots, y_r, y_0)\ . 
\end{equation}
\end{lem}
\begin{proof}  An element $\delta\in \Der$ is uniquely determined by $\delta(a)$.  Since $a$ is primitive for $\Delta_{\sha}$, so is $\delta(a)$. Equivalently,  $\delta(a)$ satisfies the shuffle equations modulo products,  which proves the first part. Now  write $f =  \rho(\gr^{r}_{\dd} \delta(a))$ and $g = \rho(\gr^{r+1}_{\dd} \delta(b))$. The equation
$$[a, \delta(b)] + [ \delta(a),b]=0$$
  is equivalent by $(\ref{abtopoly})$ to the equation
$$ (y_0-y_{r+1}) g(y_0,\ldots, y_{r+1}) =  f(y_1, y_2, \ldots, y_{r+1}) - f(y_0, y_1,\ldots, y_r)   \ .$$
  This can be solved for $g\in \Q[y_0,\ldots, y_{r+1}]$ if and only if the right-hand side vanishes at $y_{r+1}=y_0$, which precisely gives equation
  $(\ref{fcyclicsym})$.
\end{proof}

Recall that $c_r\in \Q[y_0,\ldots, y_r]$ is the cyclic polynomial of definition \ref{defncr}. 
\begin{lem} A polynomial $f\in \Q[y_0,\ldots, y_r]$ satisfies the shuffle equations modulo products in depth $r$ if and only if the rational function
$ c_r^{-1} f $ does.
\end{lem}
\begin{proof} Using the notation $(\ref{fsharpnotation})$, we have
$$\overline{c}_r^{\sharp}(\x_1,\ldots, \x_r) = x_1 \ldots x_r ( x_1+\ldots + x_r)$$
which is invariant under permutations of the variables $x_1,\ldots x_r$. Because each shuffle equation  $(\ref{shuffleequation})$  is a  linear combination  of terms
involving precisely this  permutation group, the term $(\overline{c}_r^{-1})^{\sharp}$ factors out of the equations.
\end{proof} 
Since we are mainly interested in the subalgebra
$\eis \subset \Der$, which is concentrated in even weights, we shall restrict to  the subspace
 $\Der^{ev} \subset \Der$ of derivations of even weight, although much of what follows also holds for  $\Der$.
The shuffle equations imply translation invariance. 

\begin{defn}   The two lemmas above give a map of bigraded vector spaces:
\begin{eqnarray} \label{DertoRatmap}
\Der^{ev}  & \To &  \overline{\p\p}  \\
\delta & \mapsto &  c^{-1}_r   \rho^{(r)}(\delta) \nonumber 
\end{eqnarray} 
where $\overline{\p\p}$ is the polar dihedral Lie algebra defined in  \S\ref{sectPolardihedradLie}. In this way, we uniquely associate to every derivation  
$\delta$ of even weight a sequence of  \emph{rational functions}  in $y_0,\ldots, y_r$.
\end{defn}

The action of $\Der^{ev}$ on $T ( a, b)$ can also be encoded by rational functions as follows.
 Define   graded vector spaces $Q_0 = \Q[y_0]$ and 
$$Q_r =  \ell_r^{(-1)} \,  \Q[y_0, \ldots, y_r]  $$
where we set $\ell_0=1$ and 
$$\ell_r = (y_0-y_1) (y_1-y_2) \ldots (y_{r-1}-y_r)   \hbox{for } r \geq 1 \ .  $$
Consider the map obtained by dividing $(\ref{abtopoly})$ by $\ell_r$: 
\begin{equation}
\label{abtorat}
\ell_r^{-1} \rho^{(r)}:   \gr^{r}_{\dd}  T ( a, b)   \To  Q_r 
\end{equation}
 Recall that the (shuffle) concatenation product $(\ref{shuffony})$ is defined by 
$$f\cdot g (y_0,y_1, \ldots, y_{r+s}) = f(y_0,y_1, \ldots, y_r) g(y_r, y_{r+1}, \ldots, y_{r+s})\ .$$
Since $\ell_r \cdot \ell_s = \ell_{r+s}$, the concatenation product defines  a map $Q_r \times Q_s \rightarrow Q_{r+s}$, and it follows that the map
$$ \Q\langle a,b \rangle \overset{\sim}{\To}  \qq \qquad \hbox{ where } \quad  \qq := \bigoplus_{r\geq 0 } Q_r$$
is an isomorphism of bigraded algebras with respect to concatenation. 
\begin{defn} Define a left action $\eact$ of $\overline{\p\p}$ on $\qq$ as follows:
\begin{eqnarray} 
 \overline{\p\p} \times \qq &  \rightarrow &  \qq  \\
 f \eact g & = & f\circb g - f \cdot g \nonumber
  \end{eqnarray}
  where $\circb$ is the linearized Ihara action $(\ref{circformula})$, and $\cdot$ is the shuffle concatenation product $(\ref{shuffony})$.
 The algebra $\overline{\p\p}$ is equipped with the usual bracket  $\{f, g \}  =  f\circb g - g\circb f$.
  \end{defn} 

It follows from $(\ref{IharaderivShuffle})$ that the structures are compatible, i.e., 
\begin{equation} \label{eactcompat}
 f\eact (g \eact h) - g \eact (f \eact h) = \{ f,g \} \eact h \qquad \hbox{ for all } f,g\in \overline{\p\p}, h\in \qq\ .
 \end{equation} 
\begin{lem}  \label{lemDertoppismorphism} The action of $\Der^{ev}$ on $T(a,b)$ is given by $\eact$ on the level of rational functions. In other words, 
the following diagram commutes:
$$
\begin{array}{ccccc}
 \Der^{ev} &  \times   &  T(a,b)  & \To &  T(a,b)   \\
 \downarrow &   &  \downarrow & & \downarrow   \\
 \overline{\p\p} &   \times  & \qq    & \To  & \qq
\end{array}
$$
The map $\Der^{ev} \rightarrow \overline{\p\p}$  is a morphism of bigraded Lie algebras.
\end{lem} 

\begin{proof} It suffices to verify that $\eact$ is a derivation with respect to the shuffle concatenation, and gives the correct action on 
the elements $h_{a}= \ell_0^{-1} \rho^{(0)} (a)= y_0\in Q_0$ and $h_{[a,b]}= \ell_1^{-1} \rho^{(1)} ([a,b])=1\in Q_1$. The first statement follows from $(\ref{IharaderivShuffle})$. 
We have
\begin{eqnarray}
f(y_0,\ldots, y_r)  \circb  y_0  & = & y_0 f(y_0,\ldots, y_r) \nonumber \\ 
f(y_0,\ldots, y_r) \circb 1   &= & f(y_0,\ldots, y_r) \nonumber 
\end{eqnarray}
which follows directly from $(\ref{circformula})$. This gives
$$ f\eact  h_a = (y_0- y_r)f \qquad \hbox{ and } \qquad  f\eact h_{[a,b]} = 0$$
exactly as required (since $(\ref{DertoRatmap})$ and 
$(\ref{abtopoly})$ differ by a factor of $c_r \ell_r^{(-1)}= (y_0-y_r)$). 
 The last part follows  from  $(\ref{eactcompat})$ and the faithfullness of the action.
\end{proof}

\subsection{Double shuffle equations}
Since $\eis \subset \Der^{ev}$, the map $(\ref{DertoRatmap})$ gives
\begin{equation} \label{dtocbar} \eis   \To \overline{\p \p}
\end{equation}
 which by the previous lemma, is an injective morphism of bigraded Lie algebras.  In this section, we shall work with  reduced variables $x_i$.
 
\begin{lem} \label{lemimagesofepsilons} The map  $(\ref{dtocbar})$ maps $\check{\varepsilon}_{2n}$ to $x_1^{2n-2}$.
\end{lem}
\begin{proof} The map is computed from $\check{\varepsilon}_{2n} (a) = \ad(a)^{2n} b$. This consists of the term $b a^{2n}$, which is mapped to $x_1^{2n}/x_1^2$,  plus terms beginning in $a$, which are mapped to zero. 
\end{proof}

\begin{prop} The image of $\eis$  in  $\overline{\p\p}$  is contained in the space of solutions to the polar linearized double shuffle equations $\p \ls$.
\end{prop}

\begin{proof} The generators of $\eis$ map to elements of even degree, which are  solutions to the polar linearized double shuffle equations in depth $1$.
The latter are stable under the Ihara bracket. Since $\eis \rightarrow \overline{\p\p}$ is a morphism of Lie algebras, the result follows.
\end{proof} 
The derivations  $\check{\varepsilon}_{2n}$, for $n\neq 1$, can thus be thought of the associated depth-graded elements of the canonical generators of the Lie algebra  $\Lo$:
\begin{eqnarray}
\psi^{(1)}_{-1}  & \leftrightarrow & \check{\varepsilon}_{0} \\
\psi^{(1)}_{2n+1} & \leftrightarrow & \check{\varepsilon}_{2n} \quad \hbox{ for all } n \geq 1
\end{eqnarray} 
Equivalently, one can think of the elements $\psi_{2n+1}$ as a way to lift the geometric Lie algebra $\eiso$  to the set of  solutions to the full double shuffle equations modulo products.

\begin{rem} As an immediate consequence, we notice that  the derivations $\check{\varepsilon}_{2n+2}$  in weights $\geq 4$ coincide with the depth-graded generators of the motivic Lie algebra:
$$\check{\varepsilon}_{2n+2} =  \sigma^{(1)}_{2n+1}  \quad \hbox{ for all } n \geq 1\ .$$
In particular, we  deduce that the $\check{\varepsilon}_{2n+2}$, for $n\geq 1$, satisfy the same quadratic relations as the elements $\overline{\sigma}_{2n+1}$ given by the coefficients
of even period polynomials. This observation was found independently by Matsumoto \cite{Matsumoto}.
\end{rem}

\subsection{The residue filtration} Since $\eis$ is a Lie subalgebra of $\overline{\p\p}$, it inherits the residue filtration $\RR$. By lemma \ref{lemimagesofepsilons}, 
we have 
$$\gr^{\RR}_1 \eis_1 = \check{\varepsilon}_0 \Q  \qquad  \hbox{ and }  \qquad \RR_0 \eis_1 \cong \bigoplus_{n\geq 1} \check{\varepsilon}_{2n} \Q\ .$$
By definition, $\eis$ is generated in depth $1$, so it follows that $\RR_{r-1} \eis_r = \eis_r$ for all $r\geq 2$.
By proposition \ref{propRisIharacompatible},   an iterated Lie bracket
$$\{\check{\varepsilon}_{2n_1}, \{ \check{\varepsilon}_{2n_2},\{ \ldots, \{ \check{\varepsilon}_{2n_{r-1}}, \check{\varepsilon}_{2n_r}\} \cdots \} \quad \hbox{ is in }\quad \RR_k \eis_r $$
if at most $k$  indices $n_1,\ldots, n_r$ are equal to $0$. In other words, the residue filtration is compatible with the filtration which counts the number of $\check{\varepsilon}_0$'s. Since there
are many relations between the generators $\check{\varepsilon}_{2n}$, these two filtrations are not equal.

\subsection{$\sll_2$-action}
The Lie algebra
 $\sll_2 $ acts on the algebra $T(a, b)$  as derivations. 
We shall fix  generators $\ee,\ff$  such that:
$$\ee( a) = b \ , \quad  \ee(b) =0 \qquad \hbox{ and }  \qquad \ff( a) = 0 \ , \quad  \ff(b) =a$$ 
and $\hh= [\ee, \ff]$.  Since $\ee([a,b])=\ff([a,b])=0$, the  action of $\sll_2$ on derivations preserves $\Der$.
 It follows from the definition  that the action of  $\ee$ on $\Der$ is  $\ad ( \check{\varepsilon}_0)$.
 
 Recall  that elements $\delta \in \eis \subset \Der$ of depth $r$ can be represented as rational functions in 
 $x_1,\ldots, x_r$ via the injective map $\rr: \Der  \hookrightarrow  \overline{\p\p}_r$ defined in $(\ref{DertoRatmap})$.
The action of $\ee,\ff$ on the image of $\Der$ can be computed  explicitly by the following lemma.

\begin{lem}  Let $f\in \eis_r$ of depth $r$, and $\rr(f)$ its image in $\overline{\p\p}_r$. Then
\begin{eqnarray}
\rr(\ee(f))  & = &  \{ x_1^{-2}, \rr(f)\}  \label{eeformula}   \\
\rr(\ff(f))  & =  & \sum_{i=1}^{r-1} x_i \, \Res_{z=x_i}  \rr(f)(x_1,\ldots, x_i, z,x_{i+1},\ldots, x_r)  \label{ffformula} 
\end{eqnarray}
The second equation can alternatively be written, in the case $r\geq 2$:
\begin{equation} \label{sl2asCycRot}      \rr( \ff(f))   =   \sum_{i=1}^{r-1}  x_i (\tau \sigma)^i  R_r \rr(f)\ ,
\end{equation}
where $R_r$ is the residue  map defined in $(\ref{Rrdefn})$
\end{lem}
\begin{proof}  Equation $(\ref{eeformula})$ follows immediately from the fact that $\ee(f)= \{ \check{\varepsilon}_0, f\}$, and lemmas \ref{lemDertoppismorphism} and \ref{lemimagesofepsilons}. 
For the action of $\ff$ observe that 
$$\ff( a^{k_1} b a^{k_2} b \ldots b a^{k_r})= \sum_{i=1}^{r-1} a^{k_1} b a^{k_2} b  \ldots b a^{k_i+k_{i+1}+1} b \ldots  ba^{k_r}$$
This corresponds to the map
$$\rho(f)(y_0, \ldots, y_r) \mapsto \sum_{i=1}^r  y_i \rho(f)(y_0,\ldots, y_{i-1}, y_i, y_i, y_{i+1} \ldots, y_{r-1}) $$
Since $\rho(f) = c_r \rr(f)$, and since $\rr(f)$ has at most simple poles along $y_i = y_{i+1}$,  this gives precisely formula $(\ref{ffformula})$ after passing to the reduced representation.

The last part follows from the fact that  since $\rr(f)$ is cyclically symmetric $(\ref{fcyclicsym})$, its residue along $y_i = y_{i+1}$ is obtained from its residue 
along $y_r=0$ by applying an $i$-fold cyclic rotation of the coordinates  $(\tau \sigma)^i $. 
\end{proof}

\begin{lem} The action of $\sll_2$ on $\eis$  extends to an action on $\p\ls$. In other words, the formulae $(\ref{eeformula})$  and $(\ref{ffformula})$ define maps
$$\ee, \ff : \p\ls \To \p \ls\ .$$
\end{lem} 
\begin{proof} We must show that $\sll_2$ preserves the linearized double shuffle equations.
This is clear for the operator $\ee$ by formula $(\ref{eeformula})$,  the fact that $x_1^{-2} \in \p \ls_1$,  and theorem \ref{thmplsclosed} stating that  $\p\ls$  is closed under the Ihara bracket. 
We know that the operator $\ff$ preserves the image of $\Der$ under $(\ref{rhoronder})$, which contains 
$\p \ls$ since the linearized double shuffle equations imply cyclic symmetry.  It follows from  lemma  \ref{lemimageofDer}, that 
$\ff(f)$ satisfies the shuffle equations for all $f\in \p\ls$, and furthermore, that $\overline{c}_n \ff(f)$  has no poles (this also follows easily from  $(\ref{ffformula})$).
Therefore, it is enough to show that if $f$ satisfies the linearized stuffle equations, the same is true of 
$\ff(f)$. For this, we shall use the notation
$$ f(\x_1 \ldots \x_{i-1} \x_i^2\x_{i+1} \ldots \x_n) \quad \hbox{ for }  \quad \Res_{z=x_i} f(x_1,\ldots, x_{i-1}, x_i, z, x_{i+1},\ldots, x_n)$$
With this convention, we have 
$$\ff(f)(x_1,\ldots, x_n) = \sum_{i=1}^n x_i f( \x_1 \ldots \x_i^2 \ldots x_n)$$
The $(p,q)^{\mathrm{th}}$ linearized stuffle equation (see $(\ref{linstuffpqequation})$) for $\ff(f)$ reads
\begin{multline}  \label{inproofofsl2}
\ff(f) (\x_1 \ldots \x_p \sha \x_{p+1} \ldots \x_{p+q}) = \sum_{i=1}^p x_i f(\x_1 \ldots \x_i^2 \ldots \x_p \sha \x_{p+1} \ldots \x_{p+q}) \\
+\sum_{i=p+1}^n x_i f(\x_1 \ldots  \x_p \sha \x_{p+1} \ldots \x_i^2 \ldots \x_{p+q})
\end{multline}
However, $f$ only has poles along divisors $x_i = x_j$ where $i$ and $j$ are consecutive. It follows that, with the obvious notation where the symbol $\mathsf{z}$  stands for the variable $z$, 
$$ f(\x_1 \ldots \x_i^2 \ldots \x_p \sha \x_{p+1} \ldots \x_{p+q})  = \Res_{z=x_i} f( \x_1 \ldots  \x_i \mathsf{z} \ldots \x_p \sha \x_{p+1} \ldots \x_{p+q})$$
The right-hand side is a $(p+1, q)^{\mathrm{th}}$-linearized stuffle equation and vanishes. The same holds for the second term in 
$(\ref{inproofofsl2})$, and hence $\ff(f) (\x_1 \ldots \x_p \sha \x_{p+1} \ldots \x_{p+q} ) = 0$.  
\end{proof}
The maps $\ee, \ff$ are compatible with the residue filtration in the following sense.
\begin{lem} We have $\ee: \RR_k \p \ls \rightarrow \RR_{k+1} \p \ls$  and $\ff : \RR_{k} \p \ls \rightarrow \RR_{k} \p \ls$.
\end{lem}
\begin{proof} Since $\ee$ is given by the adjoint action of $x_1^{-2}\in \RR_1 \p \ls_1$,  the first statement follows immediately from the fact that 
the residue filtration is stable under the Ihara bracket (proposition \ref{propRisIharacompatible}). For the second statement, one sees immediately from formula 
 $(\ref{ffformula})$ that $\ff \circ R_d = R_{d-1} \circ \ff$. In particular,  $\ff$ respects the filtration $\RR$.
\end{proof}

\section{Surjectivity conjecture} \label{sectSurj}

We study the elliptic analogue of Zagier's conjecture for the polar linearized double shuffle equations, and  prove it  in depths $\leq 3$ and in some stable limit.

\subsection{The surjectivity conjecture}
In \S\ref{sectEisder} we constructed  an injective  map
$$\eis \To \p\ls$$
of bigraded Lie algebras which is an isomorphism in depth one: $\eis_1 \cong \p \ls_1$.
Since $\eis$ is by definition generated by the elements $\check{\varepsilon}_{2n}$, conjecture \ref{MainConjVersion1} is equivalent to
\begin{conj}  \label{conjueisispls} The map $\eis \cong \p \ls$ is an isomorphism.
\end{conj} 
In other words, the Lie algebra of geometric derivations is isomorphic to the polar solutions to the linearized double shuffle equations.
 Since $\gr^{\RR}_0 \p \ls \cong \ls$,  the previous conjecture, combined with Zagier's conjecture $\ls  \cong  \gd$  would imply that
\begin{equation} \label{conjzerothpieces} 
\gr^{\RR}_0 \eis \cong   \gd\ .
\end{equation}Ê
In other words, the depth-graded motivic Lie algebra should be  precisely the $0$-th degree subalgebra (with respect to $\RR$) of the Lie algebra of geometric derivations.

In this section, we present the following evidence for this conjecture.
\begin{thm} \label{thmdepth3} Conjecture $\ref{conjueisispls}$ is true in depths $\leq 3$: 
$$\eis_d \cong \p\ls_d  \quad \hbox{ for } d \leq 3\ .$$
 \end{thm}
The case of depth four, and in particular equation $(\ref{conjzerothpieces})$,  is not known and is  related in a subtle way with the  generalized Broadhurst-Kreimer conjecture \ref{conjBKZls}.
In a quite different direction, we show that $\gr^{\RR}_{d-k} \p \ls_d$  stabilizes as $d \rightarrow \infty$, for $k\leq 2$ and deduce that
conjecture $\ref{conjueisispls}$ is true in the stable limit: 
\begin{equation} \label{stablelimit}
\gr_{d-k}^{\RR} \,\eis_d \cong \gr_{d-k}^{\RR} \,\p\ls_d  \quad \hbox{ whenever } k\leq 2\ .
\end{equation}

\subsection{Some basic facts about $\gr^{\RR}_\bullet \p \ls$}
We begin with some general facts about the associated graded of $\p \ls$ for the residue filtration $\RR$:
\begin{eqnarray}
\gr^{\RR}_0 \p \ls_d &  = &  \ls_d  \label{Req1} \\
\gr^{\RR}_{k} \p \ls_d & = & 0 \hbox{ if } k \geq d  \label{Req2} \\  
 \gr^{\RR}_0 \p \ls_d &\overset{\ee}{\hookrightarrow}  &  \gr^{\RR}_1 \p \ls_d   \label{Req3}
\end{eqnarray} 
The first two equations are immediate from the definitions. 
For the third equation, recall that $\ff$ factors through the residue map by $(\ref{sl2asCycRot})$, and therefore $\ff$ vanishes on  $\gr^{\RR}_0 \p \ls_d$.
Since the commutator of $\ee$ and $\ff$ acts by multiplication by the weight, 
the operator $\ee$ is injective on $\gr^{\RR}_0 \p \ls_d$.

Let $d\geq k$ and denote the $(d-k)^{\mathrm{th}}$ iterated residue by 
$$ R^{[d-k]}  :  \p \ls_d \To c_k^{-1}\Q [x_1,\ldots, x_k]$$
where $R^{[d-k]} = R_{k+1} \circ \cdots \circ R_d$.   It follows from lemma \ref{lempropertiesofRes} that 
\begin{equation} \label{resEasNabla} 
 R^{[m]} f =0 \quad \Longrightarrow \quad R^{[m]} \ee(f) = - \nabla R^{[m-1]} f\ .
 \end{equation}
As a consequence we show that
\begin{equation} \gr^{\RR}_{d-1} \p \ls_d  \cong  \bigoplus_{2n+1 \geq d}  \ee^{d-1}(x_1^{2n})\, \Q\label{Req4}
\end{equation} 
For this, consider the iterated residue:
\begin{equation} \label{R2toRd} 
R^{[d-1]}: \gr^{\RR}_{d-1} \p \ls_d \To \Q[x_1]
\end{equation} 
which is injective by definition of $\RR$. For reasons of parity (proposition \ref{propparity}), the image is contained in $\Q[x_1^2]$ when $d$ is odd, 
and in $x_1 \Q[x_1^2]$ when $d$ is even. In particular, $ \gr^{\RR}_{d-1} \p \ls_d$ is at most one-dimensional in each weight. By applying $(\ref{resEasNabla})$ we obtain
$$R^{[d-1]} \big(  \ee^{d-1}(x_1^{2n})\big) = (- \nabla_1)^{d-1} x_1^{2n}$$
where $\nabla_1= \partial / \partial x_1$.  This
 proves the surjectivity of $(\ref{R2toRd})$, and in particular, that $\gr^{\RR}_{d-1} \p \ls_d $ is spanned by the classes $\ee^{d-1}(x_1^{2n})$.

\subsection{Proof of surjectivity in depth $\leq 3$}
By equation $(\ref{Req4})$, we have already shown  that  $\gr^{\RR}_{d-1} \p \ls_{d} \cong \gr^{\RR}_{d-1} \eis_{d}$. Therefore to prove theorem \ref{thmdepth3}, it remains to show that
 $\gr^{\RR}_{r} \p \ls_{d} \cong \gr^{\RR}_{r} \eis_{d}$ for $(r,d)=(0,2), (0,3), (1,2)$.
  \vspace{0.1in}
 
\emph{Proof that $\gr^{\RR}_0 \p \ls_d = \gr^{\RR}_0 \eis_d$  for $d=2,3$.} 
By $(\ref{Req1})$,  $\gr^{\RR}_0 \p \ls_d= \ls_d$ is just the space of solutions to the linearized double shuffle equations. 
Furthermore, we know that $\ls_1 = \RR_0 \eis_1$, since both are generated by $x_1^{2n}$, for $n\geq 0$. It therefore
 suffices to show that  $\ls_d$ is spanned by $\ls_1$ in depths 2 and 3. 
We have the following exact sequences which are taken from \S7.2 and \S7.3 in \cite{Depth}, where $\Lie_{\!3}(V)$ denotes the degree $3$ component (triple brackets)
of the free Lie algebra on some vector space $V$.
\begin{eqnarray} 
0 \To \Sf \To \ls_1 \wedge \ls_1 \To \ls_2 \To 0 \label{ls_2spanned} \\
0 \To \Sf \otimes_{\Q} \ls_1 \To \Lie_{\!3}(\ls_1) \To \ls_3 \To 0
\end{eqnarray}
Here, $\Sf$ is the space of linear relations between $\{x_1^{2m}, x_1^{2n}\}$, and is isomorphic to the space of cusp forms.
The first equation is relatively straightforward to prove, the second is a consequence of a theorem due to Goncharov who computed the precise dimension of $\ls_3$ in all weights 
using a variant of the Voronoi complex for  $SL_3(\Z)$ \cite{GG}.

\begin{cor}Let $d(s)$ denote the Poincar\'e series for  $\ls_2$.  Then
\begin{equation} \label{gsd}
d(s) =   {s^8  \over  (1-s^2)(1- s^6)} \end{equation}
\end{cor}
\begin{proof}
 The  Poincar\'e series for $\ls_1$ is $d_1(s)= {s^3 \over 1-s^2}$. It follows immediately from equation $(\ref{ls_2spanned})$ that
$d(s) ={1\over 2}\big( d_1(s^2) -d_1(s)\big) - \mathbb{S}(s)$, which gives $(\ref{gsd})$.
\end{proof}

\vspace{0.1in}
\emph{Proof that $\gr^{\RR}_1 \p \ls_3 = \gr^{\RR}_1 \eis_3$.}  We shall identify, in the usual manner, the space of translation-invariant polynomials 
in three variables $y_0,y_1,y_2$ with the space of polynomials in $x_1=y_1-y_0,x_2=y_2-y_0$. Let $C_{2}\subset \Q[x_1,x_2]$ denote the space of polynomials in two variables which are antisymmetric and whose cyclic sum is zero: i.e.,
\begin{equation}\label{fcycsumzero} 
f(y_0,y_1,y_2)+ f(y_1,y_2,y_0)+f(y_2,y_0,y_1)=0\ .
\end{equation}
Since the symmetries defining $C_2$ are contained in  the permutation group on $y_0,y_1,y_2$, there is a projection map
$$ \pi_2 :\Q[x_1,x_2] \To  C_2
$$ 
which satisfies $\pi_2^3=\pi_2$. It is given explicitly by
\begin{eqnarray}  \label{pi2def}
\pi_2(f)= f(x_1,x_2) -f(x_2,x_1)- f(x_2-x_1, -x_1)+f(x_1-x_2, -x_2)\ . \nonumber
\end{eqnarray}
\begin{lem} 
The residue defines an injective map
\begin{equation} 
R_3  :\gr^{\RR}_1 \p \ls_3 \To C_{2}\ .
\end{equation}Ê 
\end{lem} 
\begin{proof} The image of the residue map
$R_3  :\gr^{\RR}_1 \p \ls_3 \rightarrow \Or_2$
 is contained in the set of homogeneous rational functions of odd degree which satisfy
\begin{eqnarray}
f(x_1,x_2) + f(x_2,x_1) & = & 0 \label{R12equations} \\
f(x_1,x_1+x_2) + f(x_2,x_1+x_2) +f(x_2,-x_1) & = & 0 \ .\nonumber 
\end{eqnarray}
The first equation is  the residue of  the equation $f(\x_1 \sha \x_2 \x_3)=0$ along $x_3=0$ (see lemma \ref{lemRespreservesStuffle}  below); the second is the residue of  
$f^{\sharp}(\x_1 \sha \x_2 \x_3)=0$  (see $(\ref{depth3shuffle})$) along $x_1+x_2+x_3=0$. It is easy to deduce from  $(\ref{R12equations})$ and the oddness of $f$ 
that 
$$f(x_1,x_2) + f(x_2-x_1,-x_1) + f(-x_2,x_1-x_2)=0$$
which is precisely equation $(\ref{fcycsumzero})$.  We know from the definition of  the residue filtration $\RR$ that the residue of an element in $R_3  ( \gr^{\RR}_1 \p \ls_3)$ along $x_2=0$ vanishes. By 
$(\ref{R12equations})$, this implies that it has no residue along $x_1=x_2$ or $x_1=0$ either. 
\end{proof} 
Let $(C_2)_n\subset C_2$ denote the subspace  of homogeneous polynomials of degree $n$.
One easily verifies using the representation theory of the symmetric group $\Sigma_3$ that
$$
\dim_{\Q} (C_2)_n = \lfloor \textstyle{ {n+2 \over 3} }\rfloor\ .
$$If $c_2(s)$ denotes the corresponding generating series, this is equivalent to
\begin{equation}\label{gsc}
c(s) = \sum_n \dim_{\Q} (C_2)_n s^n= {s \over s^4-s^3-s+1} \end{equation}
 Next we compute  lower bounds for the image of $\gr^{\RR}_1 \p \ls_3$ in $C_2$.
 By $(\ref{Req1})$ and $(\ref{Req3})$, there is an  inclusion $\e: \ls_2 \hookrightarrow \gr^{\RR}_1 \p \ls_3$ which is split by $\ff$.  As a consequence there is a decomposition
 of $\gr^{\RR}_1 \p \ls$ as a direct sum:
 \begin{equation} \label{grR13asdirectsum} 
 \gr^{\RR}_1 \p \ls_3 \cong \e( \ls_2) \oplus H_{1,3}
 \end{equation}
 where $H_{1,3}$ is spanned by  highest weight vectors: elements $\alpha$ satisfying $\ff(\alpha)=0$.
 We know by $(\ref{ls_2spanned})$ that $\e(\ls_2) $ is  spanned by elements of the form
 \begin{equation}  \label{els2elements}
 \{ x_1^{-2}, \{ x_1^{2a}, x_1^{2b}\}\} \quad \hbox{  for  } a,b\geq 1.
 \end{equation}  It suffices to exhibit sufficiently many elements in $H_{1,3}$. For this, consider  elements
 \begin{equation} \label{habdef} 
  h_{a,b}= {1 \over 2b} \{x_1^{2a}, \{ x_1^{-2}, x_1^{2b}\}\} +  { 1 \over 2a} \{x_1^{2b}, \{ x_1^{-2}, x_1^{2a}\}\}\ , \quad \hbox{ for } a, b\geq 1\ .
  \end{equation}

     \begin{lem}  \label{lemHlowerbound}  We have $\dim_{\Q} \! \big(H_{1,3}\big)_{2k} \geq h_{2k},$ where $h_{2k}$ is defined by
      \begin{equation}
      \label{gsh}
      h(s) = \sum_{k\geq 0}  h_{2k} s^{2k} = {s \over  (1- s^2)(1 -s^6)} -s \ . \end{equation}
      \end{lem} 
 \begin{proof} It is easy to check that the elements $h_{a,b}$ do indeed satisfy  $\ff(h_{a,b})=0$ using the fact that $\ff$ is a derivation,   satisfies $\ff(x_1^{-2})=0$,  and $\ff(\{x_1^{-2}, x_1^{2a}\})= x_1 \nabla x_1^{2a} =2a\, x_1^{2a}$. 
 
 Let $V\subset \RR_1 \p \ls_3$ be the $\Q$-vector 
  space spanned by the elements $h_{a,b}$ for $a,b\geq 1$, and consider the natural map
  $ \gr^{\RR}_1:  V \rightarrow H_{1,3} \hookrightarrow C_2.$ We compute its kernel as follows.
    A linear combination of highest weight vectors  $h= \sum_{a,b\geq 1} c_{a,b} h_{a,b}$ can be represented by a symmetric polynomial in two variables
 $$g(x_1,x_2) = \sum_{a,b\geq 1}  c_{a,b} \, x_1^{2a-1} x_2^{2b-1}\ .$$
One verifies using  the formulae in lemma \ref{lempropertiesofRes} that the residue along $x_3=0$ is
 $$R_3  \, {1 \over 2b} \{x_1^{2a}, \{ x_1^{-2}, x_1^{2b}\}\}  = x_1^{2a} \circb x_1^{2b-1} - x_1^{2b-1} x_2^{2a}\ .$$
 The right-hand side  is precisely $\pi_2(x_1^{2a}x_2^{2b-1})$ $(\ref{pi2def})$.
It follows that  $R_3 h = \pi_2(x_1 g)$, and therefore the kernel of the map $V \rightarrow C_2$ is isomorphic to the space of symmetric homogeneous polynomials $g(x_1,x_2)$ of odd degree, satisfying
 $\pi_2(x_1 g) =0$. Expanding out the latter equation gives
 $$(x_1-x_2) \big( g(x_1,x_2) +g(x_2-x_1, -x_1) + g(x_1-x_2,-x_2) \big)=0$$
 (this uses the fact that $g$ is symmetric). These equations are precisely the defining equations for the space of odd period polynomials, whence an exact sequence
 $$0\To \So \To V \To H_{1,3}\ .$$
The inequality  for the dimensions of $H_{1,3}$ follows from the generating function for the dimensions of the space of cusp forms.
 \end{proof} 

It follows from $(\ref{grR13asdirectsum})$ and the injectivity of the residue map, that the dimensions of $\gr^{\RR}_1 \p \ls_3$  are bounded below 
by coefficients of the generating series $s^{-3} d(s)+h(s),$ where $d$ and $h$ are defined in $(\ref{gsd})$, $(\ref{gsh})$. It is easy to check that this is precisely
${1\over 2} \big( c(s) - c(-s)\big)-s$, $(\ref{gsc}).$ It follows that there is an exact sequence
$$ 0 \To \gr^{\RR}_1 \p \ls_3 \To C^{od}_2 \To \Q \To 0$$
where $C_2^{od} \subset C_2$ is the subspace of elements of odd degree, and  the map $C^{od}_2 \rightarrow \Q$ is projection onto elements of degree $1$.
In particular, we deduce that $ \gr^{\RR}_1 \p \ls_3$ is spanned by the elements $(\ref{els2elements})$ and  $(\ref{habdef})$. Since these elements are obviously  in the image of $\eis$, we conclude that
  $$\gr^{\RR}_1 \p \ls_3 \cong \gr^{\RR}_1 \eis_3\ .$$
 
 \begin{rem}  The elements $(\ref{habdef})$ (and their generalizations) were considered by Pollack, who also showed  that they satisfy linear relations (in a certain quotient of 
 $\eis$) whose
 coefficients are given by odd period polynomials.  From our description of $\gr^{\RR}_0 \p \ls_3 = \ls_3$ above, we deduce that these do indeed lift to genuine relations in $\eis$.

 In particular, we obtain the existence of  a natural map
 $$\So \To \ls_3\ .$$
 In general, it seems that there is an abundance of maps from the space of period polynomials into the algebra $\ls$ which would warrant further investigation.
 \end{rem}
\begin{rem} The generating series for $H_{1,3}$ is very close to the generating series for $\ls_2$ after shifting the degree.
In fact, using this observation one can prove a more precise result: namely that the following  sequence is split exact
$$ 0 \To\ls_2[-5]\To C^{od, \geq 3}_2 \To \ls_2[-1] \To 0  $$
 where the first map is multiplication by  $(x_1-x_2) \{ x_1^2,x_1^4\}^{-1}$, and the second map 
 is $f(x_1,x_2) \mapsto x_1 f(x_2-x_1, -x_1)+x_2 f(-x_2, x_1-x_2)$, with section 
 $\nabla_2$. This in fact gives a much shorter but more ad hoc proof of the previous theorem.\footnote{This proof was later given in \cite{Sigma}}
\end{rem}
\subsection{Stable limits} We conclude with some remarks on the structure of $\gr^{\RR}_m \p \ls$ for large $m$ since it gives some insight
into the structure of $\p \ls$.
The first observation is that taking the residue preserves the stuffle equations.

\begin{lem} \label{lemRespreservesStuffle} Let $f \in \Or_d$ be a solution to the linearized stuffle equations such that
$$\Res_{x_i=0}\,  f = 0 \quad  \hbox{ for } \, 2 \leq i \leq d-1\ .$$
 Then the residue 
$R_d f\in \Or_{d-1}$ is also a solution to the linearized stuffle equations.
\end{lem}
\begin{proof}  The $(p,d-p)^\mathrm{th}$ linearized stuffle equation is
$$f(\x_1 \ldots \x_p \sha \x_{p+1} \ldots \x_{d})=0\ .$$
Take the residue of this equation along $x_d=0$. By assumption on $f$, all terms in 
$\x_1\ldots \x_p \sha \x_{p+1} \ldots \x_{d}$ which do not end in $\x_{d}$ have vanishing residue.
Let $\widetilde{\partial}_{d}$ denote the   deconcatenation on the right by the letter $\x_{d}$. It is the map which  maps $w \x_{d}$ to $w$ and all words not ending in $\x_{d}$ to zero.
It is easy to check that
$$\widetilde{\partial}_{d}(\x_1\ldots \x_p \sha \x_{p+1} \ldots \x_{d}) =  \x_1\ldots \x_p \sha \x_{p+1} \ldots \x_{d-1}\ .$$
This implies that the residue satisfies the $(p,d-p-1)^\mathrm{th}$ linearized stuffle equation:
$$R_d f(\x_1 \ldots \x_p \sha \x_{p+1} \ldots \x_{d-1})=0\ .$$ 
\end{proof}
In particular, if $f\in \RR_m \p \ls$ then $R^{[m]} f$ satisfies the stuffle equations.

\begin{lem} Let $f \in \RR_{d-2} \p \ls_2$. Then $R^{[d-2]} f \in \Q[x_1,x_2]$ has no poles.
\end{lem}

\begin{proof} By definition, $R^{[d-2]} f$ has no residue along $x_2=0$. By the previous lemma,  it satisfies the stuffle equations, which implies that  it has no pole along $x_1=0$ either.
By taking the residue at $x_1=0$ in  the first shuffle equation
$f^{\sharp}(x_1 \sha x_2 \ldots x_n) =0$
and then applying $R^{[d-1]}$, we deduce that $R^{[d-2]} f$ has no pole along $x_1=x_2$ either.
\end{proof}

\begin{lem} Let $p,q \geq 1$, and let $X \in \RR_p \,\p \ls_m $, $Y \in \RR_q \, \p \ls_n $.
Then
\begin{equation}
R^{[p+q]} \{ X, Y\} = \big[ R^{[p]}(X) , R^{[q]}(Y)\big]_{\star}
\end{equation}  
where the square bracket on the right is the anticommutator with respect to $\studot$.
\end{lem}
\begin{proof} Follows from lemma \ref{lempropertiesofRes}.
\end{proof}
Finally, the statement $(\ref{stablelimit})$ follows from the previous two lemmas. By $(\ref{Req4})$ we can find $X$ and $Y$ whose iterated residues are $x_1^{2r}$ and $x_1^{2r+1}$.
These generate the set of solutions of the stuffle equations under the the bracket $[ , ]_{\star}$. This proves  $(\ref{stablelimit})$ for sufficiently large $d$ (precisely, $d\geq 5$). The case $d=4$ is easily
checked by hand.
\subsection{The spaghetti junction}
Several diverse phenomena come together in the depth four component of $\p\ls$.
We give some  formulae for these phenomena in  the simplest case of weight $12$ to illustrate the complexity of the situation.

Consider the following highest-weight vectors in $\p\ls_4$:
\begin{eqnarray}
h^4_{a,b} & =  & { 1\over b(b-1)}\{ x_1^{a}, \eee^2 (x_1^b) \} - {2 \over ab} \{\eee(x_1^{a}), \eee(x_1^b)\}  + { 1\over a(a-1)}\{ x_1^{b}, \eee^2( x_1^a) \}  \nonumber \\
h^3_{a,b,c} & =  &{1\over c} \{ x_1^a, \{ x_1^b, \eee(x_1^c)\}\} + {1\over b} \{ x_1^a, \{ x_1^c, \eee(x_1^b)\}\}  \nonumber 
\end{eqnarray}
where $a,b,c>0$ are even and $\eee = \ad( x_1^{-2})$.  One easily  shows that $h^4_{2,2n} \in \RR_1 \p \ls_4$ for all $n\geq 1$. These elements are the images  in $\p \ls$ of elements in $\eis$ which were called `Eisenstein elements'
in \cite{P}.
\vspace{0.1in}

\begin{itemize}
\item  (Spectral sequence $H(\gd) \Rightarrow \gr^{\dd} H(\gm)$). The second differential gives
\begin{equation} \nonumber 
\{\sigma_3 , \sigma_9\} - 3\,  \{\sigma_5 ,  \sigma_7 \} \equiv {691 \over 144}\,  \overline{\e}_{12} \mod \hbox{ depth } \geq 5
\end{equation} 
where $\overline{\e}_{12}$ is the cusp form generator defined in \cite{Depth}. This is closely related to an observation by Ihara and Matsumoto and is discussed at
length in \cite{Depth}.
\vspace{0.1in}

\item (The map $\ls \subset \p \ls$). If conjecture  \ref{MainConjVersion1} holds, then in particular, the  exceptional generators of $\ls$ map to  $\eis$. The first instance of this is
\begin{equation}  \nonumber 
{2275 \over 12} \overline{\e}_{12} =   h^4_{2,10} +{13 \over 90} \Big( 84 h^3_{2,2,6} + 189 h^3_{6,2,2} +25 h^3_{2,4,4} -225 h^3_{4,2,4}\Big) 
\end{equation} 
Note that the left-hand side of the equation comes from the cusp form of weight 12.  The right-hand side relates to elements which in \cite{P} are called  Eisenstein elements of weight 10.

\vspace{0.1in}
\item (Action of $z_3$). The element $z_3$ is defined in \S \ref{sectMisc}. Its action on $x_1^{8}$ was computed in \cite{P} and gives
\begin{equation}  \nonumber 
\{z_3, x_1^8\} = {691 \over 27300} h^4_{2,10}  - {11 \over 95400} \Big( 252 h^3_{2,2,6} +567 h^3_{6,2,2} +450 h_{4,2,4} - 50 h^3_{2,4,4}\Big)
\end{equation}

\vspace{0.1in}
\item (Spectral sequence on $\p \ls$ coming from $\RR$). This plays a trivial role in weight 12 (except that it provides the element $h^4_{2,10}$ which plays a role in all the above), but 
will complicate the above set-up in higher weights (it gives rise to the relations coming from even period polynomials studied in \cite{P}).

\end{itemize}

In short, we have at least four interesting maps
\begin{eqnarray}
H_1(\gd_2)  & \To & ( \ls_4)^{ab} \nonumber \\
\overline{\e}:\Sf  & \To &   \ls_4 \nonumber \\
W^e  & \To & \RR_1 \p \ls_4 \nonumber \\
\ad(z_3): \p\ls_1  & \To & \p \ls_4 \nonumber 
\end{eqnarray}
The first is the non-trivial differential in the spectral sequence from the depth filtration on the motivic Lie algebra. The second is the map given by the exceptional solutions defined in \cite{Depth}.
The third is related to the non-trivial differential in the spectral sequence coming from the $\RR$ filtration on 
$\p \ls$ (it encodes the relations discovered by Pollack in depth 4 which correspond to the space $W^e$ of even period polynomials). The fourth is the action of the  element representing the infinitesimal Galois action $z_3$ defined in \S \ref{sectMisc}.
One easily checks that $\overline{\e}(x_1^{2n}-x_2^{2n}) + \{z_3, x_1^{2n-2}\}=0$.

\section{Depth-splitting and the Witt algebra} \label{sectDepthSplit}

We construct a canonical solution to the polar  double shuffle equations in weight zero with some remarkable properties. Twisting with this element gives a mechanism  to lift 
solutions to the linearized double shuffle equations to $\p \dmr$, i.e., to all depths. It gives a splitting of the depth filtration on $\p \dmr$.  We  deduce an alternative and unconditional 
`anatomical' decomposition for elements of the motivic Lie algebra.
\subsection{A polar solution of weight zero} Recall from lemma \ref{lemevenindepth1} that any power series solution to the double shuffle equations is necessarily even in depth $1$,
corresponding to the odd Riemann zeta values $\zeta(2n+1)$. The proof of the lemma allowed for the possibility of a unique odd solution to $(\ref{inproofdepth2equations})$  if one allows poles:
$$\Phi^{(1)}(x_1)= {1 \over x_1}  \quad \hbox{ and } \quad \Phi^{(2)}(x_1,x_2) = {1 \over 3} \Big({ 2  \over x_1x_2 } + {1 \over x_1 (x_1-x_2)}\Big)\ .$$
Remarkably, this solution  extends to all higher depths. The corresponding object corresponds in some sense to $\zeta(0)$.
\begin{defn} 
Define rational functions  $s_d$ for all $d\geq 1$ by the formula
\begin{equation}Ê\label{sddef}  s_d =  \sum_{k=0}^{d-1} { (d- k) \over x_{\{0,1, \ldots,  \widehat{k}, \ldots, n\} , \{k\}}} \ .
\end{equation} 
\end{defn}
We define  an element $\psi_0 \in \Or$ which is homogeneous  of weight zero by setting
\begin{equation}  \label{psi0def}
\psi^{(d)}_0  = \binom{d+1}{2}^{-1}  s_d \qquad \hbox{ for } d\geq 1\ .
\end{equation}

We claim the following is true. The proof is omitted. 
\begin{thm} The element $\psi_0$ is in  $\p \dmr$, i.e.,  satisfies the  double shuffle equations modulo products.
\end{thm}

\subsection{Twisting by the element $\psi_0$} 
\begin{defn} \label{defnTwist} Let  $0 \neq \alpha \in \Or_d$ be a solution to the linearized double shuffle equations in depth $d$.
Let  $\widetilde{\alpha}^{(i)} = 0 $ for $i<d$, and $\widetilde{\alpha}^{(d)} = \alpha$.   Recursively define
\begin{equation} 
\widetilde{\alpha}^{(d+k)} = {1 \over 2k}  \sum_{i=1}^{k} \{\psi_0^{(i)} ,\widetilde{\alpha}^{(d+k-i)}\}
\end{equation}
for $k\geq 1$. This defines an element $\widetilde{\alpha} \in \Or$ whose first non-zero component is $\alpha$. Note that there is no restriction on the pole structure of $\alpha$. 
\end{defn} 

We claim that the following theorem is true, but again omit the proof.\footnote{In \cite{Sigma}, we gave a geometric interpretation of the operation of twisting by an element $s$  in depths $\leq 3$. The element  $s$ in that paper is  equal to  ${1\over 2} \psi_0$.}
\begin{thm}  \label{thmtwistissoln} The element $\widetilde{\alpha}$ is a solution to the double shuffle equations modulo products.
\end{thm}

In this way we could have defined solutions  to the double shuffle equations modulo products in all weights by lifting the elements $x_1^{2n}$ for $n\geq -1$. 
Denote them by
\begin{eqnarray}\label{chidef}
\chi_{-1}  &= & \widetilde{x_1^{-2}}  \\ 
\chi_{2n+1}  &= & \widetilde{x_1^{2n}} \quad \hbox{ for } n \geq 1 \nonumber 
\end{eqnarray}
One can verify that $\chi_{2n+1}$ differs from $\psi_{2n+1}$ starting from depth $3$, and $\chi_{-1}$ differs from $\psi_{-1}$ starting from depth 5. The  elements 
$\chi_{2n+1}$ have a more complicated pole structure than $\psi_{2n+1}$ in general, see \S\ref{sectNewanatomy}. 

We obtain in this way a homomorphism of Lie algebras
\begin{eqnarray}
\chi : \eis  &\To & \p \dmr \\
\check{\varepsilon}_{2n} & \mapsto & \chi_{2n-1}\ .
\end{eqnarray} 
In particular, the elements $\chi$ are not linearly independent and satisfy infinitely many  relations coming from  period polynomials whose quadratic parts were written down in \cite{P}.

\subsection{Action on the algebra of vines}
An equivalent way to write the elements $s_d$ is as the rational realization of a vineyard.
Recall that the (graded) Hopf algebra of vines $\V$  is generated by elements $g_n$ for $n\geq 1$, and equipped with the coproduct for which 
$\Delta g_n =\sum_{i=0}^n g_i \otimes g_{n-i}$,
 where here, and afterwards, we set $g_0=1$.

 \begin{defn} Define elements $\sss_n \in \V$ by the recursive formula
 \begin{equation} \label{sssrecursion}
\sum_{i=0}^{n-1} g_{i}\cdot  \sss_{n-i} = n g_n  \ . 
\end{equation}Ê
 Since $\sss_n = n g_n  +  ($products of $g_i$ with $i<n)$, it follows that 
the algebra of vines $\V$  is also generated by the elements $\sss_n$ for $n\geq 1$. \end{defn} 
  
For example, we have
\begin{equation}  
\sss_1  =   g_1 \quad , \quad 
 \sss_2  =   2 g_2 -g_1g_1 \quad , \quad \sss_3= 3 g_3 -2g_1g_2-g_2g_1+g_1g_1g_1\ .
\end{equation}
Let us define, for any vine $v$, 
$$p_v = x_v^{-1}$$
where $x_v$ is the polynomial associated to $v$ in definition \ref{xofvine}.

 \begin{lem} The elements  $s_d$ satisfy $s_d = p(\sss_d)$.
 \end{lem}

\begin{proof} 
Define elements:
\begin{equation} \label{hantipodedef}
h_n = \sum_{i_1+\ldots + i_k =n} (-1)^k g_{i_1} \ldots g_{i_k} \in \V\ ,
\end{equation} 
for all $n\geq 1$. Then $h_n = S(g_n)$, where $S$ is the antipode in $\V$.
The map $p$   is a homomorphism with respect to the shuffle concatenation product (proposition \ref{propvinestoosha}). Furthermore, 
we show in the final chapter that the $(i,j)^{\mathrm{th}}$ shuffle equation factorizes through the $(i,j)^{\mathrm{th}}$ part of the coproduct.
It follows that  since the element $h_n$ is the antipode of $g_n$,  its associated rational function is obtained by applying
the shuffle involution $\sigma$ $(\ref{sigmaonx})$ to $p_{g_n}$. We have
$$ p_{h_n} = \sigma( p_{g_n}) = \sigma  (x_{\{1,\ldots, n\}, 0}^{-1}) = x_{\{0,1,\ldots, n-1\}, n}^{-1} $$ 
It follows from the definition of $s_d$ and shuffle concatenation $(\ref{shuffprod})$ that
\begin{equation}
s_d = -\sum_{k=1}^d     p_{h_{d-k}}  \cdot  k  p_{g_k}     \in \V
\end{equation}
where we set ${h_0}=-1$.
Equation $(\ref{sssrecursion})$ follows easily from  formula $(\ref{hantipodedef})$.
\end{proof}
 The stuffle concatenation product does not act in a particularly nice way on the algebra $\V$. It is however convenient to introduce operators, defined for all $nÊ\geq 1$, 
 by
 \begin{eqnarray}
g_n \studot : \V  &\To & \V \\
g_n \studot g_{a_1}g_{a_2}  \ldots g_{a_m} & = & g_{a_1+n} g_{a_2} \ldots g_{a_m} \nonumber 
\end{eqnarray}  
It is compatible with the realization map $p$ in the sense that $p_{g_n} \studot p_{v} = p_{g_n \studot v}$. Now it follows  from the 
expansion of $\sss_d$ as a vineyard that
\begin{equation} \label{sssdasstudot}
\sss_{n+1} = g_1 \studot \sss_{n} - g_1 \sss_n + g_{n+1}\ .
\end{equation} 
Note that $g_n\studot v =  g_{n-1} \studot (g_1 \studot v)$, so $g_n \studot$ is the $n$-fold iteration of $g_1 \studot$ $n$.

Now it is not true in general that the rational functions associated to the algebra of vines $\V$  are closed under the linearized Ihara action, since they span a strict subspace of the space of rational functions.
However, the following lemma shows that they are preserved by   action on the left  of the elements $s_d$.

\begin{lem} The elements $s_d$ satisfy the following identities:
\begin{eqnarray}
 \label{xdotsd} 
s_{d+1} & = & p_{g_1} \studot s_d - p_{g_1} \cdot s_d +p_{g_{d+1}}\ , \\ 
 \label{stuffinvonsd}
s_d + \tau(s_d)  & = & (d+1) \, p_{g_d} \ , \\
s_d \circb p_{g_n} & = &p_{g_n}  \studot s_d +  n\, p_{g_{n+d}}   \ .\label{soactsonxgn} 
\end{eqnarray}
\end{lem}

\begin{proof} 
Equation $(\ref{xdotsd})$ follows immediately by applying the realization map $p$ to $(\ref{sssdasstudot}).$  
A simple consequence of the definition of $s_d$ is the following formula
\begin{eqnarray} \label{Resofsd}
\Res_{x_i = x_j} s_d  = { i-j \over x_{j, \{ 0,\ldots, \widehat{j}, \ldots, d\}}}\ ,
\end{eqnarray}
which is valid for $0 \leq j <i$, with the proviso that  $x_0=0$ as usual. Equation  $(\ref{stuffinvonsd})$ follows  by computing residues on both sides.

 Next, we prove equation $(\ref{soactsonxgn})$ by induction on $n$.
For the case $n=1$, we verify using the explicit formula for $f \circb x_1^{-1}$ (which  only has three terms) that
$$ f \circb { p_{g_1}} + (1 + \tau) \big(p_{g_1} \studot f - p_{g_1}\cdot f\big) =p_{g_1}\studot f + (f+  \tau   f) \studot p_{g_1}\ ,$$
whenever $f$ has even weight. Setting $f= s_d$ and plugging in equations $(\ref{xdotsd})$ and $(\ref{stuffinvonsd})$ gives  precisely $(\ref{soactsonxgn})$ in the case $n=1$.
Now suppose that $(\ref{soactsonxgn})$ is true for $n\geq 1$. Then, by writing
$p_{g_{n+1}} = p_{g_1} \studot p_{g_n},$
it is a simple calculation to verify that  $(\ref{soactsonxgn})$ holds for $n+1$,   using the second equation in  $(\ref{IharaderivShuffle})$ which expresses the compatibility of the linearized Ihara action with the stuffle concatenation product. 
\end{proof} 

Equation $(\ref{soactsonxgn})$ can be expressed as an action
\begin{eqnarray}
    \V  & \To &  \V \\ 
v  & \mapsto  & \sss_d\circb v \nonumber 
\end{eqnarray}
which is defined on generators $g_n$  by:
\begin{equation} \label{sssdactionongn} 
\sss_d \circb g_n = g_n \studot \sss_d + n g_{n+d}
\end{equation}
and extends to $\V$ by  the first equation of $(\ref{IharaderivShuffle})$
\begin{equation} \label{sssdactiononproduct}
\sss_d \circb ( a \cdot b)  =   (\sss_d \circb a) \cdot b + a \cdot (\sss_d \circb b) - a \cdot \sss_d \cdot b\ .
\end{equation}
In particular, we have
\begin{equation} \label{s1s2actions}
s_1 \circb g_n = (n+1) g_{n+1} \qquad  , \qquad s_2 \circb g_n = (n+2) g_{n+2} - g_{n+1} g_1
\end{equation}
If we compute the action of the $\sss_d$ on themselves, we find that they generate a copy of  (the positive degree part of) the Witt Lie algebra  under the Ihara bracket.
\begin{prop} For all $m,n\geq 1$, 
\begin{equation} \label{Wittalg}
\{ s_m, s_n\} = (m-n) s_{m+n}\ .
\end{equation}
\end{prop} 
\begin{proof} It suffices to prove the formula for $m=1,2$. The general formula then follows from the Jacobi identity and induction on $m$.
We prove that 
$$ \sss_1 \circb \sss_n = n\,  \sss_{n+1} + \sss_1 \sss_n \qquad \hbox{ and } \qquad \sss_n \circb \sss_1 = \sss_{n+1} + \sss_1 \sss_n\ .$$
The second equation follows from the equations  $\sss_{n+1} = g_1 \studot \sss_n +g_{n+1}- g_1 \sss_n$ $(\ref{sssdasstudot})$
$\sss_n \circb g_1 = g_1 \studot \sss_n + g_{n+1}$ and $\sss_1=g_1$.
For the first equation, write
\begin{equation}
S(x) = \sum_{n\geq 1} \sss_n x^n \qquad  ,\qquad G(x) = \sum_{n\geq 0} g_n x^n
\end{equation} 
where $g_0=1$.  Equation $(\ref{sssrecursion})$ is simply the formula
\begin{equation}
G(x) S(x) = x G'(x)\ .
 \end{equation} 
Equation $(\ref{s1s2actions})$ implies that $\sss_1 \circb G = G'$, and hence $\sss_1 \circb x G'=xG''$. We have
$$(\sss_1 \circb G) S + G (\sss_1 \circb S) - G \sss_1 S = xG'' \ ,$$
by   $(\ref{sssdactiononproduct})$. Since $G'S+GS'=G'+xG''$,  and $G' = G x^{-1} S$, we deduce that
$$ G \big(  \sss_1 \circb S - s_1 S -S' -S x^{-1} \big)=0 \ .$$
Because $G$ is invertible, we can divide by $G$ to  deduce the required formula for $\sss_1 \circb S$.
It follows that $\{\sss_1, \sss_{n+1}\} =n \,\sss_{n+2}$, and taking the rational realisation gives the corresponding 
identity for rational functions $s_n$.

The corresponding calculations for $\sss_2 \circb \sss_n$ and $\sss_n \circb \sss_2$ can be carried out using similar
ideas, and are omitted.
\end{proof} 
\subsection{Summary: an algebraic structure}
We have the following structure.
\begin{itemize}
\item A graded Hopf algebra $\V$ spanned by words in elements $g_n$ in degree $n$ and equipped with the concatenation product. The coproduct satisfies
$$\Delta g_n =  \sum_{i=0}^{n} g_i \otimes g_{n-i} \qquad ( g_0=1)\ .$$
\item  A distinguished family of primitive elements $\sss_n \in \V$ in degree $n$, for all $n\geq 1$. The vector space they span $W= \bigoplus_{n\geq 1} \Q\, \sss_n$
acts on $\V$ on the left
 $$\circb : W \otimes_{\Q} \V \To \V\ ,$$
in such a way that  $w \circb 1 = w$ and 
$$w \circb ( a\cdot b) = (w \circb  a) \cdot b + a \cdot (w \circb b) - a \cdot w \cdot b\ .$$
\item  The subspace $W\subset \V$ is closed under the bracket $\{\sss_m, \sss_n \} = \sss_m \circb \sss_n - \sss_n \circb \sss_m$ and 
is isomorphic to the Witt Lie algebra
$$\{\sss_m, \sss_n\} = (m-n) \sss_{m+n}$$
The action of $W$ is  completely determined by the identities  $(n\geq 0)$
$$\sss_1 \circb g_n =(n+1) \, g_{n+1} \quad \hbox{ and } \quad \sss_2 \circb g_n = (n+2) g_{n+2} - g_{n+1} g_1 \ , $$
since $W$ is generated as a Lie algebra by $\sss_1=g_1$ and $\sss_2=2g_2 -g_1g_1$.
\end{itemize}

\subsection{An unconditional anatomical decomposition} \label{sectNewanatomy}
The following theorem requires the theorem \ref{thmtwistissoln}  whose proof was omitted. 

\begin{thm}  Let $\sigma \in \Or$ be any  homogeneous solution to the  double shuffle equations modulo products of  non-negative weight.
Then there is a unique sequence $\alpha_1,\alpha_2,\ldots$ of solutions to the linearized double shuffle equations such that
$$\sigma = \sum_{dÊ\geq 1} \widetilde{\alpha}_d\ .$$
\end{thm}
\begin{proof} Since $\sigma$ has non-negative weight, its depth one component $\sigma^{(1)}\in \Q[x_1]$ is of even degree. Let $\alpha_1 = \sigma^{(1)}$.
Define $\alpha_i$ recursively by
$$ \alpha_{r+1}  =\Big( \sigma - \sum_{i=1}^r \widetilde{\alpha_i} \Big)^{(r+1)}\ .$$
Since, by induction hypothesis, $\sigma \equiv  \sum_{i=1}^r \widetilde{\alpha_i}$ modulo terms of  depth $\geq r+1$,  it follows from 
theorem \ref{thmtwistissoln} that  the element $\alpha_{r+1}$ is a solution to the linearized double shuffle equations, 
and we can define $\widetilde{\alpha}_{r+1}$ by twisting with $\psi_0$ as in $(\ref{defnTwist})$.
\end{proof} 

\begin{rem}  \label{remQ4} The elements $\alpha_d$ in the previous theorem can have poles along all divisors of the form $x_i=x_j$. These are not elements of $\p \ls_d$ in general, and therefore
do not have an interpretation as geometric derivations. It is an interesting question to ask what the generators of the Lie algebra of solutions to the linearized double shuffle equations
in this more general sense are. One can verify that the element
$$Q_4= \sum_{\Z/5\Z} { 1\over x_1 (x_3-x_2) x_3 x_4}$$
is a solution to the linearized double shuffle equations which clearly does not lie in $\p \ls_4$, since it satisfies $\Res_{x_3=0} (Q_4) = { 1\over x_1 x_2 x_4} \neq 0$.
One also checks that  
$$\psi_{-1}^{(5)} -\chi_{-1}^{(5)}   =  { 1 \over 240} \{  Q_4, x_1^{-2}\}\ .$$
\end{rem} 

\begin{ex} We can compute anatomical decompositions of the elements $\sigma$ in low weights using the elements $\chi$ and $Q_4$. This gives
\begin{eqnarray}
\qquad \sigma_5 &  \equiv & \chi_5 - {5 \over 24} \{ \chi_3, \{ \chi_{3}, \chi_{-1} \}\}  \\
\sigma_7 &\equiv  & \chi_7 -  {7 \over 48} \{ \chi_5, \{ \chi_{3}, \chi_{-1} \}\} -{7 \over 96}  \{ \chi_3, \{ \chi_{5}, \chi_{-1} \}\} + {1\over 240} \{x_1^6, Q_4\} + \ldots \nonumber \\
\sigma_9 &\equiv  & \chi_9 -  {5 \over 36} \{ \chi_7, \{ \chi_{3}, \chi_{-1} \}\} -{5 \over 108}  \{ \chi_3, \{ \chi_{7}, \chi_{-1} \}\}  \nonumber \\
 & &  \qquad - {7\over 144}   \{ \chi_5, \{ \chi_{5}, \chi_{-1} \}\}  +{1 \over 240}    \{x_1^8, Q_4\} + \ldots \nonumber 
\end{eqnarray}
where the $\ldots$ denotes higher order brackets in depth $\geq 5$. Note the presence of a term involving $Q_4$ which is absent in the anatomical decompositions involving $\psi$'s. 
One can check that  we always obtain the combination
$$\chi_{2n+1} + { 1\over 240} \{x_1^{2n}, Q_4\}$$
whose depth $5$ component has no residue at $x_4=0$. Thus  the purpose of the exceptional element $Q_4$ is to provide a counterterm to kill the extra poles in $\chi_{2n+1}$. This suggests that a possible modification of  the element $\psi_0$ may be more appropriate.
\end{ex}

Note that there are obvious patterns in the coefficients occurring in the right-hand side. These come from the fact that $\chi_{2n+1}$ is a highest-weight vector for the operator 
$\ff$ (which acts on elements of $\Or$ by exactly the formula $(\ref{ffformula})$), and therefore the depth 3 components map to highest-weight vectors in $\p \ls_3$.
More precisely, the combinations occur in pairs
$$ { 1 \over 2b} \{ \chi_{2a+1}, \{\chi_{2b+1}, \chi_{-1}\}\} +  { 1 \over 2a} \{ \chi_{2b+1}, \{\chi_{2a+1}, \chi_{-1}\}\} $$
which is to be compared with $(\ref{habdef}).$ By this remark, the element $\sigma_9$ is entirely determined by just two coefficients! 

\section{Miscellaneous remarks} \label{sectMisc}

\subsection{Denominators and congruences}
The  linearized double shuffle equations are defined over the integers. We easily  verify that 
$$\Res_{x_{2}=0} \{ x_1^{-2}, x_1^{2n} \}  = - 2 n \, x_1^{2n}\ .$$
From this it follows that the total reside vanishes modulo $2n$, and therefore
$$ \{ x_1^{-2}, x_1^{2n} \}  \in \ls_2(\F_p) \quad \hbox{ for all primes } p| 2n \ .$$ 
Whenever $\ls_2(\F_p)$ is spanned by Ihara brackets of elements in $\ls_1$,   we should expect  congruences modulo such primes.
For example, we have
\begin{eqnarray} 
 \{ x_1^{-2}, x_1^4\}    & \equiv &  0 \mod 2 \nonumber \\
 \{ x_1^{-2}, x_1^{12}\} + 2  \{ x_1^{2}, x_1^{8}\}+\{ x_1^{4}, x_1^{6}\}   & \equiv &  0 \mod 3 \nonumber \\
  \{ x_1^{-2}, x_1^{10} \}  + 2 \,  \{ x_1^{2}, x_1^{6}\}   & \equiv &  0 \mod 5 \nonumber \\
 \{ x_1^{-2}, x_1^{14} \}  +4 \,  \{ x_1^{2}, x_1^{10}\} +5 \,  \{ x_1^{4}, x_1^{8}\}   & \equiv &  0 \mod 7 \nonumber 
 \end{eqnarray}
which in turn generate many more congruences for iterated brackets.
The first and third of the equations above imply that, modulo depth $\geq 4$, 
$$ \{ \psi_{-1}, \psi_{5}\}   \equiv   0 \mod 2  \qquad , \qquad 
 \{ \psi_{-1}, \psi_{11}\}  + 2 \,  \{ \psi_{3}, \psi_{7}\}   \equiv   0 \mod 5\ . $$
Combining this with the  Jacobi identity and    the anatomical decomposition of $\sigma_9$:
$$\sigma_9  \equiv \psi_9 - {1 \over 180} \{ \psi_{-1},\{  \psi_{-1} , \psi_{11} \}\} -  {7 \over 180} \{ \psi_{7},\{  \psi_{3} , \psi_{-1} \}\} - {113 \over 180} \{ \psi_{3},\{  \psi_{7} , \psi_{-1} \}\} $$
$$   \qquad \qquad  - {1 \over 16} \{ \psi_{5},\{  \psi_{5} , \psi_{-1} \}\}  + \quad \{\hbox{depth } \geq 5\} \nonumber \\$$
we immediately deduce that some prime factors in the denominators drop out and
$$\sigma_9^{(3)} \in { 1\over 72}\,  \Z[x_1,x_2,x_3]$$
It is easy to deduce from parity relations and duality that $72  \sigma_9$ is integral in all depths. Computing a single coefficient 
in $\sigma_9^{(3)}$, we find that 
$$\sigma_9^{(3)} (\zetam(5,2,2) ) = { -3319 \over 72}\,  $$
which proves that the denominator $72$ is optimal.  In this way, we see that congruences in depth two and weight 9 in  fact determine the integrality properties of the motivic generator  $\sigma_9$ 
in all higher depths. These examples, and the well-known pathologies of the  motivic Lie algebra modulo exceptional primes  merit  further study.

\subsection{Anatomy of associators} \label{secttau} The contents of this paragraph was  essentially reproduced, with slightly different normalisations, in \cite{Sigma} \S7, and therefore is omitted.

\subsection{Algorithm for  decomposing  MZV's}
The anatomical  \mbox{decompositions} described in this paper should have practical applications.
For example, the algorithm of \cite{BrDec} can be used in combination with the methods of the present paper to give a highly efficient algorithm to decompose any motivic multiple zeta value of weight $\leq N$ and depth $\leq D$ into some chosen basis. It is, of course, conditional on conjectures \ref{conjBKZls} and \ref{conjueisispls} being true in the given range. It goes as follows:
\begin{enumerate}
\item Compute (some choice of) motivic zeta elements
$\sigma_{2n+1}$  for   $2n+1 \leq N.$
By the duality relation, it is enough to compute 
$\sigma_{2n+1}^{(d)}$ for $d\leq \min\{n,D\}$ only.  In practice this simply means computing the residues of counterterms in the canonical Lie algebra $\Lo$ to kill poles, starting from  $\psi_{2n+1}$.

\item Compute (some choice of) motivic associator $\tau$, as in \S\ref{secttau}, up to weight $N$ and depth $d\leq \min\{n,D\}$ (again, by duality).
\item  Follow the algorithm described in \cite{BrDec}. The choice of elements $\sigma_{2n+1}$ define derivation operators
which  act on motivic multiple zeta values. These can be used to decompose motivic multiple zeta values into the chosen basis.
\end{enumerate}

In \cite{BrDec}, the final step of the algorithm involved a transcendental argument: one must numerically evaluate a certain rational combination of multiple zeta values 
which is a rational multiple of $\zeta(m)$. This final step is now replaced with: compute the corresponding coefficients of $\sigma_{2n+1}$ if $m=2n+1$ is odd,
or $\tau$ if $m$ is even. Each step of the algorithm reduces the weight by at least 3 and the depth by 1.

Such an algorithm should replace the vast   tables of multiple zeta values which have  been extensively used in the literature and especially in high-energy physics. 
A list of every multiple zeta value up to given weight and depth is written as a linear combination of elements in some chosen basis.
By the above approach, one only needs to specify the coefficients in the anatomical decompositions of $\sigma_{2n+1}$ and $\tau$. 
In principle, this represents an enormous factor of compression:  a glance at  $(\ref{anatomycompact})$ shows that we have reduced the 128 coefficients in $\sigma_9$ down to just 4 numbers. Likewise,
it should be possible to reconstruct all thirty thousand structure coefficients of MZV's up to weight 13 with a few dozen numbers.

\subsection{Derivations on $\p \ls$}
The  polar linearized double shuffle equations in depth $d$  are equations of the form, for all $p+q=d$:
\begin{eqnarray}
f(\x_1\ldots \x_p \sha \x_{p+1} \ldots \x_{p+q}) & = & 0 \nonumber \\
f^{\sharp}(\x_1\ldots \x_p \sha \x_{p+1} \ldots \x_{p+q}) & = & 0 \nonumber 
 \end{eqnarray} 
where $f$ is an element of $\Or_d$ with poles along $x_i =x_j$ for $i, j$ consecutive only. It turns out that 
$\pÊ\ls$ has interesting non-trivial outer derivations.  Define
$$N^0( \p \ls)_d = \{ f \in \Or_d \hbox{ homogeneous of degree } 0 \hbox{ such that } \ad_f : \p \ls \rightarrow \p \ls\} \ .$$
There is an obvious  map $\p \ls_d \rightarrow N^0( \p \ls)_d$. We are interested in the cokernel
$$  N^0( \p \ls)_d / \pÊ\ls_d\ .$$
Examples can be constructed by considering the equations
\begin{eqnarray}
f(\x_1\ldots \x_p \sha \x_{p+1} \ldots \x_{p+q}) & = & c_{p,q} \nonumber \\
f^{\sharp}(\x_1\ldots \x_p \sha \x_{p+1} \ldots \x_{p+q}) & = & c'_{p,q} \nonumber 
 \end{eqnarray} 
where $c_{p,q}, c'_{p,q}\in \Q$  are constants. This can be verified using theorem \ref{thmRacinet} and the explicit formulae for the Ihara bracket.

\begin{ex} The first non-trivial example is the element in depth 3 
\begin{equation}
z_3={1 \over 4} \sum_{ \Z / 3 \Z} {x_2-x_1 \over x_3} \in \Or_3
\end{equation}
It satisfies $z_3(\x_1 \sha \x_2\x_3)= 0$,   $z^{\sharp}_3(\x_1 \sha \x_2\x_3)= 1$ and also $\ff(z_3)=0$.
One checks that the elements $z_3$ and $1$  span $N^0( \p \ls)_3 / \pÊ\ls_3$.
\end{ex}
The element $z_3$  defined  above
is equivalent to  the element denoted $\tilde{z}_3$ in \cite{P}, possibly with a different normalization. It corresponds to  the `arithmetic' part of the action of  $\sigma_3 \in \gm$ on the Lie algebra $\eis$.  For further detail see \cite{MMV}, theorem 20.4. We therefore expect to find non-trivial derivations $z_{2n+1}$ in $ N^0( \p \ls)_{2n+1} / \pÊ\ls_{2n+1}$ for  every $n\geq 1$.
\vspace{0.2 in}

\newpage
\begin{center} \bf{III. Proofs}
\end{center} 

In this final, technical,  section we give proofs of the theorems concerning the Lie algebra $\Lo$. They  are mostly  functional identities  between infinite series of rational functions. Our general strategy is to  conceptualise these statements as much as possible by breaking them  up  into simpler statements  which  can  be interpreted as general Hopf-algebra theoretic identities.  But  this is not  possible in every case, and some proofs 
 reduce to  technical but elementary manipulations on rational functions, which are not very enlightening.  On the other hand, we derive a number of possibly new identities which follow from the  double shuffle equations which may be of independent interest. 
\section{Preliminary remarks}

\subsection{} \label{trivialdensity} The following  remark enables us
to pass from power series to rational functions.  Call a linear functional relation of finite type, any equation of the form
$$ \sum_{i,j} r_{i,j} f^{(j)}(a_{i,j}^{(1)}, \ldots, a_{i,j}^{(j)})=0\ ,$$
where the sum is finite, and $r_{i,j}, a^{(k)}_{i,j}$ are rational functions in $x_1,\ldots, x_m$.  Clearly if this equation is satisfied for all $f^{(j)}$ polynomials with rational coefficients,
then it is identically zero, and holds for all $f^{(j)}$ rational functions.

We now modify  the shuffle and stuffle Hopf algebra structures previously defined for   formal power series to incorporate rational functions in commuting variables.

\subsection{The shuffle Hopf algebra with poles} Let us write $\Or_0^{\sharp}=\Q$ and 
\begin{equation}
\Or_d^{\sharp} = \Q [x_1,\ldots, x_d, (x_i+\ldots +x_j)^{-1}_{ 1 \leq i < j \leq d}] \quad  \hbox{ for }  \quad d \geq 1\  , 
\end{equation}  
and put $\Or^{\sharp} = \prod_{d \geq 0 }Ê\Or_d^{\sharp}$. We define a linear map $\sharp: \Or \rightarrow \Or^{\sharp}$  as in $(\ref{sharpdef})$.  Consider the homomorphism
\begin{equation}
m_{p,q} : \Or^{\sharp}_p \otimes_{\Q} \Or^{\sharp}_q \To \Or^{\sharp}_{p+q}   \nonumber
\end{equation}
defined by sending 
$x_i \otimes 1 $ to $x_i$ and $1\otimes x_j$ to $x_{j+p}$ for all $1\leq i \leq p, 1\leq j \leq q$. The $(p,q)$-th shuffle equation defines a map that we shall denote by 
\begin{equation}
Sh_{p,q} : \Or_{p+q} \To \Or^{\sharp}_{p+q}\ .
\end{equation} 
\begin{defn} Let $\Osha \subset \Or$  denote the subspace defined (recursively)  by 
$$\Osha_{n}= \{ f\in \Or_{n}:   Sh_{p,q}(f) \in m_{p,q} (\sharp \otimes \sharp) ( \Osha_p \otimes_{\Q} \Osha_q) \hbox{ for all }Êp+q=n\}\ ,$$
for all $n\geq 2$, and   $\Osha_1 = \Or_1$.
\end{defn} 
Since the maps $m_{p,q}$ and $\sharp$ are injective, we obtain uniquely defined   morphisms $\Delta_{p,q} : \Osha_{p+q} \rightarrow \Osha_p \otimes_{\Q} \Osha_q$ such that the following diagram
commutes
$$
\begin{array}{ccc} 
\Delta_{p,q}: \Osha_{p+q}    & \To     &  \Osha_p \otimes_{\Q} \Osha_q  \\
\quad \qquad   \downarrow &   & \downarrow   \\
Sh_{p,q}:  \Osha_{p+q} &  \To   &   \Or^{\sharp}_{p+q}\ , 
\end{array}
$$
where the vertical map on the right is $m_{p,q} \circ (\sharp \otimes \sharp)$, the one on the left is the identity.
\begin{prop} \label{propOshaisHopf}
 The space $\Osha$, equipped with the shuffle concatenation map $(\ref{shuffleconcat})$ and the coproduct $\Delta = \sum_{p,q} \Delta_{p,q}$, is a   cocommutative  graded  Hopf algebra. 

Its Lie algebra of primitive elements consists of the set of solutions in $\Or$ to the shuffle equations modulo products.
\end{prop} 

\begin{proof}
One has to check that the map $\Delta$ defined above satisfies the usual axioms for Hopf algebras. 
For example, the fact that it is a homomorphism for the multiplication translates into some functional identities for the maps $Sh_{p,q}$ (see remark below). 
That these are satisfied follows from proposition \ref{propshuffleequations}:  the maps $Sh_{p,q}$, restricted to the subspace of elements with no poles $\Q[x_1,\ldots x_{p+q}] \subset \Osha_{p+q}$
are the images  of the coproduct 
$$\Delta_{\sha}^{p,q}: \Q[x_1,\ldots, x_{p+q}] \To  \Q[x_1,\ldots, x_{p}]  \otimes  \Q[x_1,\ldots, x_{q}] \ .$$
Therefore $\Delta$, restricted to the subspace $ \Q[x_1,\ldots, x_n]  \subset \Osha_n$, coincides with $\Delta_{\!\sha}$ for all $n$,  and satisfies 
 all the requisite axioms of a Hopf algebra on this  subspace (the Hopf algebra in question is the image in  $\Or$ of the subspace of  $\Q\langle \langle e_0, e_1\rangle \rangle^{tr}$ consisting of finite power series). 
By \S\ref{trivialdensity}, it does so on the whole of $\Osha$.
\end{proof}
In particular, the map $(\ref{wordstopolys})$ provides an injective  morphism  of Hopf algebras
$$(\Q\langle  e_0, e_1 \rangle^{tr}, \cdot, \Delta_{\!\sha})  \To (\Osha, \cdot, \Delta)\ ,$$
where $\Q\langle e_0, e_1 \rangle^{tr}$ is the weight-graded version of $\Q\langle \langle e_0, e_1 \rangle \rangle^{tr}$, defined to be the subspace of finite power series (which is a Hopf algebra by  corollary \ref{cortransinv}).

\begin{rem} 
Let $f_1\in \Or_m$, and $f_2 \in \Or_n$. The proposition implies, in particular, the following functional equations for the shuffle product:
\vspace{-0.1in}
\begin{multline} \label{Sijonprod}
Sh_{k,m+n-k}( f_1 \cdot f_2) = \sum_{\ell=0}^k  (Sh_{\ell,m-\ell} f_1) (x_1,\ldots, x_{\ell}, x_{m+1},\ldots, x_{m+n-\ell}) \nonumber \\ 
\times   (Sh_{k-\ell,n-k+\ell} f_2) (x_{\ell+1},\ldots, x_{m}, x_{m+n-\ell+1},\ldots, x_{m+n})\ .
\end{multline}
where $Sh_{0,n}= Sh_{n,0}$ is defined to be the map $\sharp$.
\end{rem}

\subsection{The stuffle Hopf algebra with poles}
Let us write $\Or_{\leq n} $ for $\Or / \dd^{n+1} \Or$. 
Consider the homomorphism
\begin{equation}
n_{p,q} : \Or_p \otimes_{\Q} \Or_q \To \Or_{p+q}   \nonumber
\end{equation}
defined by sending 
$x_i  \otimes 1$ to $x_i$ and $1\otimes x_j $ to  $x_{j+p}$ for all $1\leq i \leq p, 1\leq j \leq q$.  We obtain a map which we denote by $\overline{n}_{p,q}:   \Or_{\leq p} \otimes_{\Q} \Or_{\leq q}\rightarrow\Or_{\leq p+q}  $
by summing over $n_{i,j}$ for $i\leq p, j \leq q$.
The $(p,q)$-th stuffle equation defines a map that we shall denote by 
\begin{equation}
St_{p,q} :  \Or_{\leq p+q} \To \Or_{p+q}\ .
\end{equation} 
\begin{defn} Let $\Ostu \subset \Or$  denote the subspace defined  by 
$$\Ostu= \{ f\in \Or:   \sum_{i \leq p, j \leq q} St_{i,j}(f) \in \overline{n}_{p,q}( \Ostu_{\leq p} \otimes_{\Q} \Ostu_{\leq q}) \hbox{ for all }Êp,q \geq 1\}\ .$$
There are finitely many equations defining $\Ostu_{\leq n}$,  since 
$St_{p,q}$ only involves the components $\Or_n$ for $\max\{p,q\} \leq n \leq p+q$.
\end{defn} 
Since the map $\overline{n}_{p,q}$ is injective, we obtain in this way  uniquely defined   morphisms $\Delta_{p,  q} : \Ostu_{\leq p+q} \rightarrow \Ostu_{\leq p} \otimes_{\Q} \Ostu_{\leq q}$ such that the following diagram
commutes
$$
\begin{array}{ccc} 
 \quad  \Delta_{\leq p, \leq q}: \Ostu_{\leq p+q}    & \To     &  \Ostu_{\leq p} \otimes_{\Q} \Ostu_{\leq q}  \\
\qquad \qquad \qquad   \downarrow &   & \downarrow   \\
\sum_{i \leq p, j \leq q} St_{i,j}:  \Ostu_{\leq p+q}\quad  &  \To   &   \Or_{\leq p+q}\ , 
\end{array}
$$
where the vertical map on the right is $\overline{n}_{p,q}$, and the one on the left is the identity.
\begin{prop} The space $\Ostu$, equipped with the stuffle concatenation map $(\ref{stuffleconcat})$ and the coproduct $\Delta = \sum_{p,q} \Delta_{p,q}$, is a   cocommutative  graded Hopf algebra. 

Its Lie algebra of primitive elements consists of the set of solutions in $\Or$ to the stuffle equations modulo products.
\end{prop} 

\begin{proof} The proof is similar to the proof of proposition   \ref{propOshaisHopf}.
\end{proof}
In particular, the map $(\ref{wordstopolys})$ provides an injective  morphism  of Hopf algebras
$$(\Q\langle  Y \rangle , \studot, \Delta_{\stu})  \To (\Ostu, \studot, \Delta)\ ,$$
where $\Q\langle Y \rangle$ is the weight-graded version of $\Q\langle \langle Y \rangle \rangle$.

\subsection{Remarks on free Lie algebras}ÊWe can replace the set of all shuffle equations (or linearized stuffle equations) in a given depth 
with a single functional equation. 

\subsubsection{Lie projectors}
Let $k$ be a field of characteristic zero, and let $X$ be a finite set. Let $V$ denote the $k$-vector space with basis  $X$.  Let  
$\Lo(X)$ be the free graded Lie algebra generated by $X$, where elements in $X$ have degree $1$. Its universal enveloping $\U \Lo(X)$ is isomorphic to the tensor algebra $T(V)=\bigoplus_{ n \geq 0 } V^{\otimes n}$.  
It is equipped with the coproduct for which $x$ is primitive for all $x\in X$.  Let
$$i : \Lo(X) \To T(V)_{>0}$$
denote the inclusion of $\Lo(X)$ in $T(V)$. Its image consists of the space of primitive elements in $T(V)$.  The Lie projector is defined in degree $n\geq 2$ by
\begin{eqnarray}
p_{\Lo}: T(V)_{>0}   &\To & T(V)_{>0}   \\
x_1 \otimes \ldots \otimes x_n & \mapsto & [x_1, [ x_2,  \ldots [x_{n-1}, x_n]  \cdots ] \nonumber
\end{eqnarray} 
and is the identity in degree $n=1$. 
A  well-known result states that the image of $p_{\Lo}$ is exactly $i ( \Lo(X))$ and that 
$ f \in i (\Lo_n(X))$ if and only if $p_{\Lo}(f) = n f,$
where $\Lo_n(X)$ is the subspace of  $\Lo(X)$ of degree $n\geq 1$.  Thus for all $n\geq 1$, the normalized Lie projector defines  an idempotent projector on to the image of $\Lo(X)$:
\begin{equation}  \label{ninvpidd}
 { n^{-1}}Êp_{\Lo} :   T(V)_{n} \To  i ( \Lo_n(X)) \ . 
 \end{equation}Ê
 \vspace{-0.2in}
 \subsubsection{Dual projectors}  Denote by  
$\Q\langle X\rangle$ the commutative graded Hopf algebra whose underlying vector space is spanned by words in $X$,  equipped with the shuffle product, and deconcatenation coproduct.
Consider the map  $\lambda$ which is the identity in degree $1$, and is defined recursively in degree $n\geq 2$ by the formula 
\begin{eqnarray}
\lambda: \Q\langle X \rangle_{>0}  &\To & \Q \langle X \rangle_{>0} \\
x_1 \ldots x_n &  \mapsto  & x_1 \lambda (x_2\ldots x_n) - x_n \lambda(x_1\ldots x_{n-1})\ . \nonumber 
\end{eqnarray}
 Denote the Lie coalgebra of  indecomposable elements of $  \Q\langle X \rangle$ by 
 $$I(X) = {\Q\langle X \rangle_{>0} \over \Q\langle X \rangle_{>0}\Q\langle X \rangle_{>0}}\ .$$
It  is the graded dual to $\Lo(X)$.  
One shows that an element $ f \in \Q\langle X \rangle$ of degree $n\geq 1 $ is in the kernel of $\lambda$ if and only if it is decomposable (shuffle product), i.e.
$$ \lambda (f) =0 \qquad  \Longleftrightarrow \qquad  f \in   \Q\langle X \rangle_{>0}\Q\langle X \rangle_{>0}\ .$$
Therefore $\lambda$ induces an injective map
 $$ \overline{\lambda}: I(X)  \To \Q\langle X \rangle_{>0} \ ,$$
 which splits the natural map $\Q\langle X\rangle _{>0} \rightarrow I(X)$. 
Furthermore, an element   $f\in \Q\langle X \rangle_n $ satisfies 
 $  f \in \overline{\lambda} (I(X))$ if and only if $\lambda(f) = n f.$
 
 \subsection{Linearized double shuffle as two equations}Ê
Let $n\geq 1$, and define a map  $\epsilon_n: \Sigma_n\rightarrow \{0, \pm 1\}$, where $\Sigma_n$ is the set of  permutations  on $n$ letters, by
\begin{equation}Ê \label{lambdaprojector}
\lambda(x_1\ldots x_n) = \sum_{\sigma \in \Sigma_n}  \epsilon_n( \sigma)  x_{\sigma(1)} \ldots x_{\sigma(n)}\ .
\end{equation}
If $f \in \Q[x_1,\ldots, x_n]$, let us write 
\begin{equation} 
(\lambda_n f) (x_1,\ldots, x_n) =  \sum_{\sigma \in \Sigma_n}  \epsilon_n( \sigma)  f(x_{\sigma(1)} ,\ldots, x_{\sigma(n)}) \ .
\end{equation}
For example, $\lambda_2 f(x_1,x_2) = f(x_1,x_2) - f(x_2,x_1)$,  and
$$\lambda_3 f(x_1,x_2,x_3) = f(x_1,x_2,x_3) - f(x_1,x_3,x_2) - f(x_3,x_1,x_2) + f(x_3,x_2,x_1)$$
The following corollary follows immediately from the previous discussion, and  reduces the linearized double shuffle equations (\S\ref{sectLinDS}) in depth $n$ to just two equations.
\begin{cor} \label{corDSto2eq} An element $f\in \Q[x_1,\ldots, x_n]$, where $n\geq 2$,  satisfies the linearized stuffle equations if and only if 
\begin{equation} \label{pinfstuff}
\lambda_nf  = n f  \ .
\end{equation} 
It satisfies the shuffle equations if and only if
\begin{equation} \label{pinshuff} 
(\lambda_n f)^{\sharp}  =n f^{\sharp} \ .
\end{equation}
By \S\ref{trivialdensity}, the same statements are true for rational functions.
 \end{cor}
It is important to note that for applications, it is often more useful to describe the solutions to the double shuffle equations using more, rather than fewer,  equations.

\section{Proofs of stuffle identities}
In this section,  we prove that the elements $\Psi_{2n+1}$  and $\Psi_{-1}$ satisfy the stuffle equations modulo products by generating a large supply of solutions
to these equations, and by showing that $\Psi_{2n+1}$  and $\Psi_{-1}$ are in the Lie algebra that they generate.

\subsection{Elementary solutions} A natural way to construct solutions to the  stuffle equations modulo products is as follows. 
\begin{lem}\label{lembadlift1}
For all $n\geq 1$, the elements  $\rho_n \in \Q\langle \langle Y \rangle \rangle$ defined by 
\begin{equation} \label{rhondef}
\rho_n= \sum_{i_1+\ldots +i_k =n , \,  i_j \geq 1} { ( -1)^{k-1} \over k} \, \y_{i_1} \ldots \y_{i_k}\ ,
\end{equation}
are primitive, i.e., $\Deltat \rho_n = 1 \otimes \rho_n + \rho_n \otimes 1$. 
\end{lem}
\begin{proof} By definition,  we have $\Deltat \y_n = \sum_{i+j=n} \y_i \otimes \y_j$ for all $n$.  This is equivalent to the fact that the element
$$G(t)=1+\sum_{n\geq 1 } t^n \y_n  \in \Q[t]\langle \langle Y \rangle \rangle$$ 
is group-like. It follows that its logarithm 
$$\log (G(t))=- \sum_{k\geq 1} {1\over k} \big(1-G(t)\big)^k = \sum_{k\geq 1} \rho_k t^k  $$ is primitive, and therefore $\rho_k$ is primitive for all $k\geq 1$.
\end{proof}
By $(\ref{rhondef})$, the depth-$r$ component of $\sum_{k\geq 1} \rho_k t^k $ corresponds via $( \ref{stufftopowerseries})$  to 
\begin{equation} \label{oneminustxis}
{ (-1)^{r-1} \over r}{t^r \over (1-tx_1) \ldots (1-tx_r)} \ .
\end{equation}

\begin{cor} \label{cordepth1badlift} Let  $f \in \Q[x_1^{-1}, x_1]]$. Then the element $\ell(f)$ whose $d^{th}$ component is
\begin{equation} \label{lfddef} 
\ell(f)^{(d)}  ={1 \over d} \sum_{i=1}^d   {f(x_i) \over \prod_{j \neq i} (x_j-x_i) }  
\end{equation} 
defines a solution to the stuffle equations modulo products. 
\end{cor}
\begin{proof}    By the previous lemma, the element whose $r^{\mathrm{th}}$ component is given by $(\ref{oneminustxis})$ satisfies the stuffle equations modulo products.
By decomposing   $\prod_{i=1}^d (1-x_i t)^{-1}$ into partial fractions and expanding in $t$, we obtain
 $$ {(-1)^{d-1}\over d}    {t^d \over (1-tx_1) \ldots (1-tx_d)}= { 1 \over d} \sum_{n\geq d}  \sum_{i=1}^d   {x_i^{n-1} t^n \over \prod_{j \neq i} (x_i-x_j) }  $$
Taking the coefficient of $t^n$ in the right-hand side gives  $\ell(x_1^{n-1})^{(d)}$, when $n\geq d$.
Thus   $\ell(f)$    is a solution to the stuffle equations modulo products  up to depth $d$ when $f=x_1^{n-1}$, for all $n\geq d$.  By linearity and 
\S\ref{trivialdensity}, $\ell(f)$ is a solution to the stuffle equations modulo products for any Laurent series in $x_1$. 
 \end{proof}

\subsection{Conjugation operator} By taking $t$ to infinity in $(\ref{oneminustxis})$,  we obtain a solution  $\nu= \ell(x^{-1})$   to the inhomogeneous stuffle equations which  has components 
\begin{equation} \label{nurdef}
\nu^{(r)} = (r x_1\ldots x_r)^{-1}\ ,
\end{equation}
in depth $r\geq 1$. 
Define elements $\mu_+, \mu_- \in \Ostu$ by 
\begin{equation}
\label{muidef} 
1+ \mu_+ = \exp_{\stu}( \nu) \qquad \hbox{ and } \qquad  1- \mu_- = \exp_{\stu} (-\nu)
\end{equation}
where the multiplication in $\exp_{\stu}$ is taken with respect to the stuffle multiplication $(\ref{defstuprod})$.
One  verifies that they are given explicitly by
\begin{eqnarray} 
\mu_-& = & (x_1^{-1},0,0, \ldots )  \nonumber \\
\mu_+ & = & (x_1^{-1},(x_1x_2)^{-1},(x_1x_2x_3)^{-1},  \ldots ) \nonumber 
\end{eqnarray}
Given any element $\rho\in \Or$, define the conjugate element
\begin{equation} \label{defrhostartwist}
\rho^{\star}= (1 + \mu_{+} ) \studot \, \rho\, \studot (1 - \mu_{-}) \ . 
\end{equation}
Since the space of solutions to the inhomogeneous stuffle equations modulo products is a Lie algebra with respect to the bracket $[f,g]_{\stu} = f \studot g - g \studot f$, and since $\nu$ is such a solution,  it follows
from the definition $(\ref{defrhostartwist})$ that
$\rho\in \Or$ is a solution to the  stuffle equations modulo products if and only if 
$\rho^{\star}\in \Or$ is.

\subsection{Lifting  solutions to the linearized stuffle equations} 
The lift $\ell(f)$ involves non-trivial denominators, whereas the elements $\Psi_{2n+1}$ are defined over $\Z$. For this reason, 
the following elements are much more useful for our present purposes.

\begin{lem} \label{propD1isstuff}  Let $f \in R[x_1^{-1}, x_1]]$. Then the  element $\widetilde{f}$ defined by 
\begin{equation} \label{tilde1def}
 \widetilde{f}^{(r)} (x_1,\ldots, x_r) = {f(x_1) \over (x_2-x_1) \ldots (x_r-x_1)} 
 \end{equation}
for  $r\geq 1$, is a solution to the stuffle equations modulo products. 
\end{lem}
\begin{proof} It suffices to prove that the element $p$ whose depth-$r$ component is
$$p^{(r)} = { 1 \over (t- x_1) (x_2-x_1) \ldots (x_r-x_1)}$$
is a solution to the stuffle equations modulo products. This can be proved by induction using the recursive definition of the equations $(\ref{Stuffequationsdefn})$. For this, note that
in the second line of  the equation $(\ref{Stuffequationsdefn})$, the terms corresponding to $s_1$ in the right-hand side have poles at $t=x_1$ and those corresponding to $s_{i+1}$ at $t=x_{i+1}$. It is enough to show that 
$$ s_1 p^{(r-1)}  (\x_{2} \ldots \x_i \stu \x_{i+1} \ldots \x_r )  +   { 1  \over x_1 -x_{i+1} }s_1  p^{(r-2)} ( \x_{2} \ldots \x_i \stu \x_{i+2} \ldots \x_r ) =0 $$
and a similar equation with $s_{i+1}$ holds by symmetry. Since each term in the previous equation has a simple pole at $t=x_1$, and is translation-invariant in $x_1,\ldots, x_r$, 
 we can  take the residue at $x_1=t$, and put $t=0$. Thus the previous equation is equivalent to
 \begin{equation} \label{propD1isstuffinproof}
 \mu_1^{(r-1)}(x_2\ldots x_i \star x_{i+1} \ldots x_r) - { 1 \over x_{i+1}} \mu_1^{(r-2)}(x_2\ldots x_i \star x_{i+2} \ldots x_r)=0
 \end{equation}
 where  $\mu_1^{(k)} =  (x_1 \ldots x_k)^{-1}$ was defined in $( \ref{muidef})$. Since, by the previous section, $\mu_1$ is group-like, it satisfies
   $\mu_1^{(r-1)}(x_1\ldots x_i \star x_{i+1} \ldots x_r) = \mu_1(x_1,\ldots, x_r),$
   and therefore $(\ref{propD1isstuffinproof})$ is trivially verified. This completes the proof that $p$ is a solution to the stuffle equations modulo products.
  
  After  expanding $p$ as a power series in $t^{-1}$,we obtain the components of $\widetilde{f}^{(r)}$ where $f= x_1^{n}$. 
  We deduce using  \S\ref{trivialdensity} that $\widetilde{f}$ is a solution to the stuffle equations modulo products for all Laurent series in $x_1$.
\end{proof}

\begin{defn} Let $a = (\ell_1,\ldots, \ell_r)$ be a  partition of $n$ into $r$ subsequences, where
$\ell_i$ are disjoint sequences of consecutive integers such that $\{1,\ldots, n\} = \ell_1 \cup \ldots \cup \ell_r$. 
Define the set of initial terms $it(a) = (i_1,\ldots, i_r)$ to be the set of elements $i_k = \min_{j} \{j \in \ell_k\}$. We shall only consider partitions which are increasing, i.e., with 
$i_1 <i_2<\ldots <i_r$. 
\end{defn} 

Let $f(x_1,\ldots, x_s)$ be a rational function of $s$ variables. Define
\begin{equation} \label{tildefdef} 
\widetilde{f}^{(d)} = \sum_{a} d_{a} f(x_{i_1},\ldots, x_{i_s})
\end{equation} 
where the sum is over partitions of $d$ into $s$  subsequences  and $(i_1,\ldots, i_r) = it(a)$, and
$$d_{a} = \prod_{k=1}^r  { 1 \over x_{\ell_k \backslash i_k, i_k}}\ .$$
When $s=1$, the definition $(\ref{tildefdef})$ reduces to  $(\ref{tilde1def}).$ 

The following proposition gives a mechanism for splitting the depth filtration, i.e., for lifting solutions to the linearized stuffle equations in depth $s$ to solutions to the full stuffle equations modulo products in all higher depths.
\begin{prop} \label{proplift2}  Let $f \in \Or_s$ be a solution to the linearized stuffle equations in $s$ variables. 
Then $\widetilde{f}$ satisfies the stuffle equations modulo products.
\end{prop}

\begin{proof} 

Consider $s$ power series  in one variable $f_1,\ldots, f_s \in \Q[[x_1]]$. By lemma  \ref{propD1isstuff}, the elements $\widetilde{f_i}$ are all solutions to the  stuffle 
equations modulo products. These  correspond    to primitive elements in $\Ostu$. Since the latter
form a Lie algebra with respect to concatenation, it follows that the bracket
\begin{equation}\label{gasbracket}
g=[\widetilde{f}_1, [Ê\widetilde{f}_2, \ldots  [\widetilde{f}_{s-1}, \widetilde{f}_s]_{\star} \cdots ]_{\star} 
\end{equation}
is also primitive and hence a solution to the stuffle equations modulo products. Here the bracket is defined by $[f,g]_{\star} = f \studot g - g \studot f$, where $\studot$ was 
defined in $(\ref{defstuprod})$. By definition of the projector $(\ref{lambdaprojector})$, the first non-zero component $g^{(s)}$ of $g$ is 
\begin{equation}  \label{gs}
f(x_1,\ldots, x_s)= \lambda_s  \, (f_1(x_1) f_2(x_2) \ldots f_s(x_s))\ ,
\end{equation} 
and satisfies the linearized stuffle equations.  By unravelling $(\ref{gasbracket})$, we see that 
$$g  = \widetilde{h}$$
and its components in depth $d$ are linear combinations of $f(x_{i_1},\ldots, x_{i_s})$, where $1\leq i_1,\ldots, i_s \leq d$,  with rational coefficients. 
Thus the proposition is true whenever $f$ is of the form $(\ref{gs})$.  But by the argument of \S\ref{trivialdensity}, the proposition is true for any
$f= \lambda_s (h)$, where $h\in \Or_s$ is a rational function. By  corollary  \ref{corDSto2eq}, any solution  to the linearized stuffle equations is of this form, which completes the proof.
\end{proof}

\begin{rem} Although we shall not require it, we can use the elements $\ell(f)$ in a completely analogous way to split the depth filtration in the stuffle Hopf algebra $\Ostu$.
For this, define  a marked partition of $n$ into $r$ sequences to be  a  pair 
$$a = ((\ell_1,\ldots, \ell_r), (i_1,\ldots, i_r))$$ where
$\ell_i$ are disjoint sequences of consecutive integers such that $\{1,\ldots, n\} = \ell_1 \cup \ldots \cup \ell_r$,   $i_k \in \ell_k$ for all $1 \leq k \leq r$, and $i_1<i_2 < \ldots <i_r$.  Let $mp(n,r)$ denote the set
of marked partitions of $n$ into $r$ subsequences. It is empty if $r>n$.

Now let $f(x_1,\ldots, x_s)$ be a rational function of $s$ variables. Define
$$
\ell(f)^{(d)} = \sum_{\alpha\in mp(d,s)} c_{\alpha} f(x_{i_1},\ldots, x_{i_s})
$$
where $\alpha = (( \ell_1,\ldots, \ell_s), (i_1,\ldots, i_s))$  and $c_{\alpha}$ is defined by
$$c_{\alpha} = \prod_{k=1}^s  { 1 \over |\ell_k| x_{\ell_k \backslash i_k, i_k}}\ .$$
When $s=1$, this definition  reduces to  $(\ref{lfddef}).$ Then we show, in an identical way to proposition \ref{proplift2} that if $f(x_1,\ldots, x_s)$ satisfies the linearized stuffle equations, then 
$\ell(f)$ is a solution to the stuffle equations modulo products in all depths.
\end{rem}

\subsection{Internal structure of the elements $\Psi_{2n+1}$} \label{sectABCdef} The elements $\Psi_{2n+1}$ break up into three distinct pieces, each of which individually satisfies the stuffle equations modulo products. 
Let us denote this decomposition by
$$\Psi_{2n+1} ={1\over 2} \big( A_{2n+1} + B_{2n+1} + C_{2n+1}\big)$$
where
\begin{eqnarray}
A_{2n+1}^{(d)} & = & {x_1^{2n} \over (x_2-x_1) \ldots (x_d-x_1) } \nonumber \\
B_{2n+1}^{(d)}  & = &  \Big(\sum_{i=1}^d  { x_d^{2n} \over x_{\{1,\ldots, i-1\},\{0\}}\, x_{\{i,\ldots,d-1\},\{d\}}}Ê\Big) -\Big(\sum_{i=1}^{d-1} { x_{d-1}^{2n} \over x_{\{d, 1,\ldots, i-1\}, \{0\}} \, x_{\{i,\ldots, d-2\}, \{d-1\}}} \Big)  \nonumber \\
C_{2n+1}^{(d)}  & = & \Big(  \sum_{i=2}^d { (x_i - x_{i-1})^{2n} \over x_{\{0,\ldots, i-2\},\{i-1\}}\, x_{\{i+1,\ldots,d\},\{i\}}} \Big) +  \Big(\sum_{i=1}^{d-1}  { (x_1-x_d)^{2n} \over x_{\{2,\ldots, i\},\{ 1\} } \, x_{\{i+1,\ldots, d-1,0\}, \{d\} }}   \Big) \nonumber
\nonumber 
\end{eqnarray}
We immediately see that  $A_{2n+1}$ is the lift $\widetilde{f}$, where $f(x_1) = x_1^{2n}$.   Let   $\tau(A_{2n+1})$ denote the  involution   $(\ref{stuffinvol})$ applied to $A_{2n+1}$.
It is given by 
$$\tau(A_{2n+1})^{(d)}  =  {x_d^{2n} \over (x_1-x_d) \ldots (x_{d-1}-x_d) }$$
and satisfies the stuffle equations modulo products.
It is easily verified that
$$B_{2n+1} =(\tau( A_{2n+1}))^{\star}\ ,$$
since the term on the left in the above formula for $B^{(d)}_{2n+1}$ is $(1+ \mu_+) \studot \tau( A_{2n+1})$, and the term on the right is given by $(1+ \mu_+) \studot \tau( A_{2n+1})\studot \mu_-$.

\begin{prop} The elements $C_{2n+1}$ satisfy the stuffle equations modulo products.
\end{prop}

\begin{proof} 
Let $f_1,f_2 \in \Q[x_1]$ be any two polynomials and consider the element
$$p= [\widetilde{f}_1, \tau(\widetilde{f_2})]_{\stu} $$
where $\tau$ is the  involution $(\ref{stuffinvol})$, and $[\,,\, ]_{\stu}$ is the Lie bracket with respect to $\studot$.  The element $p$ is a solution to the stuffle equations modulo products by lemma \ref{propD1isstuff}.
Let  $p=p_+-p_-$, where   $p_+= \widetilde{f}_1 \studot  \tau(\widetilde{f_2}) $ and $p_-=\tau( \widetilde{f}_2 ) \studot  \widetilde{f_1} $, 
giving rise to  the formulae
$$ p_-{(d)}=  \sum_{i=2}^d { f_2(x_{i-1})f_1(x_i)\over x_{\{1,\ldots, i-2\},\{i-1\}}\, x_{\{i+1,\ldots,d\},\{i\}}}  \quad ,  \quad   p_+^{(d)}= \sum_{i=1}^{d-1}  { f_1(x_1)f_2(x_d) \over x_{\{2,\ldots, i\},\{ 1\} } \, x_{\{i+1,\ldots, d-1\}, \{d\} }}   \ .$$
By taking linear combinations of such elements, we can replace $f_1(x_a) f_2(x_b)$ in the above expressions for $p_{\pm}$ by ${1 \over x_b} (x_a-x_b)^{2n}$. This gives precisely 
  $ C^{(d)}_{2n+1}$. 
\end{proof} 

\begin{cor} The elements $\Psi_{2n+1}$ for all $n\geq 1$, satisfy the stuffle equations modulo products.
\end{cor}

\subsection{Internal structure of the  element $\Psi_{-1}$}
 The element $\Psi_{-1}$ splits into infinitely many components, each of which individually satisfies the stuffle equations modulo products. 
Let $\mathcal{V}^n_d$ denote the set of vines  $v \in \mathcal{V}_d$  of height $n$.    Define
 \begin{equation} \label{Pdef}
 P^{(d)}_n = \sum_{v\in \mathcal{V}^n_d }    { 1\over x_v x_d}\ .
 \end{equation}Ê
Then we have $$
 \Psi_{-1} = 2 \sum_{n\geq 1} (-1)^n {P_n \over n}\ .
 $$
 and it suffices to show that each $P_n$ satisfies the stuffle equations modulo products.
 \begin{lem} \label{lemP1isstuff}The element $P_1$ satisfies the stuffle equations modulo products.
 \end{lem}
 \begin{proof}
 We work in the Hopf algebra $\Ostu$. Consider the element $\kappa \in \Ostu$ defined by
  $$\kappa=(x_1^{-2}, 0, 0, \ldots )\ .$$
 It is almost immediate from the formula
  $$P_1^{(d)} = {1 \over x_1  x_2 \ldots x_{d-1} x_d^2}\quad  \ ,\quad d \geq 1$$
 that $P_1 = (1+ \mu_+) \studot \kappa$. One easily verifies the depth two stuffle equation:
 $$\kappa(x_1 \star x_2) = -{ 1\over x_1^2 x_2} - { 1 \over x_1 x_2^2} = (- \mu_- \studot \kappa - \kappa \studot \mu_-)^{(2)} \ .$$
Clearly, all higher stuffle equations for $\kappa$ vanish. Thus $\kappa$ corresponds to an element satisfying
$\Delta(\kappa) = (1- \mu_-)\otimes \kappa + \kappa \otimes ( 1- \mu_-)$. Since $1+ \mu_+$ is group-like and its inverse is $1- \mu_-$, 
we immediately deduce that $(1+\mu_+) \studot \kappa$ is primitive (in any Hopf algebra, given an element 
$b$ such that $\Delta(b) = g^{-1} \otimes b + b \otimes g^{-1}$, where $g$ is group-like, the element $gb$ is primitive). Therefore $P_1$ satisfies the stuffle equations modulo products.
  \end{proof}

 Observe that since there are no vines of height $n$ with fewer than $n$ grapes,  $P^{(d)}_n=0$ for $d<n$. Since the unique vine with $n$ grapes and height $n$ is given by 
 $g_1\ldots g_1$, it follows that the first non-zero component $P^{(n)}_n$ is the  following element  $c_n$.  

\begin{lem}  \label{lemcnisstuff} For all $n\geq 1$, the cyclically invariant  elements
\begin{equation} \label{cndef}
c_n (x_1,\ldots, x_n) =  { 1 \over x_1 (x_2-x_1) \ldots (x_n-x_{n-1})x_n}\  
\end{equation}
are solutions to the linearized stuffle equations. 
\end{lem}

\begin{proof} 
By corollary \ref{corDSto2eq},   it suffices to verify
 $(\ref{pinfstuff})$. 
The equation  $\lambda_n c_n = n c_n$ can be verified by taking residues and using the formula 
\begin{equation}   \Res_{x_i =x_{j}} \, c_n    =\left\{
                           \begin{array}{ll}
                            \pm c_{n- 1} (x_1,\ldots ,\widehat{x_i}, \ldots, x_n)\  & \hbox{if } |i-j|=1  \ ,  \\
                                                         0\   & \hbox{otherwise }  
                           \end{array}
                         \right.
                     \end{equation}
  in which  we use the convention $x_0=0, x_{n+1}=0$.
\end{proof}
 It follows from  lemma  $\ref{lemcnisstuff}$ and proposition \ref{proplift2} that the elements $\widetilde{c}_n$ are solutions to the stuffle equations modulo products for all $n$.
 Now define
 $$Q_n = ( 1- \mu_-) \studot P_n\ .$$
 It is straightforward to verify that 
 \begin{equation}
Q_n^{(d)} = \sum_{v\in {}_1\!\mathcal{V}^n_{d} }    { 1\over x_v x_d}\ .
 \end{equation} 
 where ${}_1\!\mathcal{V}^n_{d} $ denotes the set of all vines $v\in \mathcal{V}^n_{d} $  which are of the form $v=g_1 w$.
 \begin{prop} The following equation holds for all $n\geq 1$:
 \begin{equation} \label{cnasdifferenceofQs}
 \widetilde{c}_n \studot ( 1 - \mu_-) = Q_n + Q_{n+1} \ .
 \end{equation}
 \end{prop}
  \begin{proof}
  First of all, observe that the elements  $\widetilde{c}_n$ are given explicitly by 
 $$\widetilde{c}^{(d)}_n = Q_n^{(d)} +  R_n^{(d)}\ ,$$
where we define $R_n$ by the formula
$$R_n^{(d)} =  \sum_{v\in {}_1\!\mathcal{V}^{n+1}_{d}} x_v^{-1}  x^{-1}_{f(v)}\ .$$
Here,  we define $f(w g_k)=d-k\in \{n,\ldots,d-1\}$ for any word $w$ in $\{g_1,g_2,\ldots, \}$.
The previous formula follows easily from the definition of $x_v,$ formula $(\ref{tildefdef})$, and the fact that $c^{(d)}_n= x_{g_1\ldots g_1} x^{-1}_d$. 
The equation $(\ref{cnasdifferenceofQs})$  is then equivalent to 
$$
R_n \studot ( 1 - \mu_- ) = Q_n \studot \mu_-  +  Q_{n+1}
$$
Each term in this equation is a sum over certain vines $v$.  It decomposes into the following pair of identities, valid for every vine $v$:
\begin{eqnarray}
(x_{v g_m} x_{d-m})^{-1} - (x_{v g_{m-1}} x_{d-m}x_d)^{-1}  & = &  (x_{vg_m}x_d)^{-1}  \qquad \qquad  \hbox{ if } m\geq 2 \nonumber \\
(x_{v g_1} x_{d-1})^{-1} & = &  (x_{v}x_{d-1}x_d)^{-1}  + (x_{vg_1}x_d)^{-1}  \nonumber 
\end{eqnarray} 
For example, the first equation is equivalent,  after multiplying  through by $x_{vg_{m-1}}$,  to  the obvious identity
$(x_{d}-x_{d-m})^{-1}x_{d-m}^{-1}- (x_{d-m}x_d)^{-1} =  (x_{d}-x_{d-m})^{-1} x_d^{-1}$.
  \end{proof} 
  
 \begin{prop} Each  $P_n$   is a solution to the stuffle equations modulo products.
  \end{prop} 

\begin{proof}
Multiply   equation $(\ref{cnasdifferenceofQs})$  on the left by $(1+ \mu_+)$. This gives
$$(\widetilde{c}_n )^{\star}= (1+ \mu_+) \studot (Q_n + Q_{n+1})  =P_n +P_{n+1} \ .$$
The left-hand side is a solution to the stuffle equations modulo products, by lemma \ref{lemcnisstuff} and  proposition \ref{proplift2}, and hence by induction on $n$ and lemma  \ref{lemP1isstuff},  $P_n$ is a solution to the stuffle equations modulo products.
 \end{proof}

 \begin{cor}  $\Psi_{-1}$   is a solution to the stuffle equations modulo products.
 \end{cor}

\section{Proofs of shuffle equations}

\subsection{Some properties of the shuffle equations}
The following lemma follows immediately from the definitions of the shuffle equations.
\begin{lem} \label{orbitlemma} Let $g \in \Or_n$, and set $f= g  x_n^m$ for some $m\in \Z$. Then
$$f^{\sharp}(x_1 \ldots x_k \sha x_{k+1} \ldots x_n) = g^{\sharp}(x_1 \ldots x_k \sha x_{k+1} \ldots x_n) (x_1+\ldots +x_n)^m$$
In particular, $f$ is a solution to the shuffle equations modulo products if and only if $g$ is.
\end{lem} 
We first compute the shuffle equations on $x_{g_n}$, for a single grape.

\begin{lem}  The elements $g_n = (x_1\ldots x_n)^{-1}$ satisfy the following identity:
\begin{equation}\label{gnshuffle} 
g_n^{\sharp} (x_1\ldots x_k \sha x_{k+1} \ldots x_n) = g^{\sharp} _k(x_1,\ldots, x_k) g^{\sharp} _{n-k}(x_{k+1},\ldots, x_n)\ .
\end{equation} 
\end{lem} 
\begin{proof}
Let $f_{n-1} \in \Or_{n-1}$, and define $f_n= f_{n-1} x_n^m$, for any $m\in \Z$. Then 
\begin{multline}  \nonumber 
f_n^{\sharp} (x_1\ldots x_k \sha x_{k+1} \ldots x_n) = \big( f_{n-1}^{\sharp} (x_1\ldots x_{k-1} \sha x_{k+1} \ldots x_n) \\
 +   f_{n-1}^{\sharp} (x_1\ldots x_{k} \sha x_{k+1} \ldots x_{n-1})  \big)(x_1+\ldots + x_n)^{m}
\end{multline} 
Equation $(\ref{gnshuffle})$ is proved by induction on $n$.  Applying the previous identity with $f_{n-1} =g_{n-1}$, and $m=-1$. Using the notation
$x_{[i,j]}= x_i + x_{i+1}+ \ldots +x_j$ for $i< j$, 
\begin{multline}  \nonumber 
g_n^{\sharp} (x_1\ldots x_k \sha x_{k+1} \ldots x_n) = \Big( {1 \over  x_1 x_{[1,2]}\ldots x_{[1,k-1]}  x_{k+1} x_{[k+1, k+2]} \ldots x_{[k+1,  n]})} \\
 +  {1 \over  x_1 x_{[1,2]}\ldots x_{[1,k]}  x_{k+1} x_{[k+1, k+2]} \ldots x_{[k+1,  n-1]})} \Big) { 1 \over x_{[1,n]}}
\end{multline} 
which immediately reduces to $(\ref{gnshuffle})$ on noting that $x_{[1,n]}= x_{[1,k]}+x_{[k+1,n]}$.
\end{proof}

\subsection{Shuffles and vines}
\begin{prop} \label{propvinestoosha}  The rational realisation map
\begin{eqnarray}
p: \mathcal{V}  &\To & \Osha \\
v & \mapsto & x^{-1}_v \nonumber 
\end{eqnarray} 
is a morphism of Hopf algebras.
\end{prop} 
\begin{proof} This map  is a homomorphism for the multiplication by definition.  To check that it respects the coproducts, it suffices to compute the coproduct on 
the generators $g_n$.  Since $\Delta(g_n) = \sum_{i+j= n} g_i \otimes g_j$, this precisely reduces  to  $(\ref{gnshuffle})$ via the definition of the Hopf algebra structure on $\Osha$. 
\end{proof}Ê
In particular, primitive elements in $\mathcal{V}$ map to solutions of the shuffle equations modulo products.
\begin{prop} For all $d\geq 1$, the depth $d$ component of  $\Psi_{-1}$ is a solution to the shuffle equations modulo products.
\end{prop}

\begin{proof} By the proof of lemma $(\ref{lembadlift1})$, the element
$$\beta= \sum_{i_1+\ldots +i_k=d} {(-1)^k \over k} g_{i_1}\ldots g_{i_k}$$
is a primitive vineyard. Therefore $p(\beta)$ is a solution to the shuffle equations modulo products.    By  lemma \ref{orbitlemma},  so too is the element
$$\Psi_{-1}^{(d)}= {p(\beta)\over x_d} \ .$$
\end{proof}

\subsection{Conjugation by grapes $g_n$}
Let $S$ denote the antipode in the Hopf algebra of vines. 
It corresponds, in the Hopf algebra of rational functions, to  the antipode $\sigma$ defined in $(\ref{sigmaantipodedef})$, by proposition $\ref{propvinestoosha}$.
 We therefore have
$$p_{S(g_d)}= \sigma(p_{g_d}) = {(-1)^d \over x_d (x_d-x_1) \ldots (x_d-x_{d-1})}\ .$$
Consider the two generating series
$$G= \sum_{n\geq 1} g_n\quad \hbox{ and } \quad S(G)= \sum_{n\geq 1} S(g_n)\ .$$
Since $G$  is group-like, $\Delta G = G \otimes G$, we have  $S(G) = G^{-1}$.  It follows that if $\alpha$ is any primitive element in the Hopf algebra of vines, then 
$S(G) \cdot \alpha \cdot G$ is also  primitive.

\subsection{Internal structure of  $\Psi_{2n+1}$}
The elements $\Psi_{2n+1}$ again break up into three distinct pieces, each of which individually satisfies the shuffle equations:
$$\Psi_{2n+1} ={1\over 2} \big( D_{2n+1} + E_{2n+1} + F_{2n+1}\big)\ ,$$
where (with the convention $x_0=0$),
\begin{eqnarray}
D_{2n+1}^{(d)} & = &  \sum_{i=1}^d  { x_d^{2n} \over x_{\{1,\ldots, i-1\},\{0\}}\, x_{\{i,\ldots,d-1\},\{d\}}}  \nonumber \\
E_{2n+1}^{(d)}  & = &  \sum_{i=1}^d { (x_i - x_{i-1})^{2n} \over x_{\{0,\ldots, i-2\},\{i-1\}}\, x_{\{i+1,\ldots,d\},\{i\}}}Ê   \nonumber \\
F_{2n+1}^{(d)}  & = &  \sum_{i=1}^{d-1}   \Big( { (x_1-x_d)^{2n} \over x_{\{2,\ldots, i\},\{ 1\} } \, x_{\{i+1,\ldots, d-1,0\}, \{d\} }}   -  { x_{d-1}^{2n} \over x_{\{d, 1,\ldots, i-1\}, \{0\}} \, x_{\{i,\ldots, d-2\}, \{d-1\}}} \Big) \nonumber
\nonumber 
\end{eqnarray} 
We immediately notice that  $E_{2n+1} = S(G) \cdot \alpha_n \cdot G$,  where $\alpha_n \in \Or$ is $x_1^{2n}$ in depth $1$, and zero in all other depths.
By the above remarks, this element is  primitive, and therefore $E_{2n+1}^{(d)}$ satisfies the shuffle equations for all $d$. Next, we see that 
$$D^{(d)}_{2n+1} =   (G\cdot \alpha_0 \cdot S(G))  \, x_d^{2n} $$
which, by  lemma \ref{orbitlemma} satisfies the shuffle equations since $G\cdot \alpha_0 \cdot S(G) = G\cdot \alpha_0 \cdot G^{-1}$ is  primitive in $\Osha$. 
To construct $F$ from $D$, we need the following lemma.
\begin{lem} If $f\in \Or^d$ satisfies the shuffle equations,  then so does
\begin{equation}  \label{fdflipped}
f(x_1,\ldots, x_d)  + (-1)^d f(x_{d+1} -x_d, \ldots, x_{d+1}- x_1)
\end{equation}
\end{lem}
\begin{proof}  Since $x_1^0$ (the element $1$ viewed in $\Or_1$) and $f$ are primitive elements of $\Osha$,  so  too is the element $ f \cdot x_1^0 -  x_1^0 \cdot f \in \Osha_{d+1}$.
Because $f$ is primitive, this is again equal to  $f \cdot x_1^0+ x_1^0 \cdot \sigma(f) $, where $\sigma$ is given by  $(\ref{sigmaantipodedef})$. This is precisely 
$(\ref{fdflipped})$.
\end{proof} 
Finally, it follows easily from the explicit formulae above that 
$$F_{2n+1}^{(d)} =  - \big(D_{2n+1}^{(d-1)}(x_1,\ldots, x_{d-1}) + (-1)^{d-1} D_{2n+1}^{(d-1)}(x_{d+1}- x_d,\ldots,x_{d+1}- x_1)  \big) { 1\over x_d}$$
This is the product of $x_d^{-1}$ with an element which satisfies the shuffle equations, by the previous lemma. Therefore, by  lemma \ref{orbitlemma}, $F_{2n+1}$ is also a solution to the shuffle equations.
We conclude 
\begin{cor} The elements $\Psi_{2n+1}$ satisfy the shuffle equations.
\end{cor} 

\subsection{Proof of lemma \ref{lemweightzeroofpls}} \label{sectproofoflemweightzero}
By definition, $W_{-1} \p \ls \subset \bigoplus_{n\geq 1} c_n(x_1,\ldots, x_n) \Q$, where $c_n$ is defined by $(\ref{cndef})$. It is equal to 
$(x_{v} x_n)^{-1}$, where $v$ is the vine $g_1^n$. By lemma \ref{orbitlemma} and proposition \ref{propvinestoosha}, it satisfies 
the shuffle equations if and only if $v$ is primitive in the Hopf algebra of vines. This occurs only if $n=1$. Now consider an element  $\xi\in\p \ls$ of weight
zero and depth $n$. It is of the form $\xi=(\sum_{i=1}^n \alpha_i x_i) c_n(x_1,\ldots, x_n)$, where $\alpha_i \in \Q$. Since $c_n$ is
cyclically-invariant, we deduce that the numerator $\sum_{i=1}^n \alpha_i x_i$
must also be  cyclically-invariant for $\xi$ to be in $\p\ls$. This forces all $\alpha_i =0$.

\section{Structure of residues} \label{sectResiduesFinal}
\subsection{Cancellation of double poles}

\begin{prop}  \label{propCancelofpoles} Let $\xi \in \Lo$ be a homogeneous element of  non-negative weight. Then the depth $d$ component  $\xi^{(d)}\in \Or_d$  has at most simple poles
for all $d$. 
\end{prop}

\begin{proof} Consider  rational functions $f\in \Or_d$ such that:
\begin{enumerate}
\item $f$ has at most simple poles along $x_i=0$ for $i=1,\ldots, d-1,$
\item  $f$ has at most a double pole along $x_d=0$.
\end{enumerate} 
We first show that  these properties are stable under $\{, \}$. For this, let $f_1 \in \Or_p, f_2 \in \Or_q$ satisfying $(1)$ and $(2)$. In the formula $(\ref{circformula})$
we see  that $f_1 \circb f_2$ has possible double poles at $x_p=0, x_q=0$, coming from a single term each, and possible double poles along $x_{p+q}=0$. All other $x_i=0$ have at most simple poles.   
The total contribution from the two terms in $\{f_1,f_2\}$ which give possible double poles at $x_p=0$ are
$$f_1(x_1,\ldots, x_p) \big(f_2(x_{p+1},\ldots, x_{p+q}) -f_2(x_{p+1}-x_p,\ldots, x_{p+q}-x_p) \big)\ . $$
Since the right-hand factor vanishes at $x_p=0$, the  double pole in fact cancels.

Now it follows from their definitions that the elements $\psi_{2n+1}^{(d)}$ and $\psi_{-1}^{(d)}$ satisfy $(1)$ and $(2)$ for all $d$. Therefore every depth $d$ component 
$\xi^{(d)}$ of  any element $\xi \in \Lo$  satisfies $(1)$ and $(2)$.  Now suppose that $\xi$ has non-negative  weight. In particular, $\xi^{(1)}= x_1^n$ for some $n\geq 0$ and has no poles.
We prove by induction on $d$ that $\xi^{(d)}$ has no  double poles. For this, notice that if $\xi^{(1)}, \ldots, \xi^{(d-1)}$ are in $\Or_1,\ldots, \Or_d$ and have no double poles, all the terms of depth $\leq d-1$ in the stuffle equations modulo products  for $\xi^{(d)}$ have no double poles either. The double poles of $\xi^{(d)}$ therefore satisfy the linearized double stuffle equations and are  stable under the dihedral group of  lemma \ref{lemfisdihed}. Since all divisors $x_i=x_j$ are in the orbit of some  divisor of the form $x_i=0$ for $1\leq i \leq d-1$,
it follows from property $(1)$  that  $\xi^{(d)}$ has no double poles at all.
\end{proof}

\subsection{Formulae for the residues of $\Psi_{\bullet}$}  A crucial property of our chosen generators $\Psi_{2n+1}$ is that their residues have the following, very particular,  structure.

\begin{prop} \label{propResstructure} Let $n=-1$ or $n\geq 1$.  Then  for all $i<d$, we have 
\begin{equation}Ê\label{Resstructure}
\Res_{x_i=0} \Psi_{2n+1}^{(d)} = { 1 \over x_1 \ldots x_{i-1} }\,   \studot \,    \Res_{x_1=0} \Psi_{2n+1}^{(d-i+1)} \ .
\end{equation}

\end{prop} 

\begin{proof} Consider first the case $n=-1$.  Let $v$ be a vine, and let $m\geq 1$. It follows immediately from the definition of the rational realization $(\ref{xofvine})$  of a vine that
$$\Res_{x_i=0}  \, x^{-1}_{g_mv}=  \begin{cases}
     0 & \text{if } \quad m< i \ , \\
  { 1\over x_1\ldots x_{i-1}}  \, \studot  \,  1 \studot \,  x^{-1}_{g_{m-i} v} & \text{if } \quad m\geq i\ ,
  \end{cases}
$$
where we write $g_0 v=v$.  Equation $(\ref{Resstructure})$ follows easily from  definition \ref{Psi-1def}. 

For the case $n\geq 1$,  recall the decomposition of $\Psi_{2n+1}$ into pieces  $A,B,C$ defined in \S \ref{sectABCdef}.
The components of $A$  have no poles along $x_i=0$. From its definition,  
\begin{equation}Ê \label{Cresi}
\Res_{x_i=0}\,  C_{2n+1}^{(d)} = { 1\over x_1 \ldots x_{i-1}} \, \studot  \, 1 \, \studot \,   A^{(d-i)}_{2n+1}
\end{equation} 
whenever $i<d$. The main contribution comes from $$B_{2n+1}= (1+ \mu_+) \, \studot \,  b_{2n+1}\ ,$$
 where  we set $b_{2n+1}=\tau(A_{2n+1}) \studot (1- \mu_-)$. It satisfies $\Res_{x_i=0} b^{(d)}_{2n+1}=0$ for all $i< d$. 
 It follows immediately from the definition of $\mu_+$ that 
 \begin{equation} \label{Bresi}
 \Res_{x_i=0} B^{(d)}_{2n+1} = { 1 \over x_1 \ldots x_{i-1}} \, \studot \, 1 \, \studot \,  b^{(d-i)}_{2n+1}
 \end{equation}
 for all $i<d$. Therefore $(\ref{Resstructure})$ follows from $(\ref{Cresi})$ and $(\ref{Bresi})$.
\end{proof}

A more detailed analysis in the proof of proposition \ref{propResstructure} leads to the following formulae for the residues. We shall omit the proofs, which are straightforward.
\begin{prop}  Let $n\geq 1$. Then 
\begin{eqnarray} \Res_{x_1=0} \Psi_{2n+1}^{(d)} & =  & { 1 \over 2}  \, \studot \, (B_{2n+1}^{(d-1)} - A_{2n+1}^{(d-1)}) \ ,\\
 \Res_{x_d=0} \Psi_{2n+1}^{(d)} & =  & { 1 \over 2}  \, \big[ A_{2n+1} \, \studot \, \mu_+ - \mu_+  \, \studot \,  \tau(A_{2n+1})\big] \ , \nonumber \\
 \Res_{x_d=0} (\Psi_{-1}^{(d)}x_d ) & =  & ( x_1 \ldots x_{d-1})^{-1}\ .\nonumber 
 \end{eqnarray} 
 where $A_{2n+1},B_{2n+1}$ were defined in \S \ref{sectABCdef}.
\end{prop}

\subsection{General residue structure}
The residue structure for the elements $\Psi_{2n+1}$ holds more generally for any commutators in the elements $\Psi_{2n+1}$.

\begin{thm} \label{thmGenResstructure} Let  $\xi \in \Lo$.   Then  for all $i<d$, we have 
\begin{equation}Ê\label{GeneralResstructure}
\Res_{x_i=0} \, \xi^{(d)} = { 1 \over x_1 \ldots x_{i-1} }\,   \studot \,    \Res_{x_1=0} \,\xi^{(d-i+1)} \ .
\end{equation}
\end{thm} 
Before proving the theorem, we state the following important technical result.
Let  $\alpha = 1 + \mu_+$ denote the element which in depth $k$ is 
$$\alpha^{(k)} = (x_1 \ldots x_k)^{-1}\ .$$

\begin{thm} \label{thmxicircalpha} Let $\xi\in \Or$ satisfy the double shuffle equations mod products. Then
\begin{equation} 
\xi \circb \alpha \, +  \, \alpha \, \studot\,  \widetilde{\xi}=0 \nonumber \ , 
\end{equation}
where $\widetilde{f}^{(d)}(x_1,\ldots, x_d) = (-1)^d f^{(d)}(-x_d,-x_{d-1}, \ldots,- x_1)$.
\end{thm} 
The proof of this theorem is postponed to \S\ref{sectproofofalphaid}.

\begin{lem}  Suppose that  $f \in \Or_r$, and $g\in \Or_s$ have at most simple poles. Let $1\leq i< r+s$, and assume that there exist $g_A, g_B $ such that
\begin{eqnarray} \label{resgassumpinlem}
\Res_{x_i=0} \, g &  = & \alpha^{(i-1)} \, \studot \, 1  \, \studot \, g_A \qquad \quad \hbox{ if } \quad i < s \\
\Res_{x_{i-r}=0} \, g & =  &\alpha^{(i-r-1)}\,  \studot \, 1\,  \studot  \, g_B  \qquad \hbox{ if } \quad i  > r \ .\nonumber 
\end{eqnarray} 
Then the residue of $f\circb g$ is a  sum over the following terms:
\begin{eqnarray} \label{GeneralResFormula}
\Res_{x_i=0} (f \circb g) &  =  &  (\Res_{x_i=0} f) \, \studot\,  g \qquad  \qquad \quad \qquad  \qquad \qquad \qquad \qquad \hbox{ if }  \quad i \leq r \nonumber \\
 & + &    (\Res_{x_s=0} g) \, \studot\, 1  \, \studot\, f  \qquad \qquad \qquad \qquad \qquad \qquad \qquad \hbox{ if }  \quad  i =s   \nonumber \\
  & + &   \alpha^{(i-1)} \, \studot\, 1  \, \studot\,  (f \circb g_A) \qquad \qquad \qquad  \qquad \qquad  \qquad \quad \, \,  \hbox{ if }  \quad  i <s   \nonumber \\
    & + &   (f \circb \alpha^{(i-r-1)}  ) \, \studot\,  1  \, \studot\, g_B   +     (\alpha^{(i-r-1)} \, \studot \, \widetilde{f}  ) \, \studot\, 1  \, \studot\, g_B  \quad \quad \, \,  \hbox{ if }  \quad  i >r   
\end{eqnarray}
\end{lem} 
\begin{proof} By inspecting formula $(\ref{circformula})$ for $f\circb g$, we see that none of the terms in $f$ has  a pole along $x_i=0$ except the first term in the sum, which is
$$f(x_1,\ldots, x_r) g(x_{r+1}, \ldots, x_{r+s})\ .$$
This gives rise to the first line  of $(\ref{GeneralResFormula})$. All other contributions 
come from residues in the $g$-arguments.  The only place in which a $\Res_{x_s=0} g$ can occur is from the term
$$ g(x_1,\ldots, x_s) f(x_{s+1}, \ldots, x_{r+s})\ ,$$
which gives rise to the second line of $(\ref{GeneralResFormula})$. All other contributions come from 
a $\Res_{x_i=0} \, g$ when $i<s$  or $\Res_{x_{i-r}=0} \, g$ when $i>r$. By $(\ref{resgassumpinlem})$ and analysing the terms in $(\ref{circformula})$,  we obtain the third and fourth lines of 
$(\ref{GeneralResFormula})$. Note that in the case $i>r$, and $f$ homogeneous,   the final term $  (\alpha^{(i-r-1)} \, \studot \, \widetilde{f}  ) \, \studot\, 1  \, \studot\, g_B$ comes from the single term
$$(-1)^{\deg(f) + r} f(x_{i-1}-x_i,\ldots, x_{i-r}-x_i) \Res_{x_i=0} \, g(x_1,\ldots, x_{i-r-1}, x_i, \ldots, x_{r+s})\ .  $$
\end{proof}

\begin{cor} \label{corbrackethasgoodres} Suppose that $f, g\in \Or$ are solutions to the double shuffle equations modulo products, have at most simple poles,  and satisfy, for some $\Phi_0$, 
\begin{equation}\label{ResformasPhi}
\Res_{x_i=0} \, \Phi^{(d)} = \alpha^{(i-1)} \,   \studot \,    \Phi_0^{(d-i+1)} 
\end{equation} 
for all $i<d$, where $\Phi=f$ or $g$. Then the same equation holds 
for $\Phi = \{f,g\}$.
\end{cor}

\begin{proof} Consider the terms in $(\ref{GeneralResFormula})$ for $ f\circ g$. All terms in the first line of 
$(\ref{GeneralResFormula})$ for $i<r$ are of the required form, by assumption $(\ref{ResformasPhi})$ for $\Phi=f$. All terms in the third line 
of $(\ref{GeneralResFormula})$ are clearly of the required form, and all terms in the fourth line vanish by theorem  $\ref{thmxicircalpha}$ applied to $f$. Thus all remaining terms are
$$ \sum_r  (\Res_{x_r=0} f^{(r)}) \, \studot \, 1 \, \studot \,  g^{(d-r)} + (\Res_{x_r=0} g^{(r)}) \, \studot  \, 1 \, \studot \, f^{(d-r)}\ ,$$
which cancel out in the anticommutator $\{f,g\}$. So $\{f,g\}$ has the desired property.
\end{proof}

Theorem  $\ref{thmGenResstructure}$ follows easily by induction from the previous corollary,  
since the  $\Psi_{2n+1}$ have the requisite
residue structure by  proposition \ref{propResstructure}. We know by theorems \ref{theorem: main1} and  \ref{theorem: main2} that all elements of $\Lo$ satisfy the double shuffle equations modulo 
products. It is  worth noting   that $\Psi^{(d)}_{-1}$ in fact has  a double pole along $x_d=0$. However, proposition \ref{propCancelofpoles} says that $\{\Psi_{-1}, \xi\}$ has at most simple poles, and it is easy to see that corollary $\ref{corbrackethasgoodres}$ applies in the case $f= \Psi_{-1}$ and $g \in \Lo$  also.

\subsection{Vanishing of residues (proof of theorem \ref{thmcycpoles})}
Let $d\geq 2$, and  $\xi \in \Lo$ such that $\xi^{(1)}, \ldots, \xi^{(d-1)}$ has no poles. We wish to show that 
$\xi^{(d)}$ has poles along the main cyclic orbit $x_1=0,\ldots,  x_i= x_{i+1}, \ldots, x_d=0$ only. The statement is vacuous for $d=2$.
Therefore, assume that $d\geq 3$. By theorem \ref{thmGenResstructure}, we have for $2\leq i \leq d-1$,
\begin{equation} \label{resxiiszero}Ê\Res_{x_{i}=0}\, \xi^{(d)} = { 1 \over x_1 \ldots x_{d-1}}\, \studot \, \Res_{x_1=0}\, \xi^{(d-i+1)} =0\ .
\end{equation}Ê
By the assumption on $\xi$, all terms of depth $<d$ in the stuffle equations for $\xi^{(d)}$ are pole-free (hereafter referred to as `polynomial'). It follows that 
\begin{equation} \label{xidstufflemodp}
 \xi^{(d)} (x_1\ldots x_i \star  x_{i+1} \ldots x_{d}) \equiv \xi^{(d)}(x_1 \ldots x_i \sha x_{i+1} \ldots x_{d}) \pmod{ \hbox{polynomials}} \ .
 \end{equation}
 Since $\xi$ satisfies the stuffle equations, the left-hand side vanishes. Since the antipode for the shuffle product is signed reversal of words, we deduce that 
$$\xi^{(d)} + \tau \xi^{(d)}   \equiv 0 \pmod{ \hbox{polynomials}} \ . $$
 Similarly, since $\xi$ satisfies the shuffle equations, it satisfies
 $$\xi^{(d)}+ \sigma{\xi^{(d)}} =0\ .$$
 Now, by \S\ref{sectDihedralsymm}, $\sigma, \tau$ generate a dihedral group. By cyclic symmetry, we deduce from
 $(\ref{resxiiszero})$ that $ \Res_{x_{i}=x_j}\, \xi^{(d)} =0$ whenever $2\leq |i-j|$, which proves the theorem.

\subsection{Proof of theorem    \ref{thmxicircalpha}} \label{sectproofofalphaid}
We shall prove theorem   \ref{thmxicircalpha} in a  different form.
\begin{lem} Let $f\in \Or$. The equation $f \circb \alpha + \alpha \, \studot \, \widetilde{f}$ is equivalent to 
\begin{equation}  \label{eqnfcircmuminus}  (f +  \widetilde{f})  \, \studot \, (1- \mu_-)    = \mu_- \, \studot \, f - f \circ \mu_-\ .
\end{equation}

\end{lem} 
\begin{proof}
By definition of $\alpha$, we have $\alpha = \alpha \,  \studot \, \mu_-  +1$. Therefore,
$$ f \circb \alpha = ( f \circb \alpha) \, \studot \, \mu_- - \alpha \, \studot \, f \, \studot \, \mu_- + \alpha \, \studot \, ( f\circb \mu_-) + f$$
by $(\ref{IharaderivShuffle})$.  Suppose that $f \circ \alpha + \alpha \, \studot \, \widetilde{f} $ vanishes in all components $<d$. Then
the depth $d$ component of $f \circ \alpha + \alpha \, \studot \, \widetilde{f} $ is equal to the depth $d$ component of
$$   - \alpha \, \studot \, \widetilde{f} \, \studot \, \mu_- - \alpha \, \studot \, f \, \studot \, \mu_- + \alpha \, \studot \, ( f\circb \mu_-) + f + \alpha \, \studot \, \widetilde{f}  \ .$$
Multiplying on the left by $1-\mu_-$,  and using  the fact that $(1- \mu_-) \, \studot \, \alpha=1$,  this is
$$        \widetilde{f} \, \studot \, (1- \mu_-)    -   f \, \studot \, \mu_- +   f\circb \mu_- + (1- \mu_-) \, \studot \, f \ .$$
Thus, by induction on the depth,  theorem $ \ref{thmxicircalpha}$ is equivalent to  $(\ref{eqnfcircmuminus})$.
\end{proof}
Suppose that $f\in \Or$ is homogeneous of even weight, and satisfying the  shuffle and stuffle equations modulo products.  Then 
$$\widetilde{f}^{(d)}(x_1,\ldots, x_d) = (-1)^d (-1)^{d-1} f(x_d,\ldots, x_1)\ . $$
It follows from $(\ref{circformula})$ and $(\ref{muidef})$ that $(\ref{eqnfcircmuminus})$ is equivalent to the following six-term equation.

\begin{prop} \label{prop6term} Let $f\in \Or$ be homogeneous of even weight. Then for all $d\geq 2$, 
\begin{eqnarray}\label{sixtermequation}
&&f^{(d)}(x_1,\ldots, x_d ) -  f^{(d)}(x_d,\ldots, x_1)\quad  = \quad { f^{(d-1)}(x_2,\ldots, x_d) \over x_1} -{ f^{(d-1)}(x_{d-1},\ldots, x_1) \over x_d}  \nonumber    \\
&&\qquad \qquad - {f^{(d-1)}(x_2-x_1,\ldots, x_d- x_1)\over x_1} - (-1)^d{f^{(d-1)}(x_d-x_{d-1},\ldots, x_d- x_1)\over x_d}   
\end{eqnarray}
\end{prop}

\begin{rem} The previous equation is an explicit expression of the form
$$ f^{(d)} + (-1)^d  \overline{\tau}(f)^{(d)} =  \hbox{ lower depth } ,$$
where $\overline{\tau}$ was defined in \S\ref{sectDihedralsymm}. On the other hand, the stuffle antipode gives an explicit, but rather different, expression with a different sign of the form
$$ f^{(d)} +  \overline{\tau}(f)^{(d)} = \hbox{ lower depth }. $$
Combining the two gives a canonical way to lift solutions to the double shuffle equations from depth $d-1$ to $d$ in the case
when $f$ is homogeneous of even weight, and $d$ is odd.  This is to be compared  with proposition \ref{propparity}.
\end{rem}

\subsubsection{Proof of the six-term relation $(\ref{sixtermequation})$}
Let $f\in \Or$ be a solution to the double shuffle equations modulo products, which is homogeneous of even weight. 
We shall work in the dihedrally symmetric coordinates $y_0,\ldots, y_d$.
Define
$$I(y_0,\ldots, y_d) = f^{(d)} (y_0,y_1,\ldots, y_d) - f^{(d)}(y_0, y_d ,\ldots, y_1)\ .$$
Since $f$ has even weight, we have $f^{(d)}(y_0,\ldots, y_d) = f^{(d)}(y_d,\ldots, y_0)$, and thus
\begin{equation}\label{frotasI}
f^{(d)}(y_1,\ldots, y_d, y_0)  = f^{(d)}(y_0,y_1,\ldots, y_d) - I(y_0,y_1, \ldots, y_d)\ .
\end{equation}
Let us write the $(1,d-1)^{\mathrm{th}}$ stuffle equation for $f$ in the form
\begin{equation} \label{1d-1asA} f^{(d)}(y_0, y_1 \sha y_2  \ldots y_d) =  A(y_0,y_1 \sha y_2\ldots y_d)\  ,
\end{equation}
where $A$ is a sum of terms involving $f^{(d-1)}$. Change  variables $y_i \mapsto y_{i+1}$ to give
$$  f^{(d)}(y_1, y_2 \sha y_3  \ldots y_d y_0) =  A(y_1,y_2 \sha y_3\ldots y_dy_0)\ . 
$$
Now apply $(\ref{frotasI})$ to each term in the previous equation to give
\begin{multline} \label{sixtermpf1}
  f^{(d)}(y_0, y_1, y_2 \sha y_3  \ldots y_d) + f^{(d)}(y_2,y_1,y_3,\ldots, y_d, y_0  ) \\
  -  I(  y_0, y_1, y_2 \sha y_3  \ldots y_d) -   I(y_2,y_1,y_3,\ldots, y_d, y_0  )  =   A(y_1,y_2 \sha y_3\ldots y_dy_0)\ . 
\end{multline}
On the other hand, interchange the variables $y_1$ and $y_2$ in $(\ref{1d-1asA})$ to give:
\begin{equation} \label{sixtermpf2}
 f^{(d)}(y_0, y_2 \sha y_1 y_3  \ldots y_d) =  A(y_0,y_2 \sha y_1 y_3 \ldots y_d)
 \end{equation}
The difference $(\ref{sixtermpf2})$ $ - $  $(\ref{sixtermpf1})$ is
\begin{multline}
  f^{(d)}(y_0,y_2,y_1,y_3,\ldots, y_d ) - f^{(d)}(y_2, y_1,y_3  \ldots y_d,y_0) \\
  + I(  y_0, y_1, y_2 \sha y_3  \ldots y_d) +  I(y_2,y_1,y_3,\ldots, y_d, y_0  )   \\
  =       A(y_0,y_2 \sha y_1 y_3 \ldots y_d)
- A(y_1,y_2 \sha y_3\ldots y_dy_0)\ . 
\end{multline}
Applying equation $(\ref{frotasI})$  to the first two terms gives the equation
\begin{multline} \label{IandAonlyequation}
  I(y_0,y_2,y_1,y_3,\ldots, y_d)   +        I(  y_0, y_1, y_2 \sha y_3  \ldots y_d) +  I(y_2,y_1,y_3,\ldots, y_d, y_0  )\\
  =       A(y_0,y_2 \sha y_1 y_3 \ldots y_d)
- A(y_1,y_2 \sha y_3\ldots y_dy_0)\ ,
\end{multline}
which involves only $I$'s and $A$'s. It follows from the definition of $I$ and $(\ref{1d-1asA})$ that
$$ I( y_0, y_1 \sha y_2\ldots y_d) = A(y_0,y_1 \sha y_2\ldots y_d) - A(y_0,y_1 \sha y_d\ldots y_2)\ . $$
Replacing the middle term in $(\ref{IandAonlyequation})$ in this way, we obtain
\begin{equation}  
  I(y_2,y_1,y_3,\ldots, y_d, y_0  )
  =     A(y_0,y_2 \sha y_d \ldots  y_3 y_1) - A(y_1,y_2 \sha y_3\ldots y_dy_0)   
     \end{equation}
By changing variables, this gives the equation:
\begin{equation} \label{IasAfinal}
  I(y_0,y_1,y_2,\ldots, y_d  )
  =     A(y_d,y_0 \sha  y_{d-1} \ldots  y_2 y_1) - A(y_1,y_0 \sha y_2\ldots y_{d-1}y_d)   \ .
     \end{equation}
Now, by the definition of the stuffle product, we verify that
\begin{multline}\nonumber
A(y_0, y_1 \sha y_2\ldots y_d) =  { 1 \over y_2 -y_1} \Big(   f^{(d-1)}(y_0, y_1,y_3,   \ldots, y_d)-f^{(d-1)}(y_0,y_2,\ldots, y_d)\Big) \\
+\sum_{i=3}^d { 1 \over y_i -y_1} \Big( f^{(d-1)}(y_0,y_2,\ldots, y_d) - f^{(d-1)}(y_0, y_2, \ldots, y_{i-1}, y_1, y_{i+1}, \ldots, y_d)\Big) \ .
\end{multline}
Since $f$ is of even weight,  $f^{(d-1)}(z_1,\ldots, z_d) = f^{(d-1)}(z_d,\ldots, z_1)$, for all $z_i$. Therefore most  terms in the previous
equation will cancel when we take the difference
\begin{multline}\nonumber
A(y_0, y_1 \sha y_2\ldots y_d) -A(y_d, y_1 \sha y_{d-1}\ldots y_2y_0)    \\
 =  { 1 \over y_d-y_1} \Big( f^{(d-1)}(y_0, y_2,y_3,   \ldots, y_{d-1},y_1)-f^{(d-1)}(y_0,y_2,\ldots, y_d) \Big) \\
- { 1 \over y_{0} -y_1} \Big( f^{(d-1)}(y_d, y_{d-1},   \ldots, y_2,y_1)- f^{(d-1)}(y_d,y_{d-1},\ldots, y_2,y_0) \Big) \ .
\end{multline}
Substituting into $(\ref{IasAfinal})$ after an appropriate change of variables gives
\begin{multline} \nonumber
 I(y_0,y_1,y_2,\ldots, y_d  )
  = { 1 \over y_{1} -y_0} \Big(  f^{(d-1)}(y_d, y_{d-1}, \ldots, y_2,y_0  ) - f^{(d-1)}(y_d,y_{d-1},\ldots, y_1)\Big) \\
- { 1 \over y_{d} -y_0} \Big( f^{(d-1)}( y_1,   \ldots, y_{d-2},y_{d-1},y_0) -f^{(d-1)}(y_1,y_{2},\ldots, y_d) \Big) \ .
\end{multline} 
Reversing all arguments of $f^{(d-1)}$ in the right-hand side, and setting $y_0=0$, gives 
\begin{multline} \nonumber
 \overline{I}(x_1,x_2,\ldots, x_d  )
  = { 1 \over x_{1} } \Big( \overline{f}^{(d-1)}( x_2, \ldots, x_d )- \overline{f}^{(d-1)}(x_{2}-x_1,\ldots, x_d-x_1) \Big) \\
- { 1 \over x_{d}} \Big( \overline{f}^{(d-1)}( x_{d-1},   \ldots, x_1)- \overline{f}^{(d-1)}(x_{d-1}-x_d,\ldots, x_1-x_d)  \Big) \ , 
\end{multline} 
which is precisely equation $(\ref{sixtermequation})$, as required.

\subsection{Full residue structure for $\Lo$}
\begin{lem} Suppose that $f\in \Or$ satisfies the double shuffle equations modulo products and is homogeneous of even weight. If it satisfies the residue  equation $(\ref{GeneralResstructure})$, then more generally, for any 
indices $2 \leq i-j <d$, 
 we have
\begin{eqnarray} \label{multiresgeneral}
\Res_{x_i=x_j} f^{(d)}(x_1,\ldots, x_d)  &=  &  {1 \over (x_{j+1}-x_j) \ldots (x_{i-1}-x_j) } \times \nonumber \\
& & \Res_{x_{j+1}=x_j} f^{(d-i+j+1)}(x_1,\ldots, x_j, x_{j+1}, x_{i+1},\ldots, x_d) 
\end{eqnarray} 
 This expression in the case $j=0$ reduces to  $(\ref{GeneralResstructure})$ on setting $x_0=0$ as usual.
\end{lem}

\begin{proof} The proof is by induction on $j$.  First of all, for all $d\geq 1$, define 
$$A^{(d)}(x_1,\ldots, x_d) = f^{(d)}(x_1,\ldots, x_d) - { 1\over x_1} f^{(d-1)}(x_2,\ldots, x_d) + { 1 \over x_1} f^{(d-1)}(x_2-x_1,\ldots, x_d-x_1)$$
We have the following  two  symmetries:
\begin{eqnarray}
f^{(d)}(x_1,\ldots, x_d)  & =  & (-1)^{d+1} f^{(d)}(x_d-x_{d-1}, \ldots, x_d-x_1, x_d) \label{fsymRespf} \\
A^{(d)}(x_1,\ldots, x_d)  & =  & - A^{(d)}(x_d,\ldots, x_1) \ .\label{AsymRespf}
\end{eqnarray} 
The first equation is the usual antipodal symmetry for the shuffle equations, the second is simply a restatement of the six-term relation (proposition \ref{prop6term}).
Now write $P(g^{(d)}, j, i)$ to denote the statement that $(\ref{multiresgeneral})$ holds for a function $g$ satisfying the conditions of the lemma.  Then $(\ref{fsymRespf})$  and $(\ref{AsymRespf})$ yield 
$$P(f^{(d)}, j, i )   \quad  \Longleftrightarrow \quad P(f^{(d)}, d-i,  d-j) $$ 
and, for all indices $2 \leq i -j <d$ with the extra condition $j>0$: 
$$P(A^{(d)}, j, i )   \quad  \Longleftrightarrow \quad P(A^{(d)}, d-i+1,  d-j+1)\ . $$ 
 Finally, using the definition of $A$, one checks  that  if
 $P(f^{(d)}, i, j)$ holds, then 
 $$  P(f^{(d+1)}, i+1, j+1) \qquad  \Longleftrightarrow \qquad   P(A^{(d+1)}, i+1, j+1)$$
 The proof then proceeds by induction using the three previous equivalences.
  We have $P(f^{(d)}, 0, *) $ for all $d$ by assumption  $(\ref{GeneralResstructure})$, where $*$ denotes any choice of index satisfying the conditions of the lemma. Then for all $d$, 
  $$P(f^{(d)}, \leq k , *) \Rightarrow P(f^{(d)}, * , \geq d-k )   \Rightarrow   P(A^{(d)}, * , \geq d-k )  \Rightarrow   P(A^{(d)}, \leq k+1 , *  )  $$
This  implies  $  P(f^{(d)}, \leq k+1 , *  ) $, which completes the induction step. 
 \end{proof}

\bibliographystyle{plain}
\bibliography{main}

\end{document}